\documentclass[12pt,centertags]{amsart}
\usepackage{amsmath,amstext,amsthm,a4,amssymb,amscd}
\usepackage[mathscr]{eucal}
\usepackage{mathrsfs}
\usepackage{epsf}
\numberwithin{equation}{section}

\newcommand{\field}[1]{\mathbb{#1}}
\newcommand{\Z}{\field{Z}}
\newcommand{\R}{\field{R}}
\newcommand{\C}{\field{C}}

 \def\cC{\mathscr{C}}

\def\mM{\mathcal{M}}
\def\mO{\mathcal{O}}

\def\mU{\mathcal{U}}

\def\kr{\mathfrak{r}}
\def\kt{\mathfrak{t}}
\def\kg{\mathfrak{g}}

\DeclareMathOperator{\End}{End}
\DeclareMathOperator{\Hom}{Hom}
\DeclareMathOperator{\Ker}{Ker}

\DeclareMathOperator{\Dom}{Dom}

\DeclareMathOperator{\Id}{Id}
\DeclareMathOperator{\supp}{supp}
\DeclareMathOperator{\tr}{Tr}
\DeclareMathOperator{\Ind}{Ind}

\newcommand{\om}{\omega}

\newtheorem{thm}{Theorem}[section]
\newtheorem{lemma}[thm]{Lemma}
\newtheorem{prop}[thm]{Proposition}
\newtheorem{cor}[thm]{Corollary}
\theoremstyle{definition}
\newtheorem{rem}[thm]{Remark}
\theoremstyle{definition}
\newtheorem{defn}[thm]{Definition}
\newcommand{\be}{\begin{eqnarray}}
\newcommand{\ee}{\end{eqnarray}}
\newcommand{\ov}{\overline}
\newcommand{\wi}{\widetilde}
\newcommand{\var}{\varepsilon}

\newcommand{\comment}[1]{}

\begin{document}



\title{Geometric quantization for proper moment maps: the Vergne conjecture}

\author{Xiaonan Ma}

\address{Universit{\'e} Paris Diderot - Paris 7,
UFR de Math{\'e}matiques, Case 7012,
Site Chevaleret,
75205 Paris Cedex 13, France}
\email{ma@math.jussieu.fr}

\author{Weiping Zhang}

\address{Chern Institute of Mathematics \& LPMC, Nankai
University, Tianjin 300071, P.R. China}
\email{weiping@nankai.edu.cn}

\begin{abstract} We establish a geometric quantization
formula for a Hamiltonian action of a compact Lie group acting on a
noncompact symplectic manifold
with proper  moment map.
\end{abstract}

\maketitle
\tableofcontents

\setcounter{section}{-1}

\section{Introduction} \label{s0}

The purpose of this paper is to  establish a geometric quantization
formula for a Hamiltonian action of a compact Lie group acting on a
noncompact symplectic manifold with proper  moment map.
Our results provide a solution to a conjecture of
Mich\`ele Vergne in her ICM 2006 plenary lecture \cite{Ve07}.

Let $(M,\omega)$ be a  symplectic manifold with symplectic form
$\omega$, and $\dim M=n$. 
We assume that $(M,\omega)$ is prequantizable, that is,
there exists a complex line bundle $L$ (called a prequantum line
bundle) carrying a Hermitian metric $h^L$ and a Hermitian connection
$\nabla^L$ such that the associated curvature
 $R^L=\left(\nabla^L\right)^2$ verifies
\begin{align}\label{0.1}
\frac{\sqrt{-1}}{2\pi} R^L=\omega.
\end{align}

 Let $TM$ be the tangent vector bundle of $M$.
Let  $J^M$ be an almost complex structure  on
$TM$ such that
\begin{align}\label{0.2}
g^{TM}(u,v)=\omega(u,J^Mv),\ \ \ u,\ v\in TM,
\end{align}
defines a $J^M$-invariant Riemannian metric $g^{TM}$ on $TM$.

Let $G$ be a compact connected Lie group.
Let  $\kg $ denote the Lie algebra of $G$ and
$\kg^*$ denote the dual of $\kg$. Let $G$ act  on
$\kg^*$ by the coadjoint action.

We assume that  $G$ acts on the left on $M$, that this action
lifts to an action on $L$, and that $G$ preserves
$g^{TM}$, $J^M$, $h^L$ and $\nabla^L$.

For $K\in \kg$, let $K^M\in\cC^\infty(M,TM)$ denote the vector field
generated by $K$ over $M$.
The moment map  $\mu : M \to \kg^*$ is defined by the Kostant formula
 \cite{Ko70},
 \begin{align}\label{0.3}
 2 \pi \sqrt{-1}\mu(K) : 
 =\nabla ^L_{K^M}- L_K, \, \, K\in \kg.
 \end{align}
  Then, for any $K\in \kg$, we have
 \begin{align}\label{0.4}
 d\mu(K)= i_{K^M}\om.
 \end{align}

 From now on, we make the following  assumption.

 \vspace*{0.3cm}
 \noindent{\bf Fundamental Assumption}. The moment map $\mu:M\rightarrow
 {\kg}^* $ is proper,
 i.e.,  for any compact   subset $B\subset \kg^*$,
 the subset $\mu^{-1}(B)\subset M$ is compact.
 \vspace*{0.3cm}

Let $T$ be a maximal torus of $G$,
  let $\kt$ be its Lie algebra and $\kt^*$ the dual of $\kt$.
The integral lattice $\Lambda\subset \kt$ is defined as the kernel
of the exponential map $\exp: \kt\to T$, and the real weight lattice
$\Lambda^*\subset \kt^*$ is defined by $\Lambda^*:= \Hom(\Lambda,
2\pi \Z)$. We fix a positive Weyl chamber $ \kt^*_+ \subset \kt^*$.
Then the set of finite dimensional $G$-irreducible
representations is parametrized by 
$\Lambda^*_+ := \Lambda^*\cap \kt^*_+$.

Recall that $\kg= \kt\oplus \kr$, with $\kr=[\kt,\kg]$, and so
$\kg^*= \kt^*\oplus \kr^*$. So we identify 
$\Lambda^*_+$ to a subset of $\kg^*$.
For $\gamma\in \Lambda^*_+ $, we denote by $V_\gamma^G$ the
irreducible $G$-representation with highest weight $\gamma$. The
$V_\gamma^G$, $\gamma\in \Lambda^*_+ $, form a $\Z$-basis of the
representation ring $R(G)$.
Let $R[G]$ be the formal representation ring of $G$.
 For $W\in R[G]$, we denote by $W_\gamma\in \Z$ the multiplicity of
 $V^G_\gamma$ in $W$.

Take $\gamma\in \Lambda^*_+ $. If $\gamma$ is a regular value of
the moment map $\mu$, then one can construct the Marsden-Weinstein
symplectic reduction $(M_\gamma,\omega_\gamma)$, with 
$M_\gamma=G\backslash \mu^{-1}(G\cdot\gamma)$  a {\it compact} orbifold
(since $\mu$ is proper). Moreover, the line bundle $L$ (resp. the almost
complex structure $J$) induces a prequantum line bundle $L_\gamma$
(resp. an almost complex structure $J_\gamma$) over
$(M_\gamma,\omega_\gamma)$. One can then construct the associated
Spin$^c$-Dirac operator (twisted by $L_\gamma$), $D^{L_\gamma}_+:
\Omega^{0,\text{even}}\left(M_\gamma,L_\gamma\right) \to
\Omega^{0,\text{odd}}(M_\gamma,L_\gamma)$ 
(cf. \eqref{1.4}, Section \ref{s1.6})
 on $M_\gamma$, of which the index is defined by
\begin{align}\label{b0.4}
Q\left(L_\gamma\right)  = \Ind\left(D^{L_\gamma}_+\right)
:= \dim\Ker \left(D^{L_\gamma}_+\right)- \dim
\text{Coker}\left(D^{L_\gamma}_+\right) \in \Z.
\end{align}
If $\gamma\in \Lambda^*_+ $ is not a regular value
of $\mu$, then by a perturbation argument
(cf. \cite{MS99}, \cite[\S 7.4]{Par01}),
 one still gets a well-defined quantization  number
$Q(L_\gamma)$ extending the above definition.

We equip $\kg$ with an ${\rm Ad}_G$-invariant scalar product.
We will identify  $\kg$ and $\kg^*$ by this scalar product. 
Let $\pi:TM\rightarrow M$ denote the  projection from   $TM$ to $M$.
We identify $T^*M$ with $TM$ by the scalar product $g^{TM}$.

 Set ${\mathcal H}=|\mu|^2$. Let $X^{\mathcal H}=-J^M(d{\mathcal H})$
be the Hamiltonian vector field associated with ${\mathcal H}$. Then
(see (\ref{1.22}))
\begin{align}\label{bb0.4}
 X^{\mathcal H}= 2\,  \mu^M,
\end{align}
where $\mu^M\in \cC^{\infty}(M,TM)$ is
the vector field on $M$ generated by $\mu:M\rightarrow \kg$, i.e.,
for any $x\in M$, $\mu^M(x)=(\mu(x))^M(x)$.


For $a\geqslant 0$, set $M_a:={\mathcal H}^{-1}([0,a])=\{x\in M:
{\mathcal H}(x)\leqslant a\}$. 
For any regular value $a>0$ of ${\mathcal H}$, by (\ref{bb0.4}),
  $\mu^M$ does not vanish on $\partial M_a={\mathcal H}^{-1}(a)$,
the boundary of the compact $G$-manifold $M_a$.
According to Atiyah \cite[\S 1, \S 3]{A74}  and Paradan
\cite[\S 3]{Par01} (cf. also Vergne \cite{Ve96}), the triple $(M_a,
\mu^M, L)$ defines a transversally elliptic symbol
$$\sigma^{M_a}_{L,\mu}:=\pi^*\left(\sqrt{-1} c\left(\cdot+
{\mu^M}\right)\otimes \Id_L\right): \pi^*\left(\Lambda
(T^{*(0,1)}M_a)\otimes L\right)\longrightarrow \pi^*\left(\Lambda
(T^{*(0,1)}M_a)\otimes L\right) ,$$ where $c(\cdot)$ is the Clifford
action on $\Lambda (T^{*(0,1)}M)$
(cf. \eqref{a1.20}).\footnote{The symbol $\sigma_{L,\mu}^{M_{a}}$
is the (semi-classical) symbol of Tian-Zhang's \cite{TZ98}, \cite{TZ99}
deformed Dirac operator \eqref{1.26}
in their approach to the
Guillemin-Sternberg geometric quantization conjecture \cite{GuSt82}.
The associated symbol was used by Paradan \cite{Par01}, \cite{Par03}
in his approach to the same conjecture.}
 Let $\Ind(\sigma^{M_a}_{L, \mu})\in R[G]$ denote the corresponding
 transversal index in the sense of  Atiyah \cite[\S 1]{A74}.


\begin{thm}\label{t0.1} a) For $\gamma\in\Lambda^*_+ $, there exists
$a_\gamma\geqslant 0$ \footnote{In view of Theorem  \ref{t1.5}, we can take
$a_\gamma= \frac{c_\gamma}{4\pi^2}$ with $c_\gamma$
being defined in (\ref{1.23x}).}
such that $\Ind(\sigma^{M_a}_{L,\mu})_\gamma\in \Z$ 
does not depend on
the regular value $a>  a_\gamma$ of ${\mathcal H}$.

b) $\Ind(\sigma^{M_a}_{L,\mu})_{\gamma=0}\in \Z $ does
not depend on the regular value $a> 0$ of ${\mathcal H}$.
\end{thm}

By Theorem \ref{t0.1}, for $\gamma\in\Lambda^*_+ $, 
 we can associate an integer $Q(L)_\gamma $ that is equal to 
$\Ind(\sigma^{M_a}_{L,\mu})_\gamma$ for
 large enough regular value $a> 0$ of ${\mathcal H}$. 

We can now state the  main result of this paper.
\begin{thm}\label{t0.2} For  $\gamma\in\Lambda^*_+ $, the
following identity holds:
\begin{align}\label{0.5}
Q(L)_\gamma =Q(L_\gamma).
\end{align}
\end{thm}

\begin{rem}\label{t0.4a}
When $M$ is compact,
 Theorem \ref{t0.2} is the Guillemin-Sternberg geometric quantization
conjecture \cite{GuSt82} which was first proved 
by Meinrenken \cite{Mein95} and Vergne \cite{Ve96} in the case where
 $G$ is abelian, and by Meinrenken \cite{Mein98}
and Meinrenken-Sjamaar   \cite{MS99} in the general case. We refer
to \cite{Ve02} for a survey on the Guillemin-Sternberg geometric
quantization conjecture.

If $M$ is noncompact but the zero set of $X^{\mathcal H}$ is
compact, then Theorem \ref{t0.1} is already contained in \cite{Par03}
and \cite{Ve07}, while Theorem \ref{t0.2} was conjectured  by
  Mich{\`e}le Vergne in her ICM 2006 plenary lecture \cite[\S 4.3]{Ve07}.
Special cases of this conjecture, related to the discrete series
of semi-simple Lie groups, have been proved by
Paradan \cite{Par03}, \cite{Par08}.

Theorem \ref{t0.2} provides a solution to Mich{\`e}le Vergne's
conjecture even when the zero set of $X^{\mathcal H}$ is noncompact.
\end{rem}

 Theorem \ref{t0.2} is a consequence of a more general
 result that we will  now describe. 

  Let $(N,\omega^N, J^N)$ be a compact symplectic manifold with
compatible almost complex structure $J^N$.  
Let $(F,h^F,\nabla^F)$ be the prequantum line bundle over $N$ carrying
a Hermitian metric $h^F$ and a Hermitian connection $\nabla^F$
verifying $\frac{\sqrt{-1}}{ 2\pi}(\nabla^F)^2=\omega^N$. We
assume  that $G$ acts on $N$, $F$ as above. Let $\eta: N\to \kg^*$ be
the associated moment map.

Let $D_+^F: \Omega^{0,\text{even}}(N,F) \to
\Omega^{0,\text{odd}}(N,F)$ be the associated Spin$^c$ Dirac operator on $N$.  
Then as a virtual representation of $G$, we have
\begin{align}\label{b0.5}
\Ind\left(\sigma^N_{F,   \eta}\right)  
 = \Ind\left(D^F_+\right):=  \Ker \left(D^F_+\right)-
\text{Coker}\left(D^F_+\right) \in R(G).
\end{align}
For $\gamma\in \Lambda^*_+$, let $Q\left(F\right)_{\gamma,*}$ be 
 the multiplicity of the $G$-irreducible representation 
$(V^G_\gamma)^*$ in $\Ind\left(D^F_+\right) \in R(G)$.

 Let $L {\otimes} F$ be the prequantum line bundle over $M\times N$
obtained by the tensor product of the natural lifts of $L$ and
$F$ to $M\times N$.

\begin{thm}\label{t0.4} For the induced action of $G$ on
$(M\times N,\omega\oplus\omega^N)$ and  $L\otimes F$,
the following identity holds:
\begin{align}\label{0.7}
Q\left(\left(L {\otimes} F\right)_{\gamma=0}\right)
=\sum_{\gamma\in\Lambda^*_+ } Q(L )_\gamma\cdot
Q\left(F\right)_{\gamma,*}.
\end{align}
\end{thm}

For $\gamma\in \Lambda^*_+$,  denote by
$\mO_\gamma= G\cdot \gamma$ the orbit of the coadjoint action
of $G$ on $\kg^*$. 
Let $L^\gamma$ be the canonical prequantum holomorphic line bundle 
on $\mO_\gamma$, such that the associated moment map is the inclusion
$\mO_\gamma\hookrightarrow \kg^*$.
 By the Borel-Weil-Bott theorem 
and the solution of the  Guillemin-Sternberg
geometric quantization conjecture 
for the compact manifold $\mO_{\nu_1}\times\mO_{\nu_2}$, one has
that $\Hom_G(V^G_{\nu_3}, V^G_{\nu_1}\otimes  V^G_{\nu_2})\neq 0$
if and only if $\nu_3 \in G\cdot \nu_1+ G\cdot \nu_2$.  In
particular, one has  $|\nu_1|\leqslant |\nu_3|+|\nu_2|$.
For $ \nu_1,\nu_2\in \Lambda^*_+$, set
\begin{align}\label{b0.6}\begin{split}
C^\gamma_{\nu_1,\nu_2}&= \dim \Hom_G(V^G_\gamma,
V^G_{\nu_1}\otimes  V^G_{\nu_2}).
\end{split}\end{align}

By taking $N, F$ to be $\mO_\gamma, (L^\gamma)^*$,
we recover Theorem \ref{t0.2} from Theorem \ref{t0.4}
by using the Borel-Weil-Bott theorem.

By applying  Theorems \ref{t0.2}, \ref{t0.4} to $M\times N\times
\mO_\gamma$, we get the following result which is trivial in the compact case.
\begin{cor}\label{t0.5a}
For any $\gamma\in \Lambda^*_+$, the following identity holds:
\begin{align}\label{b0.8}
Q\left(L {\otimes} F \right)_{\gamma}
=\sum_{\nu_1,\nu_2\in\Lambda^*_+ } C^\gamma_{\nu_1,\nu_2}
 Q\left(L \right)_{\nu_1}\cdot Q\left(F\right)_{\nu_2} ,
\end{align}
where there are only finitely many non-vanishing terms in the right-hand side.
\end{cor}

We now explain briefly the main ideas of the proof of Theorems
\ref{t0.1} and \ref{t0.4}.

The first observation is that in the case when $\gamma=0$, 
 both Theorems
\ref{t0.1} and \ref{t0.2} are relatively easy to prove. 
On the other hand, in the case when $\gamma\neq 0$,
one needs to establish the more
general Theorem  \ref{t0.4}, in order to prove (\ref{0.5}).

In fact, it is relatively  easy to see  that (cf. (\ref{2.8}) and (\ref{44.1}))
\begin{align}\label{0.13}
 Q \left(L {\otimes} F\right)_{\gamma=0}
=Q\left(\left(L {\otimes} F\right)_{\gamma=0}\right).
\end{align}
Thus Theorem \ref{t0.4} is 
a consequence of \eqref{0.13} and the following identity,
\begin{align}\label{0.14}
Q \left(L {\otimes} F\right)_{\gamma=0} =\sum_{\gamma\in\Lambda^*_+} 
Q(L)_\gamma\cdot Q\left(F\right)_{\gamma,*}.
\end{align}

Assume that $M$ is compact. Then \eqref{0.14} is trivial and this is
why one only needs to prove \eqref{0.5} for $\gamma=0$, in order to
establish \eqref{0.5}.  

However, if $M$ is noncompact, 
although  the geometric data on $M\times N$
have product structure, and the associated moment map
is $\theta(x,y)=\mu(x)+\eta(y)$, the vector field $\theta^{M\times
  N}$ on $M\times N$ induced by $\theta$ is not a sum of two  vector
fields lifted from  $M$ and $N$ (cf. \eqref{2.25}). Thus one cannot
compute directly $Q \left(L {\otimes} F\right)_{\gamma=0}$ as the
right hand side of  \eqref{0.14}.

To be more precise, let $a>0$ be a regular value of ${\mathcal H}$
so that $\mu^M$ does not vanish on $\partial M_a$.  
By the multiplicativity of the transversal index,
\begin{align}\label{0.101}
 \sum_{\gamma\in\Lambda^*_+ } 
\Ind(\sigma_{L,\mu}^{M_a})_{\gamma}\cdot Q\left(F\right)_{\gamma,*} 
= \Ind\left(\sigma_{L\otimes F,\mu}^{M_a\times N}\right)_{\gamma=0}.
\end{align}
Let $b>0$  be a regular value of ${\mathcal
H}'=|\theta|^2$. Then $\theta^{M\times N}\in T(M\times N) $ 
does not vanish on the boundary $\partial (M\times N)_b$ of 
$(M\times N)_b=\{(x,y)\in M\times N,\,|\theta(x,y)|^2\leqslant b\}$.
By Theorem \ref{t0.1}b), we have
\begin{align}\label{0.102}
 Q \left(L {\otimes} F\right)_{\gamma=0}
=\Ind\left(\sigma_{L\otimes
F,\theta}^{(M\times N)_b}\right)_{\gamma=0}.
\end{align}

We take $b>0$ large enough so that
$M_a\times N\subset (M\times N)_b$ and that
$(\partial(M\times N)_b)\cap( \partial
(M_a\times N))=\emptyset$. Denote by ${\mathcal M}_{a,b}$ the
closure of $(M\times N)_b\setminus M_a\times N$. Then ${\mathcal
M}_{a,b}$ is a manifold with boundary $\partial {\mathcal
M}_{a,b}=(\partial(M\times N)_b)\cup (\partial (M_a\times N))$.

Let $\Psi_{a,b}:{\mathcal M}_{a,b}\rightarrow \kg$ be a
$G$-equivariant map such that $\Psi_{a,b}|_{\partial (M_a\times
N)}=\mu $, while $\Psi_{a,b}|_{\partial (M\times N)_b}=\theta $. 
From the additivity of the transversal index, we get
\begin{align}\label{0.17}
  \Ind\left(\sigma_{L\otimes
F,\Psi_{a,b}}^{{\mathcal M}_{a,b}}\right)_{\gamma=0}=
\Ind\left(\sigma_{L\otimes F,\theta}^{(M\times N)_b}\right)_{\gamma=0}
- \Ind\left(\sigma_{L\otimes F,\mu}^{M_a\times N}\right)_{\gamma=0}.
\end{align}
We infer from (\ref{0.14})-(\ref{0.17}) that Theorem \ref{t0.4}
is equivalent to
\begin{align}\label{0.103}
  \Ind\left(\sigma_{L\otimes
F,\Psi_{a,b}}^{{\mathcal M}_{a,b}}\right)_{\gamma=0}=0.
\end{align}

Let $a_1>0$ be another large enough regular value of ${\mathcal H}$.
By the additivity and the homotopy invariance of the transversal index,
we have,
\begin{multline}\label{0.104}
  \Ind\left(\sigma_{L\otimes
F,\Psi_{a,b}}^{{\mathcal M}_{a,b}}\right)_{\gamma=0}
- \Ind\left(\sigma_{L\otimes F,\Psi_{a_1,b}}^{{\mathcal
M}_{a_1,b}}\right)_{\gamma=0}\\
=  \Ind\left(\sigma_{L\otimes
F,\mu}^{{\mathcal M}_{a_{1}}\times N}\right)_{\gamma=0}
- \Ind\left(\sigma_{L\otimes F,\mu}^{{\mathcal
M}_{a}\times N}\right)_{\gamma=0}.
\end{multline}
By (\ref{0.101}), (\ref{0.104}), and by taking $N, F$ to be
$\mO_\gamma, (L^\gamma)^*$ for $\gamma\in\Lambda_+^*$, we find
that Theorem \ref{t0.1}a) is a consequence of the vanishing result
(\ref{0.103}).

 Note that in the situations considered in \cite{Par03},
 \cite{Par08}, for 
$a, b>0$ large enough, one is able to find
 $\Psi_{a,b}:{\mathcal M}_{a,b}\rightarrow \kg$ such that 
$\Psi_{a,b}^{{\mathcal M}_{a,b}}\in T{\mathcal M}_{a,b}$ 
does not vanish on ${\mathcal M}_{a,b}$.
 From this, (\ref{0.103}) follows tautologically.
However, there is no canonical way to construct $\Psi_{a,b}$ such that 
$\Psi_{a,b}^{{\mathcal M}_{a,b}}\in T{\mathcal M}_{a,b}$ 
does not vanish on ${\mathcal M}_{a,b}$ 
in the general situation considered here.

 Our proof of (\ref{0.103}) consists of two steps. In a first step, we
express the transversal index
as an Atiyah-Patodi-Singer (APS) type index on corresponding  manifolds with
 boundary. Then in a second step, we construct a
 specific deformation  map $\Psi_{a,b}$, when $a,\, b>0$ are large enough,
so that we can apply the analytic localization techniques
developed in \cite{BL91}, \cite{TZ98} and \cite{TZ99} to the current problem. 
This allows us to show that the 
APS type index corresponding to the left-hand
 side of (\ref{0.103}) vanishes\footnote{In fact, the corresponding vanishing
result for the APS index, in the case of $N={\rm point}$ and
$\eta=0$, has already been proved in \cite[Theorems 2.6,
4.3]{TZ99}}. 

This paper is organized as follows. In Section \ref{s1}, 
we express the transversal index as an 
APS type index.
 In Section \ref{s1.6}, we establish
Theorem \ref{t0.1}, 
by applying the identification  of the transversal
index to an APS index that was established in Section \ref{s1}, 
 as well as the analytic localization
techniques developed in
\cite{BL91}, \cite{TZ98} and \cite{TZ99}. In Section \ref{s33}, we
present our proof of (\ref{0.103}). Finally, in Section \ref{s4}, we
provide details of the proofs  of (\ref{0.13}) and (\ref{0.101}),
thus completing the proof of Theorem \ref{t0.4}.
We explain also the compatibility of  quantization and its
restriction to a subgroup.

The results contained in this paper have been announced 
in \cite{MZ09a} (cf. also \cite[\S 4]{Ma10}).


\subsection{Notation}

In the whole paper, $G$ is a compact connected Lie group 
with Lie algebra $\kg$.
Let ${\rm Ad}_G(g)$ denote the adjoint action of $g\in G$ on $\kg$.
We equip $\kg$ with an ${\rm Ad}_G$-invariant scalar product,
and we identify $\kg$ and $\kg^*$ by this scalar product.
Let $V_1,\cdots, V_{\dim G}$ be an orthonormal basis of $\kg$.

 If a Hilbert space $H$ is a $G$-unitary
representation space,  by the Peter-Weyl theorem, one has the
orthogonal decomposition of Hilbert spaces
\begin{align}\label{a1.111}
H  =\bigoplus_{\gamma\in\Lambda_+^*} H^\gamma, \quad \text{ with  }
H^\gamma = \Hom_G (V^G_\gamma, H)\otimes V^G_\gamma.
\end{align}
We will call $H^\gamma$ the $\gamma$-component of $H$.
Moreover, if $W\subset H$ is a $G$-invariant linear subspace, for
 $\gamma\in\Lambda_+^*$, we denote by
\begin{align}\label{a1.1111}
W^\gamma  =W\cap H^\gamma
\end{align} 
and call it the $\gamma$-component of $W$.
If $D: \Dom(D)\subset H\to H$ is a $G$-equivariant linear operator, 
where $\Dom(D)$ is a dense $G$-invariant subspace of $H$, 
for $\gamma\in \Lambda^*_+$, we denote by 
$D(\gamma)$ the restriction of $D$ to $\Dom(D)^\gamma$
 which is dense in $H^\gamma$.

If $G$ acts on the left on a manifold ${\bf M}$, 
for $K\in \kg$,  we denote by
 $K^{\bf M}(x)= \frac{\partial}{\partial t}e ^{tK}x|_{t=0}$\index{$K^{\bf M}$}
 the corresponding vector field on ${\bf M}$.
 
For any $\Phi\in \cC^\infty({\bf M},\kg)$, we denote 
 $\Phi_i$, $1\leqslant i\leqslant \dim G$, 
the smooth functions on ${\bf M}$ defined by
\begin{align}\label{1.11}
 \Phi(x)=\sum_{i=1}^{\dim G}\Phi_i(x) V_i\  \ \ {\rm for}\ \  x\in {\bf M}.
\end{align}
Let $\Phi^{\bf M}$ denote the
vector field over ${\bf M}$ such that for any $x\in {\bf M}$,
\begin{align}\label{1.2z}
\Phi^{\bf M}(x)=(\Phi(x))^{\bf M}(x) = \sum_{i=1}^{\dim G}\Phi_i(x) 
V_i^{\bf M}(x),
\end{align}
where $(\Phi(x))^{\bf M}$ is the vector field
over ${\bf M}$ generated by $\Phi(x)\in\kg$.

Finally, when a subscript index  appears two times in a formula,  we
sum up with this index unless other notification is given.

$\ $

\noindent{\bf Acknowledgments}. We would like to thank Professor
Jean-Michel Bismut for many helpful discussions, as well as for kindly 
helping us to revise an ealier version of our manuscript.
X. M. thanks Institut Universitaire de France for support. The
work of  W. Z. was partially supported by MOEC and NNSFC. Part of
the paper was written while W. Z.  was visiting the School of
Mathematics of Fudan University during November and December of
2008. He would like to thank Professor Jiaxing Hong and other
members of the School for hospitality. 
We are also indebted to George Marinescu for his critical comments.
Last but not least, we would like to thank the referees of this paper 
for their critical reading and very helpful comments and suggestions.

\section{Transversal index and APS index}\label{s1}

In this section we 
express the transversal index as an 
Atiyah-Patodi-Singer\footnote{ In the sequel, 
Atiyah-Patodi-Singer will be abbreviated to APS.
} type index which have been studied in \cite{TZ99}
for $\gamma=0$ component.


This section is organized as follows. In Section \ref{s1.1}, we
recall the definition of the transversal index in the sense of
Atiyah \cite{A74} 
for manifolds with boundary. In Section \ref{s1.2}, we consider 
instead an index problem on a manifold with boundary 
for a Dirac operator with APS boundary conditions.
In Section \ref{s1.3}, we prove
the corresponding Dirac operator on the boundary is invertible.
 This guarantees that the APS index of the Dirac operator 
is invariant under deformation.
In Section \ref{s1.5}, we show that the transversal index can be 
identified with the APS index using a result by Braverman \cite{Brav02}.


We use the same notation as in the Introduction.

\subsection{Transversal index }\label{s1.1}

Let $M$ be an even dimensional  compact oriented Spin$^c$-manifold
with non-empty boundary $\partial M$, and $\dim M=n$.  
In the following, the boundary $\partial M$ carries the induced 
orientation. Let $g^{TM}$ be a Riemannian metric 
on the tangent vector bundle $\pi:TM\rightarrow M$.
Let $E$ be a complex vector bundle over $M$.



 We assume that the compact connected Lie group $G$ acts  
isometrically  on the left on $M$,  
and that this action lifts to an action of $G$ on
the Spin$^{c}$-principal bundle of $TM$ and on $E$.
Then the $G$-action also preserves $\partial M$.


We identify  $TM$ and $T^*M$ 
by the $G$-invariant metric $g^{TM}$.
Following \cite[p. 7]{A74} (cf. \cite[\S 3]{Par01}), set
\begin{align}\label{1.1}
T_G M=\left\{(x,v)\in T_xM: x\in M \text{ and }
\left\langle v, K^M(x)\right\rangle=0\,  {\rm for\ all}  \  K\in\kg\right\}.
\end{align}
Let $S(TM)=S_+(TM)\oplus S_-(TM)$ be the vector bundle of spinors
associated with the 
spin$^c$-structure on $TM$ and $g^{TM}$. 
For any $V\in TM$, the Clifford action $c(V)$ exchanges $S_\pm(TM)$.

 Let $\Psi:
M\rightarrow \kg$ be a $G$-equivariant smooth map. 
Assume that   $\Psi^M$ does not vanish on $\partial M$,
i.e., for any $x\in \partial M$, $\Psi^{M}(x)\neq 0$.

Let $\sigma_{E, \Psi}^M\in {\rm Hom}(\pi^*(S_+(TM)\otimes
E),\pi^*(S_-(TM)\otimes E))$ be the symbol 
\begin{align}\label{1.2}
\sigma_{E, \Psi}^M (x, v)=\left. \pi^*\left(\sqrt{-1}c(v + \Psi^M)
\otimes{\rm Id}_E \right)\right|_{(x,v)} \quad \text{ for } x \in
M,\  v\in T_{x}M.
\end{align}

Since $\Psi^M$ does not vanish on $\partial M$, the set
$\{(x,v)\in T_GM: \sigma_{E, \Psi}^M(x,v)\ \mbox{is
non-invertible}\}$ is a compact subset of
$T_G{\widehat M}$ (where $\widehat{M}=M\setminus \partial M$ is
the interior of $M$),
so that $\sigma_{E, \Psi}^M$ is a
$G$-transversally elliptic symbol on $T_G{\widehat M}$ in the
sense of Atiyah \cite[\S 1, \S 3]{A74} and Paradan \cite[\S
3]{Par01}, \cite[\S 3]{Par03}.
The associated transversal index can be written in the form
\begin{align}\label{1.3}
\Ind\left(\sigma_{E, \Psi}^M\right)
=\bigoplus_{\gamma\in \Lambda^*_+}\,
\Ind \left(\sigma_{E, \Psi}^M\right)_{\gamma} \cdot V^G_\gamma\in R[G],
\end{align}
with each $\Ind \left(\sigma_{E, \Psi}^M\right)_{\gamma}\in \Z$. Moreover, 
$\Ind \left(\sigma_{E, \Psi}^M\right)$ only depends on 
the homotopy class of $\Psi$  
as long as $\Psi^M$ does not vanish on $\partial M$,
but not on $g^{TM}$. 
Note that the number of $\gamma\in
\Lambda^*_+$ such that $\Ind(\sigma_{E, \Psi}^M)_\gamma\neq 0$ could
be infinite.

\comment{Note that the number of $\gamma\in
\Lambda^*_+$ such that $\Ind_\gamma(\sigma_{E, \Psi}^M)\neq 0$ could
be infinite.}

\subsection{The Atiyah-Patodi-Singer (APS) index}\label{s1.2}

 We make the same assumptions and use the same notation as
 in Section \ref{s1.1}.

Let $h^E$ be a $G$-invariant Hermitian metric on $E$, $\nabla^E$ a
$G$-invariant Hermitian connection on $E$ with respect to $h^E$. Let
$h^{S(TM)}$ be the $G$-invariant Hermitian metric on $S(TM)$ induced
by $g^{TM}$ and a $G$-invariant metric  on the line bundle defining
the spin$^c$ structure (cf. \cite[Appendix D]{LaMi89}).
Let $h^{S(TM)\otimes E}$ be the metric on
$S(TM)\otimes E$ induced by the metrics on $S(TM)$ and   $E$.

Let $\nabla^{S(TM)}$ be the Clifford connection on $S(TM)$
 induced by the Levi-Civita connection $\nabla^{TM}$ of
$g^{TM}$ and a $G$-invariant Hermitian connection on the line bundle
defining  the spin$^c$ structure (cf. \cite[Appendix D]{LaMi89}).
Let $\nabla^{S(TM)\otimes E}$ be the  Hermitian
connection on $S(TM)\otimes E$  
induced by $\nabla^{S(TM)}$ and $\nabla^E$.

Let $dv_{M}$ denote  the Riemannian volume form on $(M, g^{TM})$.
For $s\in \cC^\infty (M, S(TM)\otimes E)$, 
its $L^2$-norm $\|s\|_{0}$
is defined by
\begin{align}\label{a1.4}
    \|s\|_{0}^2 = \int_{M} |s(x)|^2 dv_{M}(x).
    \end{align}
Let $\left\langle  \cdot,\cdot\right\rangle$ denote the Hermitian
product on $\cC^\infty (M, S(TM)\otimes E)$ corresponding to
$\|\cdot\|_{0}^2$, and let $L^2(M, S(TM)\otimes E)$ be 
the space of $L^{2}$-sections of $S(TM)\otimes E$ on $M$.

Let $D^{E}_{M}$ be the Spin$^c$-Dirac operator 
defined by (cf. \cite[Appendix D]{LaMi89})
\begin{align}\label{1.4}
D^{E}_{M}=\sum_{i=1}^{n} c(e_i)\nabla^{S(TM)\otimes E}_{e_i}:
\cC^\infty (M, S(TM)\otimes E)\rightarrow \cC^\infty (M, S(TM)\otimes E),
\end{align}
where $\{e_i\}$ is an oriented orthonormal frame of $TM$.


Let $\varepsilon>0$ be less than the
injectivity radius of $g^{TM}$. We use the inward geodesic flow to
identify a neighborhood of the boundary $\partial M$ with  the
collar $\partial M\times[0,\varepsilon]$, and we identify $\partial
M\times \{0\}$ to the boundary $\partial M$.

Let $e_{n}$ be the inward unit normal vector field
perpendicular to $\partial M$.
 Let $e_1,\cdots,e_{n-1}$ be an
oriented orthonormal frame of $T\partial M$ so that
$e_1,\cdots,e_{n -1},$ $e_{n  }$ is an oriented orthonormal
frame of $TM|_{\partial M}$.
By using parallel transport with respect to $\nabla^{TM}$ along the
unit speed geodesics perpendicular to $\partial M$, $e_1,\cdots,
e_{n  }$ give rise to an oriented orthonormal frame of $TM$ over
$\partial M\times[0,\varepsilon]$.

The operator $D^{E}_{M}$ induces a Dirac operator on $\partial M$,
$D^E_{\partial M}:\cC^\infty (\partial M, (S(TM)\otimes E)|_{\partial
M})\rightarrow \cC^\infty (\partial M, (S(TM)\otimes E)|_{\partial M})$
defined by (cf. \cite[p. 142]{Gilkey93})
\begin{align}\label{1.5}
D^E_{\partial M}=-\sum_{i=1}^{n-1}c\left(e_{n}\right)
c\left(e_i\right)\nabla^{S(TM)\otimes E}_{e_i}
+\frac{1}{2}\sum_{i=1}^{n-1}\pi_{ii},
\end{align}
where
\begin{align}\label{1.6}
\pi_{ij}=\left.\left\langle\nabla^{TM}_{e_i}e_j,e_{n}
\right\rangle\right|_{\partial M},\ \ \ 1\leqslant i,\
j\leqslant n-1,
\end{align}
is the second fundamental form of the isometric embedding
$\imath_{\partial M}:\partial M\hookrightarrow M$. Let
$D^E_{\partial M,\pm}$ be the restrictions of $D^E_{\partial M}$ to
$\cC^\infty (\partial M, (S_\pm(TM)\otimes E)|_{\partial M})$.

As in (\ref{a1.4}), we define the Riemannian volume form
$dv_{\partial M}$ on $\partial M$, the Hermitian product
$\left\langle \cdot,\cdot\right\rangle_{\partial M, 0}$ and
the $L^2$-norm $\|\cdot\|_{\partial M, 0}$ on
$\cC^\infty (\partial M, (S(TM)\otimes E)|_{\partial M})$.

By \cite[Lemma 2.2]{Gilkey93}, $D^E_{\partial M}$ is a 
self-adjoint first order elliptic differential operator  
defined on $\partial M$. Moreover, the following
identity holds on $\partial M$:
 \comment{\begin{align}\label{a1.6}
  c\left(e_{n}\right)  D^E_{\partial M}\left(s|_{\partial
 M}\right)
= -D^E_{\partial M}\, c\left(e_{n}\right)\left(s|_{\partial
 M}\right),
 \end{align}
 from which one gets}
\begin{align}\label{a1.61}
     D^E_{\partial M,\pm}
= c\left(e_{n}\right)^{-1}\left(- D^E_{\partial M,\mp}\right)\,
c\left(e_{n}\right) .
 \end{align}
 Since the $G$-action preserves $\partial M$, the restriction of 
$\Psi^M$ to $\partial M$ is a section of $T\partial M$, i.e.,
\begin{align}\label{1.7}
\left.\Psi^M\right|_{\partial M}\in \cC^\infty (\partial M, T\partial M).
\end{align}

For $T\in{\R}$, set 
\begin{align}\label{1.8}  \begin{split}
D^E_{M,T}&=D^{E}_{M} +\sqrt{-1}T\, c\left(\Psi^M\right), \\
D^{E}_{M,\pm,T}&=
D^E_{M,T}|_{\cC^\infty  \left(M,  S_\pm(TM)\otimes E\right)},
\end{split}\end{align}
and
\begin{align}\label{1.9} \begin{split}
D^E_{\partial M, T}&=D^E_{\partial M }
- \sqrt{-1}T \, c\left(e_{n}\right)c\left(\Psi^M\right),\\
D^E_{\partial M,\pm,T}&
= D^E_{\partial M, T}|_{\cC^\infty  \left(\partial M,  \left(
S_\pm(TM)\otimes E\right)|_{\partial M}\right)}.
\end{split}\end{align}
Then $D^E_{M,T}$ exchanges the spaces associated with $S_\pm(TM)\otimes E$, 
and by (\ref{1.7}), $D^E_{\partial M, T}$ is 
self-adjoint and preserves
$\cC^\infty  \left(\partial M,  \left( S_\pm(TM)\otimes
E\right)|_{\partial M}\right)$.
Let ${\rm Spec}(D^E_{\partial M,\pm, T})$ be the spectrum of
$D^E_{\partial M,\pm, T}$.
For  $\lambda\in {\rm Spec}(D^E_{\partial M,\pm, T})$,
let $E_{\lambda,\pm,T}$ be the corresponding eigenspace.
Let $P_{\geqslant 0,\pm,T}$ (resp. $P_{>0,\pm,T}$) be
 the orthogonal projection from  
$L^{2} (\partial M,  ( S_\pm(TM)$ $\otimes E)|_{\partial M})$
onto $\oplus_{\lambda\geqslant 0}E_{\lambda,\pm,T}$ 
 (resp. $\oplus_{\lambda> 0}E_{\lambda,\pm,T}$).
We will call $P_{\geqslant 0,+,T}$
(resp. $P_{>0, - ,T}$) the APS projection associated with
$D^E_{\partial M,+, T}$ (resp. $ D^E_{\partial M,-, T}$).

For  $T\in\R $, let $(D^{E}_{M,+,T}, P_{\geqslant 0,+,T})$ (resp.
$(D^E_{M,-,T},P_{> 0,-,T})$) denote the corresponding operator with the 
the APS boundary condition \cite{APS}.
More precisely, the boundary condition of $D^{E}_{M,+,T}$ is
$P_{\geqslant 0,+,T} (s |_{\partial M})=0$ for
$s\in \cC^\infty  (M,  S_+(TM)$ $\otimes E)$
(resp. of $D^E_{M,-,T}$ is $P_{>0,-,T} (s |_{\partial M})=0$ for
$s\in \cC^\infty  (M,  S_-(TM)$ $\otimes E)$).

Both $(D^{E}_{M,+,T}, P_{\geqslant 0,+,T})$ and $(D^E_{M,-,T},P_{> 0,-,T})$
are elliptic,  and $(D^E_{M,-,T},P_{> 0,-,T})$ is the adjoint of
$(D^{E}_{M,+,T}, P_{\geqslant 0,+,T})$ (cf. \eqref{a1.61}, \cite[Theorem
2.3]{Gilkey93}). In particular, they are Fredholm operators  
and they commute with the $G$-action.

Let $Q_{APS,T}^M\left(E, \Psi \right)_{\gamma}\in \Z $, 
$\gamma\in\Lambda^*_+ $, be defined by
\begin{multline}\label{1.10}
 \bigoplus_{\gamma\in \Lambda^*_+}\,
 Q_{APS,T}^M\left(E, \Psi \right)_{\gamma}\cdot V^G_\gamma
 =   \Ind\left(D^{E}_{M,+,T}, P_{\geqslant 0,+,T}\right)\\
   :=\Ker\left(D^{E}_{M,+,T}, P_{\geqslant 0,+,T}\right)
 -\Ker\left(D^E_{M,-,T},P_{> 0,-,T}\right)\in R(G).
\end{multline}

\subsection{An invariance property of the APS   index}\label{s1.3}
\begin{prop}\label{t1.2} For $\gamma\in \Lambda^*_+$, there
exist $C_{\gamma}>0$, $T_\gamma\geqslant 0$ such that  
for $T> T_\gamma$, $s\in\cC^\infty
(\partial M, (S (TM)\otimes E )|_{\partial M} )^\gamma$, we have
\begin{align}\label{1.16c}
\left\|D^E_{\partial M, T}s\right\|^2_{\partial M,0}
\geqslant \left\|D^E_{\partial M}s\right\|^2_{\partial M,0}
+C_{\gamma} \, T^2 \|s\|^2_{\partial M,0},
\end{align}
in particular, $D^E_{\partial M, T}(\gamma)$ is invertible.
\end{prop}

\begin{proof}
From (\ref{1.5}), (\ref{1.7}) and (\ref{1.9}), we get
\begin{multline}\label{1.13}
\left(D^E_{\partial M, T}\right)^2=\left(D^E_{\partial M
}\right)^2-\sqrt{-1}T\, \sum_{i=1}^{n -1} \pi_{ii}\, \, 
c\left(e_{n}\right) c\left(\Psi^M\right)\\
+\sqrt{-1}T\,  \sum_{i=1}^{n-1} c\left(e_{n}\right)
c\left(e_i\right) \Big(\nabla^{S(TM)\otimes
E}_{e_i}\left(c\left(e_{n}\right)
c\left(\Psi^M\right)\right)\Big)\\
-2 \sqrt{-1}T\, \nabla^{S(TM)\otimes E}_{\Psi^M}
+T^2 \left|\Psi^M\right|^2.
\end{multline}

For any $K\in \kg$, let $L_{K}$ denote the
Lie derivative of $K$ acting on $\cC^\infty (M, S(TM)\otimes
E)$ and thus also on $\cC^\infty (\partial M, (S(TM)\otimes
E)|_{\partial M})$. Then
\begin{align}\label{1.914}
\mu^{S(TM)\otimes E}(K):=\nabla^{S(TM)\otimes E}_{K^M}-L_{K}
\in \cC^\infty (M, \End(S(TM)\otimes E)).
\end{align} 
By 
(\ref{1.11}) and (\ref{1.2z}), we have
\begin{align}\label{1.14}
\nabla^{S(TM)\otimes E}_{\Psi^M}=\sum_{i=1}^{\dim G}\Psi_iL_{V_i }
+\sum_{i=1}^{\dim G}\Psi_i\left(\nabla^{S(TM)\otimes
E}_{V_i^M}-L_{V_i}\right).
\end{align}
In view of  (\ref{a1.111}), it is clear that each $L_{V_i }$,
$1\leqslant i\leqslant \dim G$, acts as a bounded operator on 
$L^{2}(\partial M, (S (TM)\otimes E)|_{\partial M} )^\gamma$.

On the other hand, since $\Psi^M$ does not vanish on $\partial M$,
there exists $C>0$ such that
\begin{align}\label{1.15}
\left|\Psi^M\right|^2\geqslant 4 C\ \ \ {\rm on}\ \ \partial M.
\end{align}

We deduce from (\ref{1.13})-(\ref{1.15}) that there exists 
 $C_\gamma'>0$ such that for any $s\in\cC^\infty
(\partial M, $ $(S (TM)\otimes E )|_{\partial M} )^\gamma$, we have
\begin{align}\label{1.16}
\left\|D^E_{\partial M, T}s\right\|^2_{\partial M,0}
\geqslant \left\|D^E_{\partial M}s\right\|^2_{\partial M,0}
-T C_\gamma' \|s\|^2_{\partial M,0} + 4 T^2C \|s\|^2_{\partial M,0}.
\end{align}
The  (\ref{1.16}) implies that Proposition \ref{t1.2} holds
with $T_\gamma =2C_\gamma'/C$.
\end{proof}

\begin{prop}\label{t1.1} For  $\gamma\in \Lambda^*_+$, there
exists $T_\gamma\geqslant 0$ such that $Q_{APS,T}^M (E, \Psi)_{\gamma}$ 
does not depend on $T> T_\gamma$.
\end{prop}

\begin{proof} For $\gamma\in \Lambda^*_+$, let 
$(D^{E}_{M,+,T}(\gamma),P_{\geqslant 0,+,T}(\gamma))$ denote the
corresponding operator with the APS 
boundary condition \cite{APS}, 
which is just the restriction of $(D^{E}_{M,+,T}
,P_{\geqslant 0,+,T} )$ to the corresponding  $\gamma$ component.
Thus, $(D^{E}_{M,+,T}(\gamma),P_{\geqslant 0,+,T}(\gamma))$ is
elliptic and defines a Fredholm operator, the index of which is
given by (\ref{1.10}),
\begin{align}\label{a1.10}
\Ind\left(D^{E}_{M,+,T}(\gamma), P_{\geqslant 0,+,T}(\gamma)\right)=
Q_{APS,T}^M \left(E, \Psi\right)_{\gamma}\cdot V^G_\gamma.
\end{align}
By Proposition \ref{t1.2}, there exists $T_\gamma\geqslant 0$ such
that $(D^{E}_{M,+,T}(\gamma),P_{\geqslant 0,+,T}(\gamma))$ forms a
continuous family of Fredholm operators for $T> T_\gamma$.
Therefore, $\Ind (D^{E}_{M,+,T}(\gamma), P_{\geqslant 0,+,T}(\gamma) )$ does
not depend on $T> T_\gamma$. By  (\ref{a1.10}),
this completes the proof of our proposition.
\end{proof}

\begin{defn}\label{t11.3}
By Proposition \ref{t1.1}, for $\gamma\in \Lambda_{+}^*$,
we can associate an integer $Q_{APS }^M\left(E,\Psi\right)_{\gamma}$
that is equal to $Q_{APS,T}^M\left(E,\Psi\right)_{\gamma}$ 
for $T> T_\gamma$.
\end{defn}

\begin{rem}\label{t11.2}  The same argument  shows that
 the APS type index
$Q_{APS }^M\left(E,\Psi\right)_{\gamma}$ does not depend on 
the given metrics and connections.
It only depends on the
homotopy class of $\Psi$ as long as 
$\Psi^M|_{\partial M}$ does not vanish over $\partial M$.
\end{rem}

\subsection{Transversal index and APS index}\label{s1.5}


\begin{thm}\label{t1.3}  For
    $\gamma\in \Lambda^*_+$, the following identity holds:
\begin{align}\label{1.18}
\Ind\left(\sigma_{E, \Psi}^M\right)_{\gamma}
=Q_{APS}^M\left(E, \Psi\right)_{\gamma}.
\end{align}
\end{thm}

The proof of Theorem \ref{t1.3} consists of two steps. In a
first step, by  applying  a result of Braverman  \cite[Theorem
5.5]{Brav02}, we express $\Ind \left(\sigma_{E, \Psi}^M\right)_{\gamma}$
as the  $L^2$-index of a Dirac operator on $\wi{M}= M\cup (\partial M\times
(-\infty,0])$, and we show that the difference of the above $L^2$-index and
$Q_{APS}^M\left(E, \Psi\right)_{\gamma}$ is equal to 
an index on the cylindrical end. In a second step, we prove that 
the index on the cylindrical end is zero.


We start by deforming our geometric data to those on a manifold with
cylindrical end.

Recall that $g^{T\partial M}$ is the Riemannian metric on $\partial
M$ induced by  $g^{TM}$.
 We use the inward geodesic flow to identify a neighborhood
of   $\partial M$ with the collar $\partial M\times [0,\var]$. As
$g^{TM}$ is $G$-invariant, the $G$-action on $\partial M\times
[0,\var]$ is induced by the $G$-action on $\partial M$, and  there
exists a family of metrics
 $g^{T\partial M}(x_n)$ on $T\partial M$ verifying
\begin{align}\label{3.1c}
 g^{TM}_{(y,x_n)}= g^{T\partial M}_{y}(x_n) + (dx_n)^2,\quad
(y,x_n)\in \partial M\times [0, \var].
\end{align}

For $(y,x_n)\in \partial M\times [0, \var]$,
we identify  $S(TM)_{(y,x_n)}, E_{(y,x_n)}$
to $S(TM)_{(y,0)}, E_{(y,0)}$ by using the parallel transport with
respect to $\nabla ^{S(TM)}$, $\nabla ^{E}$ along the geodesic
$[0,1] \ni t\to (y, t x_n)$. Thus, the restrictions of
 $(S(TM), h^{S(TM)})$, $(E, h^E)$ to  $\partial M\times [0, \var]$
are the pull-back of their restrictions $(S(TM)|_{\partial M},
h^{S(TM)}|_{\partial M})$, $(E|_{\partial M}, h^E|_{\partial M})$ to
${\partial M}$. Moreover, the $G$-actions on $S(TM)$, $E$ on
$\partial M\times [0, \var]$ are induced by the $G$-actions on
$S(TM)|_{\partial M}$, $E|_{\partial M}$ under this identification.

By the homotopy invariance of the transversal index
$\Ind (\sigma^M_{E,\Psi})_{\gamma}$ (cf. \eqref{1.3}) and of the APS
index  $Q_{APS }^M\left(E,\Psi\right)_{\gamma}$ (cf. Remark \ref{t11.2}), 
to establish Theorem \ref{t1.3}, we may and we will
 assume that $\var=2$ and that $g^{TM}$, $h^{S(TM)}$,
$\nabla^{S(TM)}$,  $\nabla^{E}$, $\Psi$ 
have product structures on
$\partial M\times [0, 2]$, and that the $G$-actions on objects such
as $E, S(TM)$ on $\partial M\times [0, 2]$ are the product of the
$G$-actions on their restrictions to $\partial M$ and the identity
in the direction $[0, 2]$.

We attach now an infinite cylinder $\partial M\times (-\infty, 0]$
to $M$ along the boundary $\partial M$ and extend trivially all
objects on $M$ to $\wi{M}= M\cup (\partial M\times (-\infty, 0])$.
We decorate the extended objects on $\wi{M}$ by a ``$\
\widetilde{\ }\ $''.   
Thus for  $(y,x_n)\in \partial M\times (-\infty, 2]$, we have
\begin{align}\label{3.3c}\begin{split}
\wi{\Psi}  (y,x_n )&= \Psi(y,0)\in \kg,\quad
g^{T\wi{M}}_{(y,x_n)}= g^{T\partial M}_{y} + (dx_n)^2,\\
 (S  (T\wi{M} ),& h^{S (T\wi{M} )}, \nabla^{S (T\wi{M} )} )
|_{\partial M\times (-\infty, 2]} = \pi_{1}^* (S(TM)|_{\partial M},
h^{S(TM)}|_{\partial M},
\nabla^{S(TM)}|_{\partial M} ),\\
  (\wi{E},h^{\wi{E}}, & \nabla^{\wi{E}}  )|_{\partial M\times
(-\infty, 2]}= \pi_{1}^*  (E|_{\partial M}, h^E|_{\partial M},
\nabla^{E}|_{\partial M} ),
\end{split}
\end{align}
with $\pi_{1}: \partial M\times  (-\infty, 2]\to  \partial M$ the
natural projection.

Let $D^{\wi{E}}_{\wi{M}}$ be the Spin$^c$ Dirac operator on $\cC^\infty_{0}
(\wi{M}, S(T\wi{M})\otimes \wi{E})$ defined as in (\ref{1.4}).
By (\ref{1.4}), (\ref{1.5}) and (\ref{3.3c}), we have on  $\partial M\times
(-\infty, 2]$,
\begin{align}\label{3.62}
D^{\wi{E}}_{\wi{M}} = c(e_{n}) D^E_{\partial M} + c(e_{n})
\frac{\partial}{\partial x_{n}}.
\end{align}
For any $h\in \cC^{\infty}(\wi{M})$, 
let $D^{\wi{E}}_{\wi{M}, h}$ be the operator on $\cC_0^\infty( \wi{M},
S(T\wi{M})\otimes \wi{E})$ defined by
\begin{align}\label{3.60}
D^{\wi{E}}_{\wi{M}, h} = D^{\wi{E}}_{\wi{M}} + \sqrt{-1} \, h  \,
c\Big(\wi{\Psi}^{\wi{M}}\Big).
\end{align}
Let  ${\bf H}^1_{\pm, h}(\wi{M})$ be the Sobolev
space  obtained by completion of $\cC_0^\infty( \wi{M},$
$S_{\pm}(T\wi{M})\otimes \wi{E})$ under the norm $\|\cdot\|_{h,1}$ defined by
\begin{align}\label{3.821}
\|s\|_{h,1}^2=\|s\|_0^2+\left\|D^{\wi{E}}_{\wi{M}, h}s\right\|_0^2.
\end{align}

Let $f$ be a strictly positive $G$-invariant smooth function on $\wi{M}$
 such that $f|_M\equiv 1$,
 and such that 
for $(y,x_{n})\in\partial M\times (-\infty, 0]$,
\begin{align}\label{3.4c}
f(y,x_n) \text{ does not depend on  } y \text{ and }
f(y,x_n) = e^{-x_{n}} \text{ if } x_{n}\leqslant - 1.
\end{align}
For $T>0$, $Tf$ is an admissible function on $\wi{M}$ for the triple
$(S(T\wi{M})\otimes \wi{E}, \nabla^{S(T\wi{M})\otimes
\wi{E}},\wi{\Psi})$ in the
sense of Braverman \cite[Definition 2.6]{Brav02} as we are in the 
product case.

By a result of Braverman \cite[Theorem 5.5]{Brav02} (cf. also \cite{MZ12}), 
for $T>0$, $\gamma\in\Lambda_+^*$,
$D^{\wi{E}}_{\wi{M}, Tf}(\gamma)$,
$D^{\wi{E}}_{\wi{M},\pm, Tf}(\gamma)$ extend to  bounded  Fredholm
operators,  for which we keep  the same notation,
\begin{align}\label{3.601}
D^{\wi{E}}_{\wi{M},\pm, Tf}(\gamma): {\bf
H}^1_{\pm, Tf}(\wi{M})^\gamma \to {L}^2
\left(\wi{M},S_{\mp}(T\wi{M})\otimes \wi{E}\right)^\gamma,
\end{align}
and the following identity holds:
\begin{align}\label{3.83}
    \Ind\left(D^{\wi{E}}_{\wi{M},+,Tf }(\gamma)\right)
    =\Ind \left(\sigma^{{M}}_{{E}, \Psi}\right)_{\gamma}\cdot V^G_\gamma.
\end{align}

Set
\begin{align}\label{3.841}\begin{split}
    &\wi{M}_1=\partial {M}\times (-\infty,1]\subset \wi{M},\quad
    \wi{M}_{2}=\partial  {M}\times (-\infty,2]\subset \wi{M},\\
&Z=\partial M\times [0,2]\subset \wi{M}.
\end{split}\end{align}
Let $\xi\in\cC^\infty([0,2])$ be such that
\begin{align}\label{3.843}
\xi|_{[0,1/2]}=1,\ \ 0\leqslant\xi|_{[1/2,3/2]}\leqslant 1,\ \
\xi|_{[3/2,2]}=0 ,
\end{align}
and such that
\begin{align}\label{3.844}
\varphi=\left(1-\xi^2\right)^{1/2}
\end{align}
is smooth.
 Clearly, $\xi$ extends to $\wi{M}_{2}$ by setting $\xi=1$ on
 $\wi{M}_0=\wi{M}_1\setminus (\partial M\times(0,1])$. It also
 extends to $M$ by setting  $\xi=0$ on $M\setminus (\partial
 M\times [0,2))$. Thus $\varphi$ also extends to $\wi{M}_0$ and
 $M\setminus (\partial M\times [0,2))$. 
Set
\begin{align}\label{3.845}\begin{split}
 &H=L^2\left(\wi{M},   S(T\wi{M})\otimes \wi{E}\right)\oplus
 L^2\left(Z,(S(TM)\otimes {E})|_Z\right),\\
&H'=L^2\left(\wi{M}_{2},\left.\left( S(T\wi{M})
\otimes \wi{E}\right)\right|_{\wi{M}_{2}}\right)\oplus
 L^2\left( M, S\left(TM\right)
\otimes  E\right).
\end{split}\end{align}
 
Let $U : H\to H'$ be defined by :
\begin{align}\label{3.846}
 (s_{1},s_{2})\in H\longrightarrow \Big(\xi \,  s_{1} -\varphi \,  s_{2}, 
 \varphi\,   s_{1}+  \xi\,   s_{2}\Big)\in H'.
\end{align}
Let $U^{*}: H'\to H$ be the adjoint of $U$. By  (\ref{3.846}), 
$U^{*}(s_{1},s_{2})= (\xi \,  s_{1} +\varphi \,   s_{2},
-\varphi\, s_{1}+ \xi\,  s_{2})\in H$.  On sees easily that 
$U$ is unitary (cf. \cite[\S 3.2]{Bu95}), that is,
\begin{align}\label{3.847}
   U^*U = {\rm Id}_H,\ \ \ \ \ \ \ UU^*={\rm Id}_{H'}.
\end{align}

We fix $\gamma\in\Lambda_+^*$ and let $T>0$.
If $W$ is one of $\wi{M}_{2}$, $M$ and $Z$, 
let $(D^{\wi{E}}_{W,+,Tf}(\gamma), P^W_{\geqslant 0,+,Tf}(\gamma))$ be
the operator with the APS boundary condition:
\begin{multline}\label{3.848}
 {\bf H}^1_{+,Tf}\left(W,P^W_{\geqslant 0,+,Tf}\right)^\gamma
 =\left\{ u\in {\bf H}^1_{+,Tf}(W)^\gamma,\ \ P^W_{\geqslant
 0,+,Tf}(\gamma) \left(u|_{\partial W}\right)=0\right\}\\
\longrightarrow
 L^2\left(W,\left.\left(S_-(T\wi{M})\otimes
 \wi{E}\right)\right|_W\right)^\gamma.
\end{multline}
Since $f=1$ on $M$ and $Z$, we know that for $W=M$ or $Z$,
\begin{align}\label{3.83c}
    \Big(D^{\wi{E}}_{W,+,Tf}(\gamma), P^W_{\geqslant 0,+,Tf}(\gamma)\Big)
    =\Big(D^{{E}}_{W,+,T}(\gamma), P_{\geqslant 0,+,T}(\gamma)\Big),
\end{align}
and they are Fredholm as explained in Section \ref{s1.2}.

By (\ref{3.601}), (\ref{3.847}) and (\ref{3.83c}),  we see that
\begin{align}\label{3.849}
 U \left\{D^{\wi{E}}_{\wi{M},+,Tf }(\gamma)+\Big(D^{{E}}_{Z,+,T}(\gamma),
P_{\geqslant 0,+,T}(\gamma)\Big)\right\}U^*:\end{align}
 $${\bf H}^1_{+,Tf}\left(\wi{M}_{2},
 P^{\wi{M}_{2}}_{\geqslant 0,+,Tf}\right)^\gamma
\oplus {\bf H}^1_{+,Tf}\left(M,
P^{M}_{\geqslant 0,+,Tf}\right)^\gamma$$
$$\to L^2\left(\wi{M}_{2},\left.\left(S_-(T\wi{M})\otimes
 \wi{E}\right)\right|_{\wi{M}_{2}}\right)^\gamma
\oplus L^2\Big(M,S_-\left(TM\right)\otimes {E}\Big)^\gamma$$
 is Fredholm.

 By the construction of $U$, it is clear that $U$ preserves the
 APS boundary conditions on the corresponding boundary components.
 Moreover, the difference
 \begin{multline}\label{3.8491}
     U \left\{D^{\wi{E}}_{\wi{M},+,Tf }(\gamma) 
     +\Big(D^{{E}}_{Z,+,T}(\gamma),
P_{\geqslant 0,+,T}(\gamma)\Big)\right\}U^* \\
- \left(D^{\wi{E}}_{\wi{M}_{2},+,Tf}(\gamma),
P^{\wi{M}_{2}}_{\geqslant 0,+,Tf}(\gamma)\right)
-\Big(D^{{E}}_{M,+,T}(\gamma),
P_{\geqslant 0,+,T}(\gamma)\Big)
\end{multline}
 is a zero-order differential operator with compact
support\footnote{Indeed, for any $s\in \cC_0^\infty( \wi{M}_{2} ,
(S_+(T\wi{M})\otimes
 \wi{E})|_{\wi{M}_{2} })\oplus \cC_0^\infty(   M, (S_+(T\wi{M})\otimes
 \wi{E})|_{  M})$ which is supported
 in $\wi{M}_{2}\setminus (\partial M\times [0,2])$,
 the difference operator in (\ref{3.8491}) acts on $s$ as a zero
 operator.},
 which implies that it is a compact operator.
Thus, $$ \left(D^{\wi{E}}_{\wi{M}_{2},+,Tf}(\gamma),
P^{\wi{M}_{2}}_{\geqslant 0,+,Tf}(\gamma) \right)
+ \Big(D^{{E}}_{M,+,T}(\gamma), P_{\geqslant
 0,+,T}(\gamma) \Big)$$ is Fredholm.
In particular, $\left(D^{\wi{E}}_{\wi{M}_{2},+,Tf}, P^{\wi{M}_{2}}_{\geqslant
 0,+,Tf}(\gamma) \right)$ is Fredholm. Moreover, we have
\begin{multline}\label{3.850}
  \Ind  \left(D^{\wi{E}}_{\wi{M},+,Tf }(\gamma)\right)+
  \Ind \Big(D^{{E}}_{Z,+,T}(\gamma),
  P_{\geqslant 0,+,T}(\gamma)\Big)\\
 = \Ind \left(D^{\wi{E}}_{\wi{M}_{2},+,Tf}(\gamma),
P^{\wi{M}_{2}}_{\geqslant 0,+,Tf}(\gamma) \right)
+  \Ind \Big(D^{{E}}_{M,+,T}(\gamma),
P_{\geqslant 0,+,T}(\gamma) \Big).
\end{multline}
Note that $\partial Z=(\partial M\times\{0\})\cup(-\partial M\times\{2\})$.
By (\ref{3.3c}) and (\ref{3.62}), 
 $P_{\geqslant 0,+,T}|_{\partial M\times \{0\}} = P_{\geqslant 0,+,T}$,
 $P_{> 0,-,T}|_{\partial M\times \{0\}} = P_{> 0,-,T}$,
and $P_{\geqslant 0,+,T}|_{-\partial M\times \{2\}}$ 
(resp.  $P_{> 0,-,T}|_{-\partial M\times \{2\}}$) is the orthogonal projection 
from  $L^2(\partial M, (S_+(TM)$ $\otimes E)|_{\partial M} )$ 
onto $\oplus_{\lambda\leqslant 0} E_{\lambda,+,T}$ 
(resp. $L^2(\partial M,$ $(S_-(TM)\otimes E)|_{\partial M} )$ 
onto $\oplus_{\lambda< 0} E_{\lambda,-,T}$), thus from 
the product structure on $Z$, we get 
(compare with \cite[Proposition 3.11]{APS})
\begin{align}\label{3.84c}\begin{split}
&\Ker\left(D^{{E}}_{Z,+,T}(\gamma), P_{\geqslant 0,+,T}(\gamma)\right)
=0, \\
& \Ker\left(D^{{E}}_{Z,-,T}(\gamma), P_{> 0,-,T}(\gamma)\right)
=\Ker\left(D^E_{\partial M,-,T} (\gamma)\right).
\end{split}\end{align}
Combining (\ref{3.84c}) with Proposition \ref{t1.2},  
 for $T> T_\gamma$, we get
\begin{align}\label{3.852}
 \Ind \Big(D^{{E}}_{Z,+,T}(\gamma), P_{\geqslant 0,+,T}(\gamma)\Big) =0  .
\end{align}

By Definition \ref{t11.3},  (\ref{a1.10}), (\ref{3.83}), (\ref{3.850})
and (\ref{3.852}), 
for any $T> T_\gamma$,
\begin{align}\label{3.854}
\Ind \left(D^{\wi{E}}_{\wi{M}_{2},+,Tf}(\gamma),
P^{\wi{M}_{2}}_{\geqslant 0,+,Tf}(\gamma) \right)
= \left(\Ind \left(\sigma^{{M}}_{{E}, \Psi}\right)_{\gamma}
-Q_{APS }^M\left(E, \Psi\right)_{\gamma} \right) \cdot V_{\gamma}^G .
\end{align}


For a second step, we need to prove the following Lemma.
\begin{lemma}\label{t1.7} For $\gamma\in \Lambda^*_+$, there exists 
$T_2> T_\gamma$ such that for $T> T_2$, we have
\begin{align}\label{3.85c}
\Ind \left(D^{\wi{E}}_{\wi{M}_{2},+,Tf}(\gamma),
P^{\wi{M}_{2}}_{\geqslant 0,+,Tf}(\gamma) \right)=0.
\end{align}
\end{lemma}
\begin{proof}
Following Bismut-Lebeau \cite[pp. 115-116]{BL91}, let $U_1=\partial
M\times (-\infty, 1)$, $U_2=\partial M\times (0,2]$ be an open
covering of $\wi{M}_{2}$. Let $h_1,\, h_2$ be two smooth
$G$-invariant functions on $\wi{M}_{2}$ such that $h_1^2,\, h_2^2$
form a partition of unity associated with the covering $\{U_i\}_{i=1}^2$.

By  (\ref{1.2z}), (\ref{1.4}),  (\ref{1.914}),
(\ref{1.14}) and (\ref{3.60}), we deduce that
\begin{multline}\label{3.857}
\left(D^{\wi{E}}_{\wi{M}, Tf}\right)^2=\left(D^{\wi{E}}_{\wi{M}}\right)^2
+\sqrt{-1}T\, \sum_{i=1}^n c(e_i)
c\left(\nabla^{T\wi{M}}_{e_i}\left(f\wi{\Psi}^{\wi{M}}\right)\right)\\
 -2 \sqrt{-1}T\, f\sum_{i=1}^{\dim G}\wi{\Psi}_iL_{V_i}
 -2 \sqrt{-1}T\, f\sum_{i=1}^{\dim G}
 \wi{\Psi}_i\,\mu^{S(T\wi{M})\otimes\wi{E}}(V_{i})
 +T^2\left|f\wi{\Psi}^{\wi{M}}\right|^2.
\end{multline}
By  (\ref{3.3c}),  $\wi{\Psi}_i$, $\mu^{S(T\wi{M})\otimes\wi{E}}(V_i)$
are constant on $x_{n}$ on $\wi{M}_{2}$, thus from (\ref{1.11}), there 
exists $C_1>0$ such that the following inequality 
holds:
\begin{align}\label{3.858}
\left\|\sum_{i=1}^{\dim
 G}\wi{\Psi}_iL_{V_i}\right\|+\left\| \sum_{i=1}^{\dim
 G}\wi{\Psi}_i\,\mu^{S(T\wi{M})\otimes\wi{E}}(V_i)\right\|
\leqslant C_1,
 \end{align}
(where the norm in (\ref{3.858}) refers to operators  acting on
 $L^{2}(\wi{M},S(T\wi{M})\otimes \wi{E})^\gamma$).

By (\ref{3.4c}), (\ref{3.857}) and (\ref{3.858}), there exists $C>0$ 
such that for $T>0$,
$s\in \cC_0^\infty(U_{1},
(S(T\wi{M})$ $\otimes\wi{E})|_{\wi{M}_{2}})^\gamma$, we have
\begin{align}\label{3.86c}
  \left\|D^{\wi{E}}_{\wi{M}_{2},Tf} s\right\|^2_0 
= \left\langle \Big(D^{\wi{E}}_{\wi{M}_{2},Tf}\Big)^2 s, s\right\rangle
\geqslant \left\|D^{\wi{E}}_{\wi{M}} s\right\|_0^2
+  T^2 \left\|f |\wi{\Psi}^{\wi{M}}|s\right\|_0^2
 -CT  \|f s\|_0  \| s\|_0  \, .
\end{align}
Thus from (\ref{1.15}),  (\ref{3.3c}), (\ref{3.4c}) and (\ref{3.86c}),
 we see that  there exist $T_1> T_\gamma$, $C_2>0$ such that for
any $T> T_1$, $s\in \cC_0^\infty(U_1,
(S(T\wi{M})\otimes\wi{E})|_{\wi{M}_{2}})^\gamma$, we have
\begin{align}\label{3.855}
  \left\|D^{\wi{E}}_{\wi{M}_{2},Tf} s\right\|^2_0
\geqslant \left\|D^{\wi{E}}_{\wi{M}_{2}} s\right\|_0^2 +C_2  T^2 \|s\|_0^2.
\end{align}

By Green's formula, 
(\ref{3.62}) and (\ref{3.60}) imply that for $s\in
\cC_0^\infty(\wi{M}_{2} , (S(T\wi{M})\otimes
\wi{E})|_{\wi{M}_{2}})$, we have
\begin{multline}\label{3.856}
\left\|D^{\wi{E}}_{\wi{M}_{2},Tf} s \right\|^2_0
=\int_{\wi{M}_{2}}\left\langle s, \left(D^{\wi{E}}_{\wi{M}_{2},Tf}
\right)^2 s\right\rangle dv_{\wi{M}_{2}}
+\int_{\partial\wi{M}_{2}}\left\langle  s,
c(-e_n)D^{\wi{E}}_{\wi{M}_{2},Tf} s\right\rangle  dv_{\partial\wi{M}_{2}}\\
=\int_{\wi{M}_{2}}\left\langle
s,\left(D^{\wi{E}}_{\wi{M}_{2},Tf} \right)^2s\right\rangle
dv_{\wi{M}_{2}}-\int_{\partial\wi{M}_{2}}\left\langle
s,\nabla_{-e_n}^{S(T\wi{M})\otimes\wi{E}} s\right\rangle
dv_{\partial\wi{M}_{2}}\\
-\int_{\partial\wi{M}_{2}}\left\langle
s, D^{\wi{E}}_{\partial\wi{M}_{2},Tf}  s\right\rangle
dv_{\partial\wi{M}_{2}}.
 \end{multline}

By the Lichnerowicz formula (cf. \cite[Appendix D]{LaMi89}), we have
\begin{align}\label{3.8581}
 \left(D^{\wi{E}}_{\wi{M}}\right)^2=-\Delta^{\wi{E}}+\mO(1),
 \end{align}
where $\Delta^{\wi{E}}$ is the Bochner Laplacian, 
and $\mO(1)$ is an endomorphism of $S(T\wi{M})\otimes\wi{E}$. 
By (\ref{3.3c}), the  fiberwise norm of  this endomorphism has 
an uniform upper bound over $\wi{M}$.

By Green's formula, we have, for any
$s\in \cC_0^\infty(\wi{M}_{2} , (S(T\wi{M})\otimes
 \wi{E})|_{\wi{M}_{2}})$,
\begin{align}\label{3.8582}
\int_{\wi{M}_{2}}\left\langle -
\Delta^{\wi{E}} s,s\right\rangle dv_{\wi{M}_{2}}
-\int_{\partial\wi{M}_{2}}\left\langle
s,\nabla_{-e_n}^{S(T\wi{M})\otimes\wi{E}} s\right\rangle
dv_{\partial\wi{M}_{2}}
=\left\|\nabla^{S(T\wi{M})\otimes\wi{E}} s\right\|_0^2 .
 \end{align}
Note that $f=1$ on $\partial \wi{M}_{2}$. By (\ref{1.16c}), 
for any $T>T_{\gamma}$, $s\in\cC_0^\infty(\wi{M}_{2},
(S(T\wi{M})\otimes \wi{E})|_{\wi{M}_{2}})^\gamma$ with
$P^{\wi{M}_{2}}_{\geqslant 0,\pm,Tf}(s|_{\partial\wi{M}_{2}})=0$,
we have
\begin{align}\label{3.8580}
\int_{\partial\wi{M}_{2}}\left\langle s,
D^{\wi{E}}_{\partial\wi{M}_{2},Tf}  s\right\rangle dv_{\partial\wi{M}_{2}}
\leqslant -\sqrt{C_{\gamma}} T 
\left\| s|_{\partial \wi{M}_{2}}\right\|_{\partial \wi{M}_{2},0}^2 
\leqslant 0.
 \end{align}

 As $h_2$ has compact support in $\partial M\times
 (0,2]\subset\wi{M}_{2}$, on which $f\equiv 1$, by (\ref{1.15}),
 (\ref{3.3c}),  (\ref{3.857}), (\ref{3.858}),
 (\ref{3.856})-(\ref{3.8580}),    there exist
 $C_3, C_4, C_{5}>0$ such that for $T> 1$ and
$s\in\cC_0^\infty(\wi{M}_{2},(S(T\wi{M})\otimes
 \wi{E})|_{\wi{M}_{2}})^\gamma$ with
$P^{\wi{M}_{2}}_{\geqslant 0,\pm,Tf}(s|_{\partial\wi{M}_{2}})=0$, we
have
\begin{align}\label{3.859}
\left\|D^{\wi{E}}_{\wi{M}_{2},Tf}(h_2s)\right\|^2_0
\geqslant C_{3} \left\|D^{\wi{E}}_{\wi{M}_{2}}(h_2s)\right\|^2_0
-C_4T\|h_2s\|_0^2+ C_5T^2 \|h_2s\|_0^2.
 \end{align}

 Since $h_1^2+h_2^2\equiv 1$, for any
$s\in\cC_0^\infty(\wi{M}_{2},(S(T\wi{M})\otimes
 \wi{E})|_{\wi{M}_{2}})^\gamma$ with
$P^{\wi{M}_{2}}_{\geqslant 0,\pm,Tf}(s|_{\partial\wi{M}_{2}})=0$, we obtain
\begin{multline}\label{3.860}
\left\|D^{\wi{E}}_{\wi{M}_{2},Tf}s\right\|^2_0=
\left\|h_1D^{\wi{E}}_{\wi{M}_{2},Tf}s\right\|^2_0
+\left\|h_2D^{\wi{E}}_{\wi{M}_{2},Tf}s\right\|^2_0\\
 \geqslant \frac{1}{2}\left\|D^{\wi{E}}_{\wi{M}_{2},Tf}(h_1s)\right\|^2_0
+  \frac{1}{2}\left\|D^{\wi{E}}_{\wi{M}_{2},Tf}(h_2s)\right\|^2_0-
 \|c((dh_1)^*)s\|_0^2- \|c((dh_2)^*)s\|_0^2,
\end{multline}
where $(dh_i)^*\in T\wi{M}_{2}$ is the  dual of $dh_i$
with respect to $g^{T\wi{M}}$.

 By (\ref{3.855}), (\ref{3.859}) and (\ref{3.860}),
 there exist  $C_6, C_{7}>0$ such that for any
$T> T_1> T_\gamma$ and 
$s\in\cC_0^\infty(\wi{M}_{2},(S(T\wi{M})\otimes \wi{E})|_{\wi{M}_{2}})^\gamma$
with $P^{\wi{M}_{2}}_{\geqslant 0,\pm,Tf}(s|_{\partial\wi{M}_{2}})=0$,
we have
\begin{align}\label{3.861}
\left\|D^{\wi{E}}_{\wi{M}_{2},Tf}s\right\|^2_0
\geqslant \frac{1}{2}\left\|D^{\wi{E}}_{\wi{M}}(h_1s)\right\|^2_0
+ \frac{C_{3}}{2}\left\|D^{\wi{E}}_{\wi{M}_{2}}(h_2s)\right\|^2_0
-C_6 T\|s\|_0^2+C_7 T^2\|s\|_0^2.
 \end{align}
 By $D^{\wi{E}}_{\wi{M}_{2}}(h_is)= h_{i}D^{\wi{E}}_{\wi{M}_{2}}s +
 c((dh_{i})^{*})s$, $h_1^2+h_2^2\equiv 1$ and (\ref{3.861}), there exist
$T_2> T_\gamma$, $C_8, C_{9}>0$ such that
for $T> T_2$,
 $s\in\cC_0^\infty(\wi{M}_{2},(S(T\wi{M})\otimes
 \wi{E})|_{\wi{M}_{2}})^\gamma$ with
$P^{\wi{M}_{2}}_{\geqslant 0,\pm,Tf}(s|_{\partial\wi{M}_{2}})=0$,
the following holds:
\begin{align}\label{3.862}
\left\|D^{\wi{E}}_{\wi{M}_{2},Tf}s\right\|_0^2
\geqslant  C_{8}\| D^{\wi{E}}_{\wi{M}_{2}} s\|_{0}^2+ C_9 T^2 \|s\|_0^2 .
 \end{align}
 \comment{

}
 By Proposition \ref{t1.2},
 (\ref{a1.10}) and (\ref{3.862}), we get
Lemma \ref{t1.7}.
\end{proof}

 By (\ref{3.854}) and Lemma \ref{t1.7}, 
the proof of Theorem \ref{t1.3} is completed.

\section{Quantization   for proper moment maps:
 proof of Theorem \ref{t0.1}}\label{s1.6}

The purpose of this section is to give a proof of Theorem \ref{t0.1}.
This proof
consists of two steps. In a first step, we reduce Theorem
\ref{t0.1} to a vanishing result for the transversal index and
then use Theorem \ref{t1.3} to  interpret the later as a vanishing
result for the APS type index. In a second step, we apply
the analytic localization method developed in \cite{BL91},
\cite{TZ98} and \cite{TZ99} to prove the vanishing of this APS
type index.

We use the assumptions and the notation in the Introduction. 
Also, for any real one form
$\upsilon$ on a Riemannian manifold, we denote by $\upsilon^*$ the
corresponding vector field on this manifold.

Recall that $(M,\om, J^M)$ is a noncompact symplectic manifold with
 a compatible almost-complex structure $J^M$, and $g^{TM}=
 \omega(\cdot, J^M \cdot)$ is the associated Riemannian metric on $M$.
We have the canonical splitting
$TM\otimes_{\R}\C= T^{(1,0)}M \oplus T^{(0,1)}M $, for
the complexification of $TM$, with
\begin{align}\label{a1.19}
    \begin{split}
    T^{(1,0)}M =& \{u\in TM\otimes_{\R}\C: J^M u = \sqrt{-1} u\},\\
    T^{(0,1)}M =& \{u\in TM\otimes_{\R}\C: J^M u = -\sqrt{-1} u\}.
    \end{split}
\end{align}
Let $T^{*(0,1)}M $ be the dual of $T^{(0,1)}M $.

The almost complex structure $J^M$ on $TM$ determines a canonical
spin$^c$-structure on $TM$ with the associated Hermitian line
bundle $\det(T^{(1,0)}M)$. Moreover,  we have
\begin{align}\label{1.19}
S(TM)=\Lambda\left(T^{*(0,1)}M\right),\ \ \
S_\pm(TM)=\Lambda^{\frac{\rm even}{\rm odd}}\left(T^{*(0,1)}M\right).
\end{align}
For any $W\in TM$, we write 
$W=w+\ov{w}\in T^{(1,0)}M \oplus T^{(0,1)}M $.
Let ${w}^*\in T^{*(0,1)}M$ correspond to $w$ so that
$({w}^*, \ov{u})= g^{TM}(w, \ov{u})$ for any $\ov{u}\in T^{(0,1)}M$.
Then
\begin{align}\label{a1.20}
c(W)= \sqrt{2}( {w}^*\wedge - i_{\ov{w}})
\end{align}
defines the Clifford action of $W$ on
$\Lambda(T^{*(0,1)}M)$. It interchanges $\Lambda^{{\rm
even}}(T^{*(0,1)}M)$ and $\Lambda^{{\rm odd}}(T^{*(0,1)}M)$.
The Levi-Civita connection $\nabla^{TM}$ together with the almost
complex structure $J$ induces by projection a canonical Hermitian
connection $\nabla^{T^{(1,0)}M}$ on $T^{(1,0)}M$. This induces a
Hermitian connection $\nabla^{\det}$ on $\det(T^{(1,0)}M)$. The
Clifford connection $\nabla^{\Lambda\left(T^{*(0,1)}M\right)}$ on
$\Lambda\left(T^{*(0,1)}M\right)$ is induced by the Levi-Civita
connection $\nabla^{TM}$ and the connection $\nabla^{\det}$ (cf.
 \cite[Appendix D]{LaMi89},  \cite[\S 1.3]{MM07} and \cite[\S 1a)]{TZ98}).

We take $E=L$, with $L$ the prequantum line bundle on $M$ in
the Introduction, and denote by $\Omega^{0,\bullet}(M,L)=\cC^\infty(M,
\Lambda (T^{*(0,1)}M)\otimes L)$.
Let $D^L_{M}$ be the corresponding Dirac operator defined as in (\ref{1.4}).

Recall that 
the moment map $\mu:M\rightarrow \kg^*$ is assumed to be proper.
Let $X^{\mathcal H}$ be the Hamiltonian vector field of
${\mathcal H}=|\mu|^2$, i.e.,
\begin{align}\label{1.20a}
 i_{X^{\mathcal H}}\omega = d \mathcal H .
\end{align}
By (\ref{0.2}), (\ref{0.4}), \eqref{1.11} and (\ref{1.20a}),  we find (cf.
\cite[(1.19)]{TZ98}),
\begin{align}\label{1.22}
X^{\mathcal H}=-J^M(d{\mathcal H})^*=-2J^M\sum_{i=1}^{\dim G}\mu_i\left(
d\mu_i\right)^*=2\sum_{i=1}^{\dim G}\mu_iV_i^M=2\,\mu^M.
\end{align}

For any regular value $a>0$ of ${\mathcal H}=|\mu|^2$, denote by
$M_a$ the compact $G$-manifold  with boundary defined by
\begin{align}\label{1.20}
M_a=\left\{x\in M: {\mathcal H}(x)\leqslant a\right\}.
\end{align}
By (\ref{1.22}), $\mu^M$
does not vanish on the boundary $\partial M_a={\mathcal
H}^{-1}(a)$ of $M_a$.

Let $a'>a>0$ be two regular values of ${\mathcal H}$. Let
$M_{a,a'}$ denote the compact $G$-manifold with boundary
$M_{a,a'}=\overline{M_{a'}\setminus M_a}$.
By the additivity of the transversal index
(cf. \cite[Theorem 3.7, \S 6]{A74} and \cite[Prop. 4.1]{Par01}),
we have for  $\gamma\in\Lambda_+^*$,
\begin{align}\label{1.24}
\Ind \left(\sigma_{L,\mu}^{M_{a'}}\right)_{\gamma}
-\Ind \left(\sigma_{L,\mu}^{M_a}\right)_{\gamma}
=\Ind \left(\sigma_{L,\mu}^{M_{a,a'}}\right)_{\gamma}.
\end{align}



Let ${\rm Cas}_G= -\sum_{i} V_i \, V_i$ be the Casimir operator
associated with $G$.
Let $c_{\gamma}\geqslant 0$ be defined by the following formula,
\begin{align}\label{1.23x}
{\rm Cas}_G|_{V_\gamma^G}= c_\gamma \Id_{V_\gamma^G}.
\end{align}
Clearly, $c_{\gamma=0}=0$. 
As ${\rm Cas}_G|_{V_\gamma^G}= - \sum_i L_{V_i}(\gamma)^2$,
from (\ref{3.858}) and (\ref{1.23x}), we get
\begin{align}\label{1.24x}
c_\gamma=\Big\| \sum_{i=1}^{\dim G}L_{V_i}(\gamma)^2\Big\|.
\end{align}

By Theorem \ref{t1.3}, (\ref{1.24}) and (\ref{1.24x}),
the following result is a
reformulation of Theorem \ref{t0.1}, with a more precise form of
the bound $a_\gamma$.
\begin{thm}\label{t1.5} Fix $\gamma\in \Lambda^*_+$.
Then for any regular values
$a', a$ of ${\mathcal H}$ with $a'>a>\frac{ c_\gamma}{ 4\pi^2}$, the
following identity holds:
\begin{align}\label{1.23}
 Q_{APS}^{M_{a,a'}}\left(L,\mu\right)_{\gamma} =0.
\end{align}
\end{thm}

\begin{proof} If $\gamma=0$,  (\ref{1.23})
    has been proved in \cite[Theorems 2.6, 4.3]{TZ99}.
The proof for general $\gamma\in \Lambda^*_+$ is a 
  modification of the proof of \cite[Theorem 2.6]{TZ99}
where it is assumed that $\gamma=0$.
Let $\gamma\in \Lambda^*_+$ and $a'>a>\frac{c_\gamma}{ 4\pi^2}$ be
fixed.

By (\ref{1.22}), (\ref{1.8}) becomes  in the current situation
 (cf. \cite[(1.20)]{TZ98} and \cite[(1.19)]{TZ99}),
\begin{align}\label{1.26}
 D^L_{M,T}=D^L_{M} +\frac{\sqrt{-1}T}{2} c\left(X^{\mathcal H}\right):
 \Omega^{0,\bullet}\left(M_{a,a'},L\right) \rightarrow \Omega^{0,\bullet}
 \left(M_{a,a'},L\right).
\end{align}
Let $e_1,\cdots,e_{n}$ be an oriented orthonormal frame of $TM_{a,a'}$.
By \cite[Theorem 1.6]{TZ98}, the following 
formula holds:
\begin{multline}\label{1.27}
 \left(D^L_{M,T}\right)^2=\left(D^L_{M}\right)^2 +
 \frac{\sqrt{-1}T}{ 4}\sum_{j=1}^{n}c\left(e_j\right)
 c\left(\nabla^{TM}_{e_j}X^{\mathcal H}\right) \\
   - \frac{\sqrt{-1}T}{2} {\rm Tr}\left[\nabla^{TM}
 X^{\mathcal H}|_{T^{(1,0)}M}\right]
 +\frac{T}{ 2}\sum_{i=1}^{\dim G}
 \left(\sqrt{-1}c\left(J^MV_i^M\right)c\left(V_i^M\right)
 +\left|V_i^M\right|^2\right)\\
 +4\pi T{\mathcal H}-2\sqrt{-1}T\sum_{i=1}^{\dim
 G}\mu_iL_{V_i }+\frac{T^2}{ 4}\left|X^{\mathcal H}\right|^2.
\end{multline}

Let $\mU$ be a $G$-invariant open neighborhood of
 $\partial  M_{a,a'}$ in $M_{a,a'}$ such that 
 $X^{\mathcal H}$ does not vanish on $\ov{\mU}$.
Since $X^{\mathcal H}$ does not vanish on $\partial M_{a,a'}$,
the existence of $\mU$ is clear.
Let $\mU'$ be a $G$-invariant
open subset of $M_{a,a'}$ such that $\overline{\mU'}\cap(\partial
M_{a,a'})=\emptyset$, $\mU\cup\mU'=  M_{a,a'}$.

By using that $L_{V_i}$ acts as a bounded operator on 
$L^2( M_{a,a'}, \Lambda (T^{*(0,1)}M)\otimes L)^{\gamma}$
and (\ref{1.16c}) instead of \cite[Theorem 2.1]{TZ99}, 
then by  proceeding in exactly the same way as in 
\cite[Proof of Proposition 2.4]{TZ99}, we know that 
there exist $T_1>0$, $C_1>0$ 
(depending on $\mU$ and $\gamma$) such that for any $T> T_1$ and
 $s\in \Omega^{0,\bullet}(M_{a,a'},L )^\gamma$
 with $\supp(s)\subset \mU$ and
 $P_{\geqslant 0, \pm, T} (s|_{\partial  M_{a,a'}})=0$,
 the following inequality holds:
 \begin{align}\label{a4.30}
\left \|D^L_{M,T}  s\right\|^2_{0}\geqslant C_1 \left(\left\| D^L_{M}
s\right\|^2_{0}
 + T^2 \|s\|_{0}^2 \right).
 \end{align}

For any $\varepsilon >0$, set
\begin{align}\label{a4.30a}
G_{T,\varepsilon}^L=\left(D^L_{M,T}\right)^2-(4\pi-\varepsilon)
T{\mathcal H}+2\sqrt{-1}T\sum_{i=1}^{\dim
G}\mu_iL_{V_i}.
\end{align}
Clearly, $G_{T,\varepsilon}^L$ is of the same form as $F_T^L$ in
\cite[(2.6)]{TZ98}, with $4\pi T{\mathcal H}$ in \cite[(2.6)]{TZ98} 
being replaced by $\varepsilon T{\mathcal H}$.

By replacing $4\pi {\mathcal H}$ in \cite[(2.26)]{TZ98} 
 by $\varepsilon {\mathcal H}$ in 
the proof of \cite[Proposition 2.2, Case 2]{TZ98}, 
from  (\ref{1.27}) and (\ref{a4.30a}), we know the analogue of 
\cite[Proposition 2.2]{TZ98} holds for the operator 
$G_{T,\varepsilon}^L$: 
for any $x\in M_{a,a'}\setminus \partial M_{a,a'}$, there exist an open 
neighborhood $U_x\subset M_{a,a'}\setminus \partial M_{a,a'}$
 of $x$ and   $C_x>0$, $b_x>0$ such that for any $T> 1$ and $s\in
 \Omega^{0,\bullet}(M_{a,a'},L )$ with $\supp (s)\subset U_x$,
 we have
\begin{align}\label{1.31}
  \left\langle G_{T,\varepsilon}^Ls,s\right\rangle  \geqslant C_x
  \left(\left\|D^L_{M}s\right\|^2_{0}+ \left(T-b_x\right)  \|s\|^2_{0}\right).
\end{align}
From \eqref{1.31},  as explained in  \cite[\S 2c)]{TZ98}, there exist
$C_2>0$, $b_1>0$ such that
for any $T> 1$ and $s\in \Omega^{0,\bullet}(M_{a,a'},L)$
with $\supp (s)\subset \mU' $, we have
\begin{align}\label{1.316}
  \left\langle G_{T,\varepsilon}^Ls,s\right\rangle  \geqslant C_2
  \left(\left\|D^L_{M}s\right\|^2_{0}+ \left(T-b_1\right)  \|s\|^2_{0}\right) .
\end{align}

\begin{lemma}\label{t1.6} There exists  $0<\varepsilon<4\pi$ such that for
 any $s\in\Omega^{0,\bullet}(M_{a,a'},L)^\gamma$, one has
\begin{align}\label{1.29}
  \left\langle\Big((4\pi-\varepsilon)
 {\mathcal H}-2\sqrt{-1} \sum_{i=1}^{\dim
 G}\mu_iL_{V_i }\Big)s,s\right\rangle \geqslant 0.
\end{align}
\end{lemma}

\begin{proof}  Since $a'>a> \frac{c_\gamma}{4\pi^2}$, there exists
$\varepsilon\in(0,4\pi)$ such that the following inequality holds
on $M_{a,a'}$:
\begin{align}\label{1.298}
{\mathcal H}\geqslant \frac{4c_\gamma}{(4\pi-\varepsilon)^2}   .
\end{align}
By the Cauchy inequality and (\ref{1.24x}), we
have that for any $s\in\Omega^{0,\bullet}(M_{a,a'},L)^\gamma$,
\begin{align}\label{1.30}
\begin{split}
    \Big\vert  \Big\langle \sum_{i=1}^{\dim G}
    \mu_{i} L_{V_i } s, s\Big\rangle\Big\vert
& \leqslant \frac{1}{2}\sum_{i=1}^{\dim G}
 \left(\frac{4\pi-\varepsilon}{ 2 } \left\Vert  \mu_{i}s\right\Vert^2_{0}
+\frac{2 }{ 4\pi-\varepsilon}\left\|L_{V_i}s\right\|^2_{0}\right)\\
 &= \frac{4\pi-\varepsilon}{4}  \left\langle {\mathcal H}s,   s
 \right\rangle
 +\frac{c_\gamma}{ 4\pi-\varepsilon}
\left\|   s \right\|^2_{0}.
     \end{split}\end{align}
From (\ref{1.298}) and (\ref{1.30}), we obtain for any
$s\in\Omega^{0,\bullet}(M_{a,a'},L)^\gamma$,
\begin{align}\label{1.299}
 \left\langle\Big((4\pi-\varepsilon){\mathcal H}
-2\sqrt{-1} \sum_{i=1}^{\dim G}\mu_iL_{V_i }\Big)s,s\right\rangle
\geqslant  \left\langle\left( \frac{4\pi-\varepsilon}{2} {\mathcal H}
- \frac{ 2 c_\gamma}{4\pi-\varepsilon} \right)s,s\right\rangle \geqslant 0.
\end{align}

The proof of Lemma \ref{t1.6} is completed.
\end{proof}

Let $\varepsilon>0$ be fixed as in Lemma \ref{t1.6}.
By Lemma \ref{t1.6}, (\ref{a4.30a}) and (\ref{1.316}),  we have
that for any $T> 1$ and $s\in
\Omega^{0,\bullet}(M_{a,a'},L )^\gamma$ with $\supp (s)\subset
\mU' $,
\begin{align}\label{1.3167}
  \left\|D^L_{M,T}s\right\|^2_{0}
=\left\langle \left(D^L_{M,T}\right)^2s,s\right\rangle 
\geqslant  \left\langle G_{T,\varepsilon}^L s,s\right\rangle  \geqslant C_2
  \left(\left\|D^L_{M}s\right\|^2_{0}+ \left(T-b_1\right)  \|s\|^2_{0}\right) .
\end{align}

Let $h_1$, $h_2$ be two smooth
$G$-invariant functions on $M_{a,a'}$ such that $h_1^2$, $h^2_2$
forms a partition of unity associated with the $G$-invariant open covering 
$\mU'$, $\mU$ of $M_{a,a'}$.\footnote{
We can take $h_2$ a radical function with respect to $|\mu|^2$
near $\partial M_{a,a'}$ as in \eqref{3.843}, then  $h_1$, $h_2$
are automatically $G$-invariant.}

Let $s\in \Omega^{0,\bullet}(M_{a,a'},L )^\gamma$ with
$P_{\geqslant 0, \pm, T} (s|_{\partial  M_{a,a'}})=0$.
Clearly, $h_1s$, $h_2s$ still belong to
$\Omega^{0,\bullet}(M_{a,a'},L )^\gamma$ with ${\rm
supp}(h_2s)\subset \mU $ and $P_{\geqslant 0, \pm, T}
((h_2s)|_{\partial  M_{a,a'}})=0$, while $\supp (h_1s)\subset \mU'$.
By applying (\ref{a4.30}) to $h_{2}s$, (\ref{1.3167}) to $h_{1}s$,
and by proceeding as in
 (\ref{3.860})-(\ref{3.862}) (cf. \cite[pp. 115-116]{BL91}),
 we obtain constants $C_3>0$, $b_2>0$ such that for any $T> T_1$ and
 $s\in \Omega^{0,\bullet}(M_{a,a'},L )^\gamma$
 with $P_{\geqslant 0, \pm, T} (s|_{\partial  M_{a,a'}})=0$,
 the following inequality holds:
 \begin{align}\label{1.3168}
  \left\langle \left(D^L_{M,T}\right)^2s,s\right\rangle  \geqslant C_3
  \left(\left\|D^L_{M}s\right\|^2_{0}+ \left(T-b_2\right)  \|s\|^2_{0}\right) .
\end{align}
By Proposition \ref{t1.2}, (\ref{1.22}), (\ref{1.26}) and (\ref{1.3168}), 
we have $Q_{APS,T}^{M_{a,a'}}(L,\mu)_{\gamma}=0$ 
for $T>0$ large enough.
Combining this with Definition \ref{t11.3}, we get Theorem \ref{t1.5}.

By Theorems \ref{t1.3}, \ref{t1.5} and (\ref{1.23x}),
we get Theorem \ref{t0.1}.
\end{proof}


\section{A vanishing result for the APS index}\label{s33}

In this section, we prove 
the vanishing result (\ref{0.103}). 


This section is organized as follows. In Section \ref{s33.1}, we
state 
(\ref{0.103}) as a vanishing theorem 
on the APS index, Theorem \ref{t33.2}. 
In Section \ref{s33.2}, we construct a suitable function 
$\psi:M\times N\to\kg$ which is homotopy to the function $Y$ 
in Theorem  \ref{t33.2} such that the associated operator 
with the APS boundary condition is invertible.
In Section \ref{s33.3}, we prove the invertibility of 
the operator associated to $\psi$, Theorem \ref{ta2.3} 
up to a pointwise estimate, Lemma \ref{t33.6}.
In Sections \ref{s33.4}-\ref{sa.6}, we prove  Lemma \ref{t33.6}.

We make the same assumptions and we use the same notation as in
the Introduction and in Section \ref{s1.6}.

\subsection{A vanishing theorem for the APS index}\label{s33.1}

For convenience, we recall the basic setting.
Let $(M,\omega)$, $(N, \omega^N)$  be two symplectic manifolds with
symplectic forms $\omega,\ \omega^N$, and $\dim M=n$. 
We assume  that $M$ is noncompact and that $N$ is compact.

Let $J^M,\ J^N$ be  almost complex structures on $TM,\ TN$
  such that $\om(\cdot, J^M\cdot)$ defines a metric
$g^{TM}$ on $TM$, and $\omega^N(\cdot, J^N\cdot)$ defines a metric
$g^{TN}$ on $TN$.
Let $(L, h^L, \nabla^L)$ be a prequantum line bundle on $(M,\omega)$, 
and let $(F,h^F, \nabla^F)$ be a prequantum line bundle on $(N,\omega^N)$ 
(cf. \eqref{0.1}).
\comment{
Let $(L, h^L)$ be a Hermitian line bundle on $M$ with Hermitian
connection $\nabla^L$, and $(F,h^F)$ a Hermitian line bundle on $N$ with
Hermitian connection $\nabla^F$. Let $R^L= (\nabla^L)^2$, $R^F=
(\nabla^F)^2$ be the associated curvatures. We assume that
\begin{align}\label{2.1}
\omega= \frac{\sqrt{-1}}{2\pi} R^L, \quad \omega^N=
\frac{\sqrt{-1}}{2\pi} R^F.\end{align}
}


 Suppose that $G$ acts (on the left) on $M,\ N$ and 
its actions on $M,\ N$ lift to $L$
and $F$. Moreover, we assume that these $G$-actions
preserve the above metrics and
the connections   on $TM,\ TN,\ L,\ F$ and $J^M,\ J^N$.

Let the moment map $\mu:M\rightarrow \kg$ be
 defined as in (\ref{0.3}). Let $\eta:N\rightarrow \kg$ be the moment
 map defined in the same way for $(N,\omega^N)$ and $(F,h^F, \nabla^F)$.

We will keep 
the same notation for the natural lifts of the objects
on $M$, $N$ to $M\times N$. In particular, $L\otimes F$ is the
Hermitian line bundle on $M\times N$ induced by $L$ and $F$ with the
Hermitian connection $\nabla^{L\otimes F}$ induced by $\nabla^L,\ \nabla^F$.

The $G$-action on $M\times N$ is defined by $g\cdot (x,y)= (gx,gy)$
 for $(x,y)\in M\times  N$.  We define the symplectic form $\Omega$ and
 the almost complex structure $J$ on $M\times N$ by
 \begin{align}\label{2.2}
\Omega(x,y)= \omega(x)+ \omega^N(y),\quad J= (J^M, J^N).
\end{align}
The induced moment map $\theta: M\times N\to \kg$ is given by
 \begin{align}\label{2.3}
\theta(x,y)= \mu(x)+\eta(y).
\end{align}
Since $\mu : M \to \kg$ is proper,
 $\theta:M\times N\to \kg$ is also proper.

For $A>0$, set
\begin{align}\label{2.12}
 \begin{split}
 \mM_1=& \{(x,y)\in M\times N: |\mu(x)|^2= A\}=\partial M_A\times N,\\
\mM_2=& \{(x,y)\in M\times N: |\theta(x,y)|^2= 2A\},\\
\mM =& \{(x,y)\in M\times N: |\mu(x)|^2\geqslant A,
|\theta(x,y)|^2\leqslant 2A\} \subset  M\times N,
 \end{split}\end{align}
where $\partial M_A$ is the boundary of $M_A $ defined in (\ref{1.20}).
As  $\mu,\theta$ are proper and $M$ is noncompact, 
$|\mu(M)|^{2}$, $|\theta(M\times N)|^{2}$ 
contain  a half line of $\R$, thus for $A$ large enough, 
$\mM_1$, $\mM_2$ are nonempty.

\begin{rem} \label{t33.1} 
Since $N$ is a compact manifold, there exists $C_0>0$ such that
\begin{align}\label{33.0}
 |\eta|\leqslant C_0\ \ \ \  \ {\rm on}\  \ N.
\end{align}
By (\ref{2.3}) and (\ref{33.0}),  we have $|\theta|\leqslant |\mu|+C_0$.
Set  $A_0=\Big(\frac{C_0}{ \sqrt{2}-\sqrt{5/3}}\Big)^2$.
By (\ref{2.12}), for  $A> A_0$, 
we have
\begin{align}\label{33.2}
  |\mu| \geqslant \sqrt{2A}-C_0 \geqslant \sqrt{5 A/3}
  \quad \text{on }Ê\mM_{2} .
\end{align}
Thus for any $A> A_0$,  we have $\mM_1\cap \mM_2=\emptyset$.
\end{rem}

By Sard's theorem, given $C>0$, there exists $C'>C$ 
which is a regular value of the functions 
 $|\mu|^2$ and $\frac{1}{2}|\theta|^2$ on $M\times N$.

 From now on, let $A> A_{0}$ be a regular value of
 $|\mu|^2$ and $\frac{1}{2}|\theta|^2$. 
By Remark \ref{t33.1} and (\ref{2.12}), $\mM$ is a smooth $G$-manifold 
with boundary $\partial \mM= \mM_1\cup \mM_2 .$

From  (\ref{0.4}), (\ref{1.11}) and (\ref{2.3}),  for any
$1\leqslant i\leqslant \dim G$, we have 
\begin{align}\label{2.11}
  \begin{split}
 V_{i}^{M\times N}=& V_{i}^M+ V_{i}^N,
\quad  \theta_{i}=
\mu_{i}+\eta_{i},\\
 \left(d\mu_i\right)^*=&J^MV_i^M,\ \ \ \
 \left(d\eta_i\right)^*=J^NV_i^N.
\end{split}\end{align}
By (\ref{1.2z}) and the first equation of  (\ref{2.11}), we get
\begin{align}\label{2.25}\begin{split}
 \mu^{M\times N}
 =\mu^{M}+\mu^{N},\quad
 \theta^{M\times N}
 = \theta^{M}+\theta^{N}.
\end{split}\end{align}
By (\ref{1.22}),  $\mu^{M}$ does not vanish on $\mM_1$,
so that $\mu^{M\times N}$ also does not vanish on $\mM_1$.
Similarly, $\theta^{M\times N}$ does not vanish on $\mM_2$.

Let $Y:\mM\rightarrow \kg$ be a $G$-equivariant smooth map such that
\begin{align}\label{33.3}
    Y|_{\mM_1}=\mu|_{\mM_1},\ \quad \ \ \
    Y|_{\mM_2}=\theta|_{\mM_2}.
\end{align} 
Then $Y^\mM\in \cC^{\infty}(\mM, T\mM)$ does not vanish on $\partial \mM$.



 The main result of this section can be stated as follows.

 \begin{thm}\label{t33.2} There exists $A_1\geqslant A_0$ such that
for any regular value $A> A_1$ of $|\mu|^2$
and $\frac{1}{ 2}|\theta|^2 $, the following identity holds:
\begin{align}\label{33.4} Q_{APS}^\mM(L\otimes F,Y)_{\gamma=0}=0.
\end{align}
 \end{thm}

 \begin{rem}\label{t3.2a} 
By Theorem \ref{t1.3}, (\ref{33.4}) is equivalent to (\ref{0.103})
with $a=\frac{1}{2}b=A$.
 \end{rem}

\subsection{Proof of Theorem \ref{t33.2}
}\label{s33.2}


\begin{lemma}\label{t33.3} There exist two real smooth functions
$\wi{\alpha},\wi{\phi} \in \cC^\infty(\R)$ verifying the following properties,
\begin{align}\label{a.01}
\begin{split}
\wi{\alpha}(t) =& \begin{cases}
t^2 \,,\quad &\text{for}\,\,\,  t\leqslant \frac{1}{3},\\
1 \,,\quad &\text{for}\,\,\, t\geqslant \frac{2}{3},
\end{cases} \quad\quad
\wi{\phi}(t) = \begin{cases}
1-t^3 \,,\quad &\text{for}\,\,\,  t\leqslant \frac{4}{11},\\
2(1-t) \,,\quad &\text{for}\,\,\, t\geqslant \frac{2}{3},
\end{cases}\\
 \wi{\alpha}(t)+&\wi{\phi}(t) \geqslant \frac{29}{27} \quad
 \quad\text{ for } \,\, \,
\frac{1}{3}\leqslant t\leqslant \frac{2}{3}\, ; 
\quad \quad\wi{\phi}'(t)<0 \quad\text{ for } \,\, \,
0< t\leqslant  1. 
\end{split}\end{align}
\end{lemma}

\begin{proof}
We may set $\wi{\alpha}_{0}(t) = t^2$, $\wi{\phi}_{0}(t) =
1-t ^3$ on $ t\leqslant \frac{3}{8}$; $\wi{\alpha}_{0}(t) = 1$,
$\wi{\phi}_{0}(t) = 2(1-t)$ on $ t\geqslant \frac{5}{8}$; and assume
that $\wi{\alpha}_{0}$, $\wi{\phi}_{0}$ are linear on
$\frac{3}{8}\leqslant t\leqslant \frac{5}{8}$.
By   
smoothing out the  linear interpolation, starting from 
$\wi{\alpha}_{0}$, $\wi{\phi}_{0}$, we get $\wi{\alpha},\wi{\phi}$ verifying
\eqref{a.01}.
\end{proof}

Let  $A> A_0$ be a regular value of $|\mu|^2$ and
$\frac{1}{2}|\theta|^2$. Set
 \begin{align}\label{a.02} \alpha_{A} (t)= \wi{\alpha} \left(\frac{t}{A}-1
\right),\ \ \
 \phi_{A}(t)= \wi{\phi} \left(\frac{t}{A}-1 \right).
\end{align}
The following identities hold:
\begin{align}\label{a.02x}
    \alpha_{A}' (t)=\frac{1}{ A}\, \wi{\alpha}' \left(\frac{t}{A}-1\right),\ \ \
 \phi_{A}'(t)=\frac{1}{ A}\, \wi{\phi}' \left(\frac{t}{A}-1 \right).
\end{align}

Let  $\beta_{A}\in \cC^\infty(M\times N)$ be defined by
\begin{align}\label{a5.5}
\beta_{A} = |\mu|^2 +\alpha_{A}\left(|\mu|^2\right) \left(|\theta|^2
-|\mu|^2\right).
\end{align}
Let  $\rho_{A},\, \gamma_{A},\, \psi_{A}: M\times N\to   \kg$ be
the $G$-equivariant smooth maps defined by
\begin{subequations}
\begin{align}\label{33.5}
\rho_{A}= &\theta - \phi_{A} (\beta_{A}) \eta,\\
\label{33.6}
    \gamma_{A}=&2\Big[1+ \alpha_{A}'\left(|\mu|^2\right) \left(|\theta|^2
-|\mu|^2\right) \Big] \mu + 2 \alpha_{A} (|\mu|^2)\eta,\\
\label{33.7}
    \psi_{A}=&\rho_{A}- \phi_{A}'(\beta_{A})
\left\langle \rho_{A},\eta\right\rangle\gamma_{A}.
\end{align}
\end{subequations}
For any function $f$ on $M\times N$, we denote by $d^M f$, $d^Nf$
its differentials along $M,\,N$ respectively.


The following lemma partly motivates our choice of $\psi_{A}$  (compare
with (\ref{1.22})).

\begin{lemma}\label{ta5.3} The following identity holds:
    \begin{align}\label{aa5.9}
   2\,  \psi_{A}^M =- J^M \left(d^M |\rho_{A}|^2\right)^*.
    \end{align}
\end{lemma}
\begin{proof} By (\ref{1.11}), \eqref{2.3}
 and \eqref{a5.5}-(\ref{33.6}), we have
\begin{align}\label{aa5.10}
    \begin{split}
d \beta_{A} =& 2\Big[1+ \alpha_{A}'(|\mu|^2) (|\theta|^2 -|\mu|^2)
  - \alpha_{A} (|\mu|^2)\Big] \mu_{j} d^M \mu_{j}
+2 \alpha_{A} (|\mu|^2) \theta_{j} d \theta_{j}\\
= &\gamma_{Aj}d^M \mu_{j} +  2 \alpha_{A} (|\mu|^2)
\theta_{j} d^N \eta_{j},\\
d \rho_{Aj} = & d \theta_{j} - \phi_{A}'(\beta_{A}) \eta_{j}d \beta_{A} -
\phi_{A}(\beta_{A}) d^N \eta_{j}.
\end{split}\end{align}
From \eqref{2.11} and \eqref{aa5.10}, we get
\begin{align}\label{aa5.11}
    \begin{split}
    (d\theta_{j})^* = &J^M V_{j}^M + J^N V_{j}^N,\\
(d\beta_{A})^* =&  J^M \gamma_{A}^M + 2 \alpha_{A} (|\mu|^2)
 J^N \theta^N,\\
(d\rho_{Aj})^* = &J^M V_{j}^M
- \phi_{A}'(\beta_{A}) \eta_{j} J^M\gamma_{A}^M \\
&+ \Big(1- \phi_{A}(\beta_{A})\Big) J^N V_{j}^N 
-2 \phi_{A}'(\beta_{A}) \alpha_{A} (|\mu|^2)
\eta_{j} J^N \theta^N.
    \end{split}\end{align}
From (\ref{1.2z}), \eqref{33.7} and the third equality in (\ref{aa5.11}),
we get
\begin{align}\label{5.12}
2 \psi_{A}^M
= 2 \rho_{Aj} \left[V_j^M
 - \phi_{A}'(\beta_{A})\eta_j\gamma_{A}^M\right]  
= - 2 J^M \rho_{Aj}(d^M\rho_{Aj})^*
=- J^M \left(d^M |\rho_{A}|^2\right)^*.
\end{align}

The proof of Lemma \ref{ta5.3} is completed.
\end{proof}


\begin{lemma}\label{t33.4} There exists $A_2\geqslant A_0$ such that
    for any regular value $A> A_2$ of $|\mu|^2$ and
$\frac{1}{2}|\theta|^2$, the following identities hold:
\begin{align}\label{a5.8}
 \begin{split}
\psi_{A} |_{\mM_1} =&  \mu,     \hspace{20mm} 
\qquad \beta_{A}|_{\mM_1} =A, \\
\psi_{A} |_{\mM_2} =&\Big(1 + \frac{4}{A}\left\langle
\theta,\eta\right\rangle \Big)\theta , \quad \beta_{A}|_{\mM_2} =2A   .
 \end{split}\end{align}
 Moreover, the following inequality holds:
 \begin{align}\label{33.8}
1 + \frac{4}{A}\left\langle \theta,\eta\right\rangle\geqslant  \frac{1}{2}
\quad \mbox{ on }  \mM_2.
\end{align}
In particular, $\psi_{A}^\mM $ does not vanish on $\partial \mM$.
\end{lemma}

\begin{proof} On $\mM_1$, we have $|\mu|^2=A$. By
(\ref{a.01})-(\ref{33.5}), we deduce that on $\mM_{1}$,
 \begin{align}\label{33.8c}
\beta_{A}=A,\quad\phi_{A}(\beta_{A})=1,\quad \phi_{A}'(\beta_{A})=
\alpha_{A}\left(|\mu|^{2}\right)=0,\quad \rho_{A}=\mu.
\end{align}
The first two equalities in (\ref{a5.8}) follow 
from (\ref{33.7}) and (\ref{33.8c}).


From (\ref{33.2}) and (\ref{a.01})-(\ref{33.6}), for $A> A_0$,
we have on $\mM_{2} = (|\theta|^2)^{-1}(2A)$:  
 \begin{align}\label{33.9c}\begin{split}
&\alpha_{A}\left(|\mu|^{2}\right)=1,\quad
\alpha_{A}'\left(|\mu|^{2}\right)=0, \quad\gamma_{A}=2 \theta, \\
&\beta_{A}=2A,\quad \phi_{A}\left(\beta_{A}\right)=0,\quad
\rho_{A}=\theta,\quad \phi_{A}'\left(\beta_{A}\right)=-\frac{2}{A}.
\end{split}\end{align}
 By (\ref{33.7}) and (\ref{33.9c}),
 the last two identities in (\ref{a5.8}) hold.
Since $|\theta|=\sqrt{2A}$ on $\mM_2$,  (\ref{33.0}) implies that
there exists $A_2\geqslant A_0$ such
 that (\ref{33.8}) holds on $\mM_2$ for $A> A_2$.

 We have seen just after (\ref{2.25}) that $\mu^\mM$ and
 $\theta^\mM$ do not vanish  on $\mM_1$, $\mM_2$ respectively.
 Hence by (\ref{1.2z}), (\ref{a5.8}) and (\ref{33.8}),   
$\psi_{A}^\mM$ does not vanish  on $\partial \mM$ when $A> A_2$.

The proof of Lemma \ref{t33.4} is completed.
\end{proof}

Let $D^{L\otimes F}: \Omega^{0,\bullet}(M\times N, L\otimes F)\to
\Omega^{0,\bullet}(M\times N, L\otimes F)$
 be the Spin$^c$ Dirac operator on
 $M\times N$ (cf. \eqref{1.4} and Section \ref{s1.6}).
Following (\ref{1.8}), let $D_{\mM,T}$
be the operator defined for $T\in \R$, by
\begin{align}\label{a.2}
D_{\mM,T} = D^{L\otimes F} + \sqrt{-1}T
c\left(\psi_{A}^\mM\right): \Omega^{0,\bullet}(\mM, L\otimes F)\to
\Omega^{0,\bullet}(\mM, L\otimes F).
\end{align}
Let $P_{\geqslant 0, \pm, T}$ be the APS projections associated with 
$D_{\partial \mM, \pm, T}$ induced by $D_{\mM,T}$ (cf. \eqref{1.9}).

\begin{thm}\label{ta2.3} There exists $A_1\geqslant A_2$ such that
if $A> A_1$ is a regular value of $|\mu|^2$
 and $\frac{1}{2}|\theta|^2$, then there exist $C>0$, $T_{0}>0$ such that for
 any $T> T_{0}$ and $G$-invariant element $s$  of
 $\Omega^{0,\bullet}(\mM, L\otimes F)$ with
$P_{\geqslant  0, \pm, T}(s|_{\partial \mM})=0$, the following
inequality holds:
\begin{align}\label{3.118}
\left\|D_{\mM,T}s\right\|^2_0\geqslant C \left(\left\|D^{L\otimes F}
s\right\|^2_0+ T \|s\|^2_0\right).
 \end{align}
\end{thm}

\begin{proof}[Proof of Theorem  \ref{t33.2}] 
Let $A> A_1$ be a regular value of $|\mu|^2$
   and $\frac{1}{2}|\theta|^2$.  Then by  Theorem \ref{ta2.3}, 
$(D_{\partial \mM, \pm, T}(\gamma=0),P_{\geqslant 0, \pm, T}(\gamma=0))$
is invertible for $T> T_{0}$.
By Propositions \ref{t1.2}, \ref{t1.1} and Definition \ref{t11.3}, this implies
\begin{align}\label{a.0}
    Q^\mM_{APS}\left(L\otimes F,\psi_{A}\right)_{\gamma=0}=0.
\end{align}
 We connect  the map $Y$ defined in (\ref{33.3}) and $\psi_{A}$ via
 $\psi_{At}=(1-t)Y+t\, \psi_{A}$, $0\leqslant t\leqslant 1$.
  Lemma \ref{t33.4} 
shows that 
$\psi_{At}^\mM\in \cC^\infty(\mM, T\mM)$
 generated by $\psi_{At}$ via (\ref{1.2z}) does not vanish on $\partial
 \mM$ for any $0\leqslant t\leqslant 1$. By the homotopy invariance of
 the APS index (cf. Remark \ref{t11.2}) and (\ref{a.0}), we get (\ref{33.4}).
\end{proof}

The rest of the section is devoted to the proof of Theorem \ref{ta2.3}.

\subsection{Proof of Theorem \ref{ta2.3}}\label{s33.3}
Let $\{e_k\}_{k=1}^{n}$ (resp. $\{f_i\}_{i=1}^{\dim N}$)
 be an oriented orthonormal frame of $TM$ (resp. $TN$).
Then $\{e_a\}_{a=1}^{\dim \mM}=\{e_k\}\cup \{f_i\}$ is an oriented
orthonormal frame  of $T\mM$.
Set
\begin{align}\label{a.3}
\begin{split}
I_{A1}=& \frac{1}{2} c\left(\left(d^M\psi_{Aj}\right)^*\right)
c\left(V_j^M+2V_j^N\right)
+ c\left(\left(d^N\psi_{Aj}\right)^*\right) c\left(V_j^M\right),\\
I_{A2}=&
 \frac{1}{2}  \left\langle \left(1+ \frac{J^M}{\sqrt{-1}} \right) V_j^M,
\left(d^M \psi_{Aj}\right)^* \right\rangle
=  \tr \left[\left(d^M \psi_{Aj}\right)|_{T^{(1,0)}M}
\otimes V_j^M\right],\\
I_{A3}=& c\left(\left(d^N\psi_{Aj}\right)^*\right)
c\left(V_j^N\right).
\end{split}\end{align}

\begin{thm} \label{ta.2} The following formula holds:
\begin{multline}\label{a.4}
D_{\mM,T}^2 = D^{L\otimes F,2} +
\sqrt{-1}T\left\{
\frac{1}{2}\sum_{k=1}^{n}
c\left(e_k\right)c\left(\nabla^{TM}_{e_k} \psi_{A}^M\right)
-  \tr \left[\left(\nabla^{TM}
 \psi_{A}^M\right)|_{T^{(1,0)}M}\right]\right\}\\
 +\sqrt{-1}T \left\{ \frac{1}{2} \sum_{i=1}^{\dim N}
c\left(f_i\right)c\left(\nabla^{TN}_{f_i} V_j^N\right)
- \tr \left[\left(\nabla^{TN}
 V_j^N\right)|_{T^{(1,0)}N}\right]\right\} \psi_{Aj}\\
+ 4 \pi T \left\langle\psi_{A}, \theta\right\rangle
+\sqrt{-1}T \left(I_{A1}+I_{A2}+I_{A3}\right) 
- 2 \sqrt{-1} T\psi_{Aj} L_{V_j}   
+ T^2 \left|\psi_{A}^\mM\right|^2 .
\end{multline}
\end{thm}
\begin{proof}
    \comment{ For any $1\leqslant a\leqslant \dim\mM$, it is clear that
\begin{align}\label{aa4.18}
\nabla^{\Lambda^{0,\bullet}}_{e_a}c\left(\psi_{A}^\mM\right)
=c\left(\psi_{A}^\mM\right)\nabla^{\Lambda^{0,\bullet}}_{e_a}+
c\left(\nabla^{T\mM}_{e_a}\psi_{A}^\mM\right).
\end{align}}

Let $\nabla^{\Lambda^{0,\bullet}}$ be a brief notation for
$\nabla^{\Lambda(T^{*(0,1)}\mM)\otimes L\otimes F}$.
By \eqref{a.2}, we deduce as in \eqref{1.13} and \eqref{3.857} that
\begin{multline}\label{a4.21}
D_{\mM,T}^2 
=  D^{L\otimes F,2}+  \sqrt{-1}T \sum_{a=1}^{\dim \mM}
c(e_a) c\left(\nabla^{T\mM}_{e_a}\psi_{A}^\mM\right) -2 \sqrt{-1} T
\nabla^{\Lambda^{0,\bullet}}_{\psi_{A}^\mM} + T^2
\left|\psi_{A}^\mM\right|^2 .
\end{multline}

From \eqref{2.11}, the definition of the moment map, and 
$L_K X= \nabla^{T\mM}_{K^\mM}X-\nabla^{T\mM}_X {K^\mM}$ 
for $K\in \kg, X\in \cC^\infty(\mM,T\mM)$, we get
(cf. \cite[Lemma 1.5]{TZ98}, \eqref{1.27})
\begin{multline}\label{aa4.23}
\nabla^{\Lambda^{0,\bullet}}_{\psi_{A}^\mM} 
= \psi_{Aj} \nabla^{\Lambda^{0,\bullet}}_{V_j^\mM}
= \psi_{Aj} L_{V_j}
+ 2\pi\sqrt{-1} \left\langle\psi_{A}, \theta\right\rangle \\
+\frac{1}{4}\sum_{k=1}^{n} c(e_k)
c\left(\nabla^{TM}_{e_k} V_j^M\right) \psi_{Aj}
+ \frac{1}{4} \sum_{i=1}^{\dim N} c(f_i)
c\left(\nabla^{TN}_{f_i} V_j^N\right) \psi_{Aj} \\
+ \frac{1}{2}\psi_{Aj} \tr \left[(\nabla^{TM}
V_j^M)|_{T^{(1,0)}M}\right] + \frac{1}{2} \psi_{Aj}\tr
\left[(\nabla^{TN} V_j^N)|_{T^{(1,0)} N}\right].
\end{multline}
By  \eqref{a.3}, we get 
\begin{align}\label{aa4.24}
\begin{split}
&\frac{1}{2}\sum_{k=1}^{n} c(e_k) c\left(\nabla^{TM}_{e_k} V_j^M\right) \psi_{Aj}
 =\frac{1}{2}\sum_{k=1}^{n} c(e_k)c\left(\nabla^{TM}_{e_k} \psi_{A}^M\right)
 -\frac{1}{2}c\left(\left(d^M\psi_{Aj}\right)^*\right) c\left(V_j^M\right),\\
 &\psi_{Aj} \tr \left[\nabla^{TM}  V_j^M|_{T^{(1,0)}M}\right] 
  = \tr \left[(\nabla^{TM}\psi_{A}^M)|_{T^{(1,0)}M}\right]  - I_{A2}.  
\end{split}\end{align}

Also by \eqref{1.11} and \eqref{2.11}, we have
\begin{multline}\label{aa4.25}
\sum_{a=1}^{\dim \mM} c(e_a)
c\left(\nabla^{T\mM}_{e_a}\psi_{A}^\mM\right) =
\sum_{k=1}^{n}
c(e_k)c\left(\nabla^{TM}_{e_k}\psi_{A}^M\right)
+  c\left(\left(d^M\psi_{Aj}\right)^*\right)c\left(V_j^N\right)\\
+  \sum_{i=1}^{\dim N} c(f_i)c\left(\nabla^{TN}_{f_i}
V_j^N\right) \psi_{Aj} 
+  c\left(\left(d^N\psi_{Aj}\right)^*\right)c\left(V_j^M+ V_j^N\right).
\end{multline}
By \eqref{a.3}, \eqref{a4.21}--\eqref{aa4.25}, we get (\ref{a.4}).
 The proof of Theorem \ref{ta.2} is completed.
\end{proof}



\begin{lemma}\label{t33.6} There exists $A_1\geqslant A_2$ such that if
$A> A_1$ is a regular value for   $|\mu|^2$ and
$\frac{1}{2}|\theta|^2$, then for any $z\in\mM $ with
$\psi_{A}^\mM(z)=0$, and any $f\in (\Lambda( T^{*(0,1)}\mM)\otimes
(L\otimes F)|_\mM)|_z$,  the following inequality holds at $z$:
\begin{multline}\label{33.14}
    {\rm Re}\left\langle\sqrt{-1}\left\{\frac{1}{2} \sum_{i=1}^{\dim N}
c\left(f_i\right)c\left(\nabla^{TN}_{f_i} V_j^N\right)  
- \tr \left[(\nabla^{TN} V_j^N)|_{T^{(1,0)}N}\right]\right\} \psi_{Aj}\, f,
f\right\rangle\\
+{\rm Re} \left\langle\Big(4 \pi  \left\langle\psi_{A},
\theta\right\rangle  + \sqrt{-1} \left(I_{A1}+I_{A2}+I_{A3}\right)\Big)
f,f\right\rangle\geqslant \pi A|f|^2 .
\end{multline}
\end{lemma}

Lemma \ref{t33.6} will be proved in Sections \ref{s33.4}-\ref{sa.6}.\\

Let $F_{\mM,T}:\Omega^{0,\bullet}(\mM, L\otimes F)\to
\Omega^{0,\bullet}(\mM, L\otimes F)$ be defined by
\begin{align}\label{33.12}
F_{\mM,T}=D_{\mM,T}^2 +2 \sqrt{-1} T\psi_{Aj} L_{V_j}.
\end{align}

\begin{prop}\label{t33.5} 
    Let $A_1>0$ be as in Lemma \ref{t33.6}.
If $A> A_1$ is a regular value for $|\mu|^2$ and
$\frac{1}{2}|\theta|^2$, then for any $z\in\mM\setminus \partial \mM$,
there exist an open neighborhood $U_z$ of $z$ in $\mM$, with
$\overline{U}_z\cap\partial\mM=\emptyset$, and $C_z>0,\, b_z>0$ such
that for any $T\geqslant 1$ and $s\in\Omega^{0,\bullet}(\mM, L\otimes F)$
with ${\supp}(s)\subset U_z$,  we have
\begin{align}\label{33.13}
{\rm Re}\, \Big\langle F_{\mM,T} s,s\Big\rangle
\geqslant  C_z\left(\left\|D^{L\otimes F}s\right\|_0^2
+\left(T-b_z\right)\|s\|_0^2\right) .
\end{align}
\end{prop}

\begin{proof} Let $A> A_1$ be a fixed regular value for $|\mu|^2$ and
$\frac{1}{2}|\theta|^2$, and we fix $z\in\mM\setminus \partial \mM$.

If $\psi_{A}^\mM(z)\neq 0$, then by (\ref{a.4}) and
(\ref{33.12}), we see that Proposition \ref{t33.5} holds.

From now on we assume that $\psi_{A}^\mM(z)=0$.
We write $z=(x_0,y_0)$ with $x_0\in M$, $y_0\in N$.
From  (\ref{1.2z}), $\psi_{A}^\mM(z)= \psi_{A}^M(z) + \psi_{A}^N(z)$ and 
$\psi_{A}^M(z)\in TM$,$\psi_{A}^N(z)\in TN$, thus
\begin{align}\label{aa4.47}
\psi_{A}^\mM(z)=0 \quad\text{if and only if } \quad  \psi_{A}^M(z)=0 \,
\, \text{ and } \,  \,  \psi_{A}^N(z)=0.
\end{align}


 Let $x'=(x_1,\cdots, x_{n})$ be the normal coordinate
system with respect to $\{e_j|_{x_0}\}_{j=1}^{n}$ near
$x_0\in M$. Let $y'=(y_1, \cdots, y_{\dim N})$ be the  normal coordinate
system near $y_0\in N$ associated with $\{f_i|_{y_0}\}_{i=1}^{\dim N}$.

By 
(\ref{aa5.9}), $\psi_{A}^M(z)=0$ implies that
$(d^{M}|\rho_{A}|^2)(z)=0$.  Thus 
we can choose the orthonormal frame
$\{e_{i}\}_{i=1}^{n}$ so that the function
$|\rho_{A}(\cdot,y_0)|^2$ 
 has the following
expression near $x_0$,
\begin{align}\label{4.108}
|\rho_{A}(x',y_0)|^2 = |\rho_{A}(x_0,y_0)|^2 + \sum_{j=1}^{n} a_j x_j^2
+ \mO(|x'|^3).
\end{align}

The following Lemma is an analogue of \cite[Lemma 2.3]{TZ98}.
\begin{lemma} The following inequality holds at the point $(x_0,y_0)$,
\begin{align}\label{4.109}
 \frac{\sqrt{-1}}{2}\sum_{k=1}^{n}
c(e_k)c\left(\nabla^{TM}_{e_k} \psi_{A}^M\right) -
\sqrt{-1} \tr \left[(\nabla^{TM}
 \psi_{A}^M)|_{T^{(1,0)}M}\right]
\geqslant -  \sum_{j=1}^{n} |a_j|,
\end{align}
and the inequality is strict if at least one of the $a_j$'s is
negative.
\end{lemma}
\begin{proof} Set
\begin{align}\label{4.110}
\psi_{A}^M{(x',y')} = - \sum_{k=1}^{n} t_k(x',y') J^M {e}_k.
\end{align}
Then Lemma \ref{ta5.3} and \eqref{4.108} imply that
\begin{align}\label{4.110a}
 t_k(x',y_0)= a_k x_k + \mO(|x'|^2).
\end{align}

Let $e_{j}= e_{j}^{1,0}+ e_{j}^{0,1}\in T^{(1,0)}M \oplus T^{(0,1)}M$. 
By \eqref{a1.20}, \eqref{4.110} and \eqref{4.110a}, 
we deduce that at the point $(x_0,y_0)$,
\begin{multline}\label{4.111}
 \frac{\sqrt{-1}}{2}\sum_{k=1}^{n}
c(e_k)c\left(\nabla^{TM}_{e_k} \psi_{A}^M\right) -
\sqrt{-1} \tr \left[(\nabla^{TM}
 \psi_{A}^M)|_{T^{(1,0)}M}\right]\\
= - \frac{\sqrt{-1}}{2}\sum_{j=1}^{n} a_j c(e_j)c(J^Me_j) -
\frac{\sqrt{-1}}{2}\sum_{j=1}^{n}
 \left\langle\left(1+ \frac{J^M}{\sqrt{-1}} \right) (-a_jJ^Me_j),
e_j \right\rangle\\
= -2  \sum_{j=1}^{n} a_j
i_{e_j^{0,1}}{e_j^{1,0*}}\wedge \geqslant -
\sum_{j=1}^{n} |a_j|, 
\end{multline}
where the last inequality is strict if at least one of the $a_j$'s
is negative.
\end{proof}


Let $\Delta^M$, $\Delta^N$ be the Bochner Laplacians on $M,N$ acting on 
$\Omega^{0,\bullet}(M, L)$, $\Omega^{0,\bullet}(N, F)$, respectively.
We still denote by  $\Delta^M$, $\Delta^N$ the induced operators
acting on $\Omega^{0,\bullet}(M\times N, L\otimes F)$, then 
$\Delta^{M\times N}= \Delta^M+ \Delta^N$ 
is the Bochner Laplacian on $M\times N$.
Clearly, they are nonpositive operators acting  on
$\Omega^{0,\bullet}(M\times N, L\otimes F)$.
From the Lichnerowicz formula for $D^{L\otimes F,2}$ (cf.
\cite[Appendix D]{LaMi89} and \cite[Theorem 1.3.5]{MM07}), we get
on $\mM$,
\begin{align}\label{4.106}
D^{L\otimes F,2} =- \Delta^{M\times N} + \mO(1),
\end{align}
and where $\mO(1)$ is an endomorphism of 
$\Lambda(T^{*(0,1)}\mM)\otimes L\otimes F$.


Let $F_{\mM,T}^{*}$ be the formal adjoint of $F_{\mM,T}$.
Note that $\left| \psi_{A}^\mM\right|^2= \left| \psi_{A}^M\right|^2
 +\left| \psi_{A}^N\right|^2$. 
From (\ref{a.4}), (\ref{33.14}), (\ref{33.12}), \eqref{4.109},
\eqref{4.110} and \eqref{4.106},   we find that 
$\frac{1}{2}\left(F_{\mM,T} + F_{\mM,T}^{*} \right)
  + \Delta^{M\times N}$ is an operator of order $0$, and near $z=(x_0,y_0)$,
\begin{multline}\label{4.112}
    \frac{1}{2}\left(F_{\mM,T} + F_{\mM,T}^{*} \right)
  + \Delta^{M\times N}   \geqslant  -T  \sum_{j=1}^{n} |a_j|
+ T^2  \sum_{j=1}^{n} t_j \left(x',y'\right)^2\\
+T^2 \left| \psi_{A}^N\left(x',y'\right)\right|^2
 + \pi  T A + \mO\left(1+ T\left|x'\right|+T \left|y'\right|\right).
\end{multline}

Let ${\var_0}>0$ be sufficiently small so that the orthonormal frame 
$\{{e}_j\}_{j=1}^{n}$ is well defined over the ball 
$B^M_{\var_0} (x_0)= \left\{x'\in M: d\left(x',x_0\right)<{\var_0}\right\}$,
and $\overline{B^M_{\var_0} (x_0)}\times \overline{B^N_{\var_0} (y_0)}
\cap \partial\mM=\emptyset$.
For any $1\leqslant j\leqslant n$, let $(\nabla_{e_j})^*$ be the formal
adjoint of $\nabla^{\Lambda^{0,\bullet}}_{e_j}$. 
We have (cf. \cite[(1.2.9)]{MM07})
\begin{align}\label{4.115c}
(\nabla_{e_j})^* 
= - \nabla^{\Lambda^{0,\bullet}}_{e_j} 
+ \left\langle e_j, \nabla^{TM}_{e_i}e_i\right\rangle.
\end{align}
 Set
\begin{align}\label{4.115}
- \Delta_T^M = \sum_{j=1}^{n} \Big( (\nabla_{e_j})^*+ T
(\text{sgn}\, a_j)  t_j\left(x',y'\right)\Big)
\left(\nabla^{\Lambda^{0,\bullet}}_{e_j} + T (\text{sgn}\, a_j)
t_j\left(x',y'\right)\right).
\end{align}
Clearly, $- \Delta_T^M$ is nonnegative near $z=(x_0,y_0)$. 
We verify using \eqref{4.110a} that
\begin{align}\label{4.116}
- \Delta_T^M 
=- \Delta^M - T  \sum_{j=1}^{n} |a_j| + T^2
\sum_{j=1}^{n} t_j \left(x',y'\right)^2 +
\mO\left(1 +T\left|x'\right|+T \left|y'\right|\right).
\end{align}
\comment{
where by  $\mO(\partial_{x'} +1 +T\left|x'\right|+T \left|y'\right|)$
we mean a first order differential operator of the form
\begin{multline}\label{4.117}
  \sum_{j=1}^{n} b_j\left(x',y'\right) \frac{\partial}{\partial x_j}
  + d\left(x',y'\right)\\
\text{with} \, \, b_j\left(x',y'\right)=\mO(1),\quad
d\left(x',y'\right)=  \mO\left(1 +T\left|x'\right|+T \left|y'\right|\right).
\end{multline}
We will also use similar notation for other operators.
}


By 
\eqref{4.112}, \eqref{4.115} and \eqref{4.116},
the following identity holds for any $k>1$, when  both sides act
on sections with compact support in $B^M_{\var_0} (x_0)\times B^N_{\var_0} (y_0)$,
\begin{multline}\label{4.118}
    \frac{1}{2}\left(F_{\mM,T} + F_{\mM,T}^* \right)
    \geqslant -\Delta^{N} - \Delta_T^M + \pi T A
+ \mO\left(1+ T\left|x'\right|+T \left|y'\right|\right)\\
 \geqslant  -\frac{1}{k} \Delta^{N} - \frac{1}{k} \Delta^M -
\frac{T}{k} \sum_{j=1}^{n} |a_j| + \pi\, T A
+ \mO\left(1+ T\left|x'\right|+T \left|y'\right|\right).
\end{multline}
\comment{
It is standard to obtain by (\ref{4.106}),
that there exist $C_1,C_2,C_3>0$ such that
for any $s\in \Omega^{0,\bullet}(\mM,L\otimes F)$ with
$\supp(s)\subset B^M_{\var_0} (x_0)\times B^N_{\var_0} (y_0)$, 
\begin{align}\label{4.119}\begin{split}
- \left\langle \Delta^{M\times N} s,s\right\rangle
& \geqslant C_1 \left\|D^{L\otimes F}s\right\|^2_0  - C_2\|s\|_0^2,\\
\left|\left\langle \mO\left(1+ T\left|x'\right|+T \left|y'\right|\right)
s,s\right\rangle\right| &\leqslant C_3 (1+ T\var) \|s\|_0^2.
\end{split}\end{align}

Moreover, by the standard elliptic estimate and the Cauchy inequality,
there exist $C_4,C_5>0$ such that
\begin{align}\label{4.120}
\left|\left\langle \mO\left(\partial_{x'}\right)
s,s\right\rangle\right| \leqslant C_4 \var \left\|D^{L\otimes
F}s\right\| ^2_0 + \frac{C_5}{\var} \|s\|_0^2.
\end{align}

From 
 \eqref{4.118} and \eqref{4.119}, we get
\begin{multline}\label{4.121}
\text {Re} \left(\left\langle F^{\mM}_T s,s\right\rangle\right)
\geqslant \left( \frac{C_1}{k}- C_4 \var\right)
\left\|D^{L\otimes F} s\right\|_0^2\\
+ T\left(\pi\, A- C_3\var
- \frac{1}{k} \sum_{j=1}^{n} |a_j| \right)\|s\|_0^2
- \Big \left( \frac{C_5}{\var} + \frac{C_2}{k}+C_3\right) \|s\|_0^2.
\end{multline}

Now, we first take $k$ large enough so that $A- \frac{1}{k}
\sum_{j=1}^{n} |a_j|> A/2$, and then   choose $\var$ small
enough so that
\begin{align}\label{4.123}
\frac{C_1}{k}- C_4 \var >0, \quad  \frac{A}{ 2} - C_3\var>0.
\end{align}
}
By  (\ref{4.106}) and \eqref{4.118}, 
there exist  $C_2,C_3>0$ such that
for any $0<\var<{\var_0}$, $s\in \Omega^{0,\bullet}(\mM,L\otimes F)$ with
$\supp(s)\subset B^M_\var (x_0)\times B^N_\var (y_0)$, we have
\begin{multline}\label{4.121}
\text {Re} \Big\langle F_{\mM,T} s,s\Big\rangle
\geqslant \frac{1}{k}  \left\|D^{L\otimes F} s\right\|_0^2
+ \Big[T\Big(\pi  A- \frac{1}{k} \sum_{j=1}^{n} |a_j| - C_3\var \Big)
- \Big( \frac{C_2}{k}+C_3\Big) \Big]\|s\|_0^2.
\end{multline}
We take $k$ large enough 
and choose $\var$ small enough so that
\begin{align}\label{4.123}
\frac{A}{2}- \frac{1}{k}\sum_{j=1}^{n} |a_j|>  0,
\quad  \frac{A}{2} - C_3\var>0.
\end{align}
With $\var$ chosen as in \eqref{4.123},  the conclusion of
 Proposition  \ref{t33.5} follows from  \eqref{4.121} in
the case where $\psi_{A}^\mM(z)=0$.
The proof of Proposition  \ref{t33.5} is completed.
\end{proof}

By 
Proposition \ref{t33.5} 
and the gluing trick due to Bismut-Lebeau \cite[pp. 115-117]{BL91}
(which has been used in the proof of (\ref{1.316})), we obtain the
following:  for any open subset $\mU'\subset \mM$ with
$\overline{\mU'}\cap
\partial\mM=\emptyset$, there exist $C_6>0,\, b_1>0$ such that for
any 
$s\in \Omega^{0,\bullet}(\mM, L\otimes F)$
with $\supp (s)\subset \mU'$, we have
\begin{align}\label{33.17}
\text {Re}\, \Big\langle F_{\mM,T} s,s\Big\rangle
\geqslant C_6\left(\left\|D^{L\otimes
F}s\right\|_0^2+\left(T-b_1\right)\|s\|_0^2\right) .
\end{align}

Let $\mU$ be a $G$-invariant open neighborhood of $\partial\mM$ 
in $\mM$ such that $\psi_{A}^\mM$ does not vanish on $\ov{\mU}$. 
As $\psi_{A}^\mM$ does not vanish on $\partial\mM$, the existence of 
${\mU}$ is clear.
 Then one can  proceed in exactly the same way as in the proof 
of (\ref{3.859}) (or \cite[Proposition 2.4]{TZ99}), to see that
there exist  $T_{2}>0$, $C_7>0$  such that for any
$T> T_2$ and 
$s\in \Omega^{0,\bullet}(\mM, L\otimes F)^{\gamma=0}$ with ${\rm
supp}(s)\subset \mU$ and $P_{\geqslant 0, \pm, T}(s|_{\partial \mM})=0$,
  we have
\begin{align}\label{3.112}
\left\|D_{\mM,T}s\right\|^2_0\geqslant C_7\left(\left\|D^{L\otimes F}
s\right\|^2_0+ T^2  \|s\|^2_0\right).
 \end{align}


In view of (\ref{33.12}), (\ref{33.17}) and (\ref{3.112}), one
 can then proceed as in the proof of (\ref{1.3168}), which goes back to
 \cite[pp. 115-117]{BL91}, to see that Theorem \ref{ta2.3} holds.

\subsection{Proof of Lemma \ref{t33.6} (I): 
uniform estimates on functions}\label{s33.4}


We give first uniform estimates for some functions appeared 
 in the definition of $\gamma_A$, $\psi_A$ when $A\to\infty$.

 Recall that $A_2>0$ was determined in Lemma \ref{t33.4}.
Let  $A> A_2$ be a regular value for $|\mu|^2$ and
$\frac{1}{2}|\theta|^2$. 
Set
\begin{align}\label{aa4.26}
\begin{split}
\tau_{A1} =& 1+ \alpha_{A}'(|\mu|^2) (|\theta|^2-|\mu|^2),\\
\tau_{A2} =& 
1- 2 \phi_{A}'(\beta_{A})\left\langle\rho_{A},\eta\right\rangle \tau_{A1},\\
\tau_{A4}= &1-\phi_{A}(\beta_{A}) - 2 \phi_{A}'(\beta_{A})\alpha_{A}(|\mu|^2)
\left\langle\rho_{A},\eta\right\rangle ,\\
\tau_{A5}=& \Big[1-\phi_{A}(\beta_{A})\Big]\tau_{A1} - \alpha_{A}(|\mu|^2)\\
=&1-\phi_{A}(\beta_{A})- \alpha_{A}(|\mu|^2) 
+\Big[1-\phi_{A}(\beta_{A})\Big] \alpha_{A}'(|\mu|^2)(|\theta|^2-|\mu|^2).
\end{split}\end{align}
Then 
\begin{align}\label{aa4.26c}
\tau_{A5}=\tau_{A1}\, \tau_{A4} -\alpha_{A}(|\mu|^2) \, \tau_{A2}\, .
\end{align}


From \eqref{33.6}, \eqref{aa4.26}, we obtain
\begin{align}\label{aa4.27}
\gamma_{A}  =2\tau_{A1} \mu  + 2\alpha_{A}(|\mu|^2)\eta \, .
\end{align}
From (\ref{2.3}), (\ref{33.5}), \eqref{33.7}, \eqref{aa4.26} and
\eqref{aa4.27}, we get
\begin{align}\label{aa4.48}\begin{split}
\psi_{A} &= \mu  +\Big[1- \phi_{A}(\beta_{A})\Big] \eta
- \phi_{A}'(\beta_{A})\left\langle\rho_{A},\eta\right\rangle
\Big[ 2\tau_{A1}\mu  + 2\alpha_{A}\left(|\mu|^2\right) \eta \Big]\\
&=\tau_{A2}\, \mu  + \tau_{A4}\,  \eta \, .
\end{split}\end{align}

In the following, for $s\in \R$ and a function $f_A$ on $\mM$, 
we write $f_A= \mO_0(A^{s})$ if there exists $C>0$ 
(independent on $A$)
such that its $\cC^0$-norm on $\mM$ can be controlled by $CA^{s}$.

The following lemma contains basic asymptotic estimates for these
$\tau$ functions.

\begin{lemma}  \label{ta5.10}  There exists $A_6\geqslant A_2$ such that for
  $A> A_6$, we have
\begin{align}\label{a.100}
 A<\beta_{A}<2A,\ \ \ {\rm on}\ \  \mM\setminus \partial  \mM.
\end{align}
  Thus
\begin{align}\label{a.101}
 0<\phi_{A}(\beta_{A})<1\ \ \ {\rm on}\ \  \mM\setminus \partial \mM.
\end{align}
 Moreover, 
\begin{subequations}
\begin{align}
\label{aa4.63}
\tau_{A1} &= 1+ \mO_0\left(A^{-1/2}\right),\ \ \quad
\tau_{A2} = 1+ \mO_0\left(A^{-1/2}\right),\\
\label{aa4.64}
\tau_{A4} &= \Big[1-\phi_{A}(\beta_{A})\Big]
\Big(1+\mO_0\left(A^{-1/2}\right)\Big),\\
  \label{aa4.65}
\tau_{A5} &=  
\Big[1-\phi_{A}(\beta_{A})-\alpha_{A}\left(|\mu|^2\right) \Big]
\Big(1+ \mO_0(A^{-1/2})\Big).
\end{align}
\end{subequations}
Finally, for any $A> A_6$, we have
\begin{align}\label{aa4.77}\begin{split}
    1-\phi_{A}(\beta_{A})-\alpha_{A}\left(|\mu|^2\right) &<0  \quad
\text{ if } (x,y)\in \mM\setminus \partial \mM,\\
&=0  \quad  \text{ if } (x,y)\in \partial \mM.
\end{split}\end{align}
\end{lemma}
\begin{proof}
From  (\ref{2.3}), (\ref{2.12}) and (\ref{33.0}),
for $A> A_2\geqslant A_{0}$, we have on $\mM$,
\begin{align}\label{a.103}\begin{split}
& A^{1/2}\leqslant |\mu|\leqslant |\theta|+|\eta|
\leqslant  \sqrt{2}A^{1/2}+\left|C_0\right|
\leqslant \left( 2\sqrt{2}-\sqrt{5/3}\right)A^{1/2},\\
&|\theta|^2-|\mu|^2  =2\, \langle \mu,\eta\rangle+|\eta|^2=\mO_0(A^{1/2}).
\end{split}\end{align}
From  
(\ref{a.01})-(\ref{a5.5}) and (\ref{a.103}), for $A> A_2$, we have on $\mM$,
 \begin{align}\label{aa5.66}
    \begin{split}
 \alpha_{A}\left(|\mu|^2\right)=\mO_0(1),\quad
   \beta_{A}=|\mu|^2 + \mO_0(A^{1/2}),\quad  
\alpha_{A}'\left(|\mu|^2\right)=\phi_{A}'(\beta_{A}) =\mO_0(A^{-1}) .
 \end{split}\end{align}

If $|\mu|^2\leqslant \frac{4A}{3} $, then  (\ref{a.01}),  (\ref{a.02}),
 (\ref{a5.5}) and (\ref{a.103}) yield
\begin{align}\label{aa4.66}
    \begin{split}
\frac{\beta_{A}}{A}-1 &= \left(\frac{|\mu|^2}{A}-1\right) \left[ 1+
\frac{1}{A}\left(\frac{|\mu|^2}{A}-1\right)
\left(|\theta|^2-|\mu|^2\right) \right]\\
&=  \left(\frac{|\mu|^2}{A}-1\right)\Big(1+
\mO_0\left(A^{-1/2}\right)\Big).
\end{split}\end{align}

If $|\mu|^2\geqslant  \frac{4A}{3} $, then by (\ref{aa5.66}),
we have for $A> A_2$ large enough,
\begin{align}\label{aa4.67}
\beta_{A}\geqslant \frac{4}{3} A+  \mO_0(A^{1/2})>\frac{6A}{5}.
\end{align}
By (\ref{2.12}), (\ref{aa4.66}) and (\ref{aa4.67}),
 we have $\beta_{A}>A$ on $\mM\setminus \partial\mM$
 for $A> A_2$ large enough.

On the other hand, if  $|\mu|^2\leqslant \frac{5A}{3}$, then by
(\ref{aa5.66}),   for $A> A_2$  large enough,
$\beta_{A}<2A$.


By (\ref{a.01}), (\ref{a.02}) and (\ref{a5.5}), 
if $|\mu|^2\geqslant \frac{5A}{3}$,
\begin{align}\label{33.21}
\alpha_{A}(|\mu|^2)=1,\quad  \alpha_{A}'(|\mu|^2) =0,
\quad   \beta_{A}=|\theta|^2.
\end{align}
Combining with (\ref{2.12})  
we have $\beta_{A}<2A$ on
$\mM\setminus \partial\mM$ for $A> A_2$ large enough.
Thus there exists $A_7\geqslant A_2$ such that (\ref{a.100}) holds
for $A> A_7$.
Note that $\wi{\phi}(0)=1, \wi{\phi}(1)=0$ and $ \wi{\phi}'< 0$ 
on $(0,1]$.
Thus (\ref{a.02}) and (\ref{a.100}) imply
(\ref{a.101}).

The first identity in (\ref{aa4.63}) follows immediately from
(\ref{aa4.26}), (\ref{a.103}) and (\ref{aa5.66}).

From 
(\ref{33.5}), (\ref{a.101}) and (\ref{a.103}), 
we obtain for $A> A_7$,
\begin{align}\label{aa5.67}
     |\rho_{A}|\leqslant |\theta|+ |\eta|< 2A^{1/2}  \quad \text{ on } \mM.
 \end{align}
From (\ref{33.0}),  
\eqref{aa4.26},  the first identity  in \eqref{aa4.63}, \eqref{aa5.66} and 
  \eqref{aa5.67},
we get the second identity in \eqref{aa4.63}.
Hence the proof of \eqref{aa4.63} is completed.

We prove now \eqref{aa4.64}.
If $|\mu|^2\leqslant \frac{4A}{3} $, then by 
(\ref{aa5.66}), we have $\beta_{A}<\frac{15A}{11}$ for $A> A_7$
large enough. Then  (\ref{a.01}), (\ref{a.02}) and (\ref{aa4.66}) imply
  \begin{align}\label{aa4.68}  \begin{split}
\alpha_{A}\left(|\mu|^2\right)&  =\left(\frac{|\mu|^2}{A}-1\right)^2,\\
 1-\phi_{A}(\beta_{A}) &= \Big(\frac{\beta_{A}}{A}-1\Big)^3
= \left(\frac{|\mu|^2}{A}-1\right)^{3}\Big(1+
 \mO_0\left(A^{-1/2}\right)\Big) , \\
 \phi_{A}'(\beta_{A}) &=  - \frac{3}{A} \Big(\frac{\beta_{A}}{A}-1\Big)^2 .
   \end{split}\end{align}
  From (\ref{33.0}), (\ref{aa4.26}), (\ref{aa4.66}),
(\ref{aa5.67}) and \eqref{aa4.68}, we deduce that
\begin{align}\label{aa4.69}
 \begin{split}
 \tau_{A4} 
  =& \Big(1-\phi_{A}(\beta_{A})\Big)
\left[1+\frac{6}{A}\cdot \Big(\frac{|\mu|^2}{ A}-1\Big)
  \Big(1+\mO_0\left(A^{-1/2}\right)\Big)
\langle\rho_{A},\eta\rangle\right]\\
=& \Big(1-\phi_{A}(\beta_{A})\Big)\Big(1+\mO_0\left(A^{-1/2}\right)\Big).
\end{split}\end{align}

If $|\mu|^2\geqslant \frac{4A}{3}$,  
by (\ref{a.01}), (\ref{a.02}) and (\ref{aa4.67}), we have
$1-\phi_{A}(\beta_{A})\geqslant 1-\phi_{A}(\frac{6A}{5})= 5^{-3}>0$, from
which \eqref{aa4.64} holds, since in view of
(\ref{33.0}), (\ref{aa5.66}) and  (\ref{aa5.67}),
$\phi_{A}'(\beta_{A})\alpha_{A}\left(|\mu|^2\right)$ 
$\langle\rho_{A},\eta\rangle=\mO_0(A^{-1/2})$ holds.
Together with (\ref{aa4.69}), this implies  \eqref{aa4.64}.

For the proof of (\ref{aa4.65}) and (\ref{aa4.77}), we first
consider the region $|\mu|^2\geqslant \frac{5A}{3}$ in $\mM$.
By \eqref{aa4.26} and \eqref{33.21}, we get
\begin{align}\label{5.40}
\tau_{A5} = 1 - \phi_{A}(\beta_{A}) - \alpha_{A}(|\mu|^2) 
= - \phi_{A}(\beta_{A}) .
\end{align}
Thus (\ref{aa4.65}) holds. 
From (\ref{33.9c}), (\ref{a.101}) and (\ref{5.40}), we get (\ref{aa4.77}).

By (\ref{a.01}),  (\ref{a.02x})  and (\ref{aa5.66}),
we find that for $A> A_7$,
\begin{align}\label{a.22}
 \phi_{A}(\beta_{A})=\phi_{A}\left(|\mu|^2\right)
+ \mO_0\left(A^{-1/2}\right)  \quad \mbox{  on  } \mM.
\end{align}
\comment{
Let $0<\epsilon_0<\frac{1}{ 9}$ be such that
\begin{align}\label{a24}\wi{\alpha}(t)+&\wi{\phi_{A}}(t) 
\geqslant \frac{28.5}{27},
 \quad\quad\text{ for } \,\, \, \frac{1-2\epsilon_0}{3}\leqslant
t\leqslant \frac{2}{3}.
\end{align}
Such an  $\epsilon_0$ clearly exists in view of  (\ref{a.01}).}
If $\frac{4A}{3}\leqslant|\mu|^2 \leqslant \frac{5A}{3}$,
then from \eqref{a.01} and \eqref{a.22}, 
we have for $A$ large enough,
\begin{align}\label{a25}
1-\phi_{A}(\beta_{A})-\alpha_{A}\left(|\mu|^2\right)\leqslant -\frac{1}{27}.
\end{align}
By  \eqref{aa4.26}, \eqref{a.103},   (\ref{aa5.66}) and  \eqref{a25}, 
we get  \eqref{aa4.65} and (\ref{aa4.77}).

Finally, if $|\mu|^2\leqslant \frac{4A}{3}$, 
by (\ref{aa4.26}), the first equation of (\ref{aa4.63})
and (\ref{aa4.68}), the following identities hold 
for $A> A_7$ large enough:
\begin{align}\label{a23} \begin{split}
1-&\phi_{A}(\beta_{A})-\alpha_{A}(|\mu|^2) \\
&= - \Big(\frac{|\mu|^2}{A}-1\Big)^2
\Big[ 1-  \Big(\frac{|\mu|^2}{A}-1\Big)\Big(1+ \mO_0(A^{-1/2})\Big)\Big],\\
\tau_{A5}&=\Big[1-\phi_{A}(\beta_{A})\Big]
\Big(1+\mO_0\left(A^{-1/2}\right) \Big) -\alpha_{A}\left(|\mu|^2\right)\\
&=\Big[1-\phi_{A}(\beta_{A})-\alpha_{A}\left(|\mu|^2\right)\Big]
\Big(1+\mO_0\left(A^{-1/2}\right)\Big).
\end{split}\end{align}
From \eqref{33.8c} and the first identity
in (\ref{a23}), we get  \eqref{aa4.77} in this case.

Combining the three cases discussed above, we conclude that
there exists $A_6\geqslant A_7$ such that 
\eqref{aa4.65} and \eqref{aa4.77} hold for $A> A_6$.
The proof of Lemma \ref{ta5.10} is   completed.
\end{proof}

The following Lemma will also be used in the proof of Lemma \ref{t33.6}.
\begin{lemma}\label{t5.6}
There exists $A_8\geqslant A_6$ such that for any $A> A_8$,
\begin{align}\label{b4.96}
1< \frac{\Big(1-\phi_{A}(\beta_{A})\Big)^2 -\alpha_{A}(|\mu|^2)}
{1-\phi_{A}(\beta_{A})-\alpha_{A}(|\mu|^2)} < 12 \quad \text{ on } \mM\setminus
\partial \mM.
\end{align}
\end{lemma}
\begin{proof} By  \eqref{a.101} and \eqref{aa4.77}, we have
\begin{align}\label{b4.97}
\Big(1-\phi_{A}(\beta_{A})\Big)^2 -\alpha_{A}(|\mu|^2)
< 1-\phi_{A}(\beta_{A})-\alpha_{A}(|\mu|^2)<0
\text{ on } \mM\setminus \partial \mM.
\end{align}

To complete the proof of \eqref{b4.96},  we have to show that
\begin{align}\label{aa4.97}
    11- 10 \phi_{A}(\beta_{A})  - 11\alpha_{A}(|\mu|^2)
    - \phi_{A}(\beta_{A})^2<0
       \quad  \text{on }   \mM\setminus \partial \mM.
\end{align}
We examine three cases.
First, if $|\mu|^2\geqslant \frac{5A}{3}$,
then  \eqref{aa4.97} follows from \eqref{a.101} and \eqref{33.21}.
Secondly,  
if $|\mu|^2\leqslant \frac{4A}{3}$,
 then by 
 \eqref{aa4.68},  we get 
\begin{multline}\label{aa4.98}
11- 10 \phi_{A}(\beta_{A}) - 11\alpha_{A}\left(|\mu|^2\right)
- \phi_{A}(\beta_{A})^2
\leqslant - 11  \alpha_{A}\left(|\mu|^2\right)
+ 12 \Big(1-\phi_{A}(\beta_{A})\Big) \\
\leqslant \Big(\frac{|\mu|^2}{A}-1\Big)^2 \left[-7
+\mO_0\left(A^{-1/2}\right)\right].
\end{multline}
By 
\eqref{aa4.98}, we see that \eqref{aa4.97} holds for $A$ large enough.

Thirdly, let $\frac{4A}{3}\leqslant |\mu|^2\leqslant \frac{5A}{3}$,
from \eqref{a25} for $A>0$ large enough, we have
\begin{multline}\label{aa4.99}
    11- 10 \phi_{A}(\beta_{A})
       - 11\alpha_{A}(|\mu|^2)- \phi_{A}(\beta_{A})^2
     \leqslant - \frac{11}{27} + \phi_{A}(\beta_{A})- \phi_{A}(\beta_{A})^2 
 \leqslant  - \frac{17}{108}.
 \end{multline}
This completes the proof of Lemma \ref{t5.6}.
\end{proof}

By (\ref{aa4.63}), we may and we will assume that $A$ is large
enough so that $\tau_{A2}>1/2.$
Set
 \begin{align}\label{aa4.28}
\begin{split}
\tau_{A6}=&  - 2  \phi_{A}'(\beta_{A})\alpha_{A}''(|\mu|^2)
\langle\rho_{A},\eta\rangle
\Big(|\theta|^2  -|\mu|^2\Big) \left(\frac{\tau_{A4}}{\tau_{A2}}\right)^2\\
 &\hspace{10mm}+ 4 \phi_{A}'(\beta_{A})\alpha_{A}'(|\mu|^2)  
 \langle\rho_{A},\eta\rangle \frac{\tau_{A4}}{\tau_{A2}}\\
& + 2 \Big[- \phi_{A}''(\beta_{A})\langle\rho_{A},\eta\rangle 
+ (\phi_{A}'(\beta_{A}))^2
|\eta|^2 \Big] \left(\frac{\tau_{A5}}{\tau_{A2}}\right)^2
 +  2 \phi_{A}'(\beta_{A})\frac{\tau_{A5}}{\tau_{A2}},\\
\tau_{A7}=&   2\phi_{A}'(\beta_{A}) \alpha_{A}'(|\mu|^2)
\langle\rho_{A},\eta\rangle \frac{(\tau_{A2}-\tau_{A4})\tau_{A4}}
{\tau_{A2}^2} \\
& - 2\Big[ - \phi_{A}''(\beta_{A})
\langle\rho_{A},\eta\rangle +
 (\phi_{A}'(\beta_{A}) )^2 |\eta|^2 \Big]
\alpha_{A}(|\mu|^2)\frac{(\tau_{A2} -\tau_{A4})\tau_{A5}}{\tau_{A2}^2}\\
&+\phi_{A}'(\beta_{A})\left[ \Big(\frac{\tau_{A2} - \tau_{A4}}{\tau_{A2}} +
1-2\phi_{A}(\beta_{A})\Big)\frac{\tau_{A5}}{\tau_{A2}} - \alpha_{A}(|\mu|^2)
\frac{\tau_{A2}-\tau_{A4}}{\tau_{A2}}\right] .
\end{split}\end{align}

\begin{lemma}  \label{ta5.11}  For $A>0$  large enough, the
following identities hold on $\mM$:
\begin{align}\label{aa4.72} \begin{split}
\tau_{A6}=& 2 \phi_{A}'(\beta_{A})
\Big[1-\phi_{A}(\beta_{A})-\alpha_{A}\left(|\mu|^2\right)\Big]
\Big(1+ \mO_0\left(A^{-1/2}\right)\Big),\\
\tau_{A7}=& \phi_{A}'(\beta_{A}) \Big[\Big(1-\phi_{A}(\beta_{A}) \Big)^2
-\alpha_{A}\left(|\mu|^2\right)\Big]
\Big(1+ \mO_0\left(A^{-1/2}\right)\Big).
\end{split}\end{align}
In particular,
\begin{align}\label{aa4.73}\begin{split}
&\tau_{A6}>0, \quad \tau_{A7}>0 \quad \text{ if } (x,y)\in \mM\setminus \partial\mM,\\
&\tau_{A6}=0, \quad \tau_{A7}=0 \quad  \text{ if } (x,y)\in \partial\mM
=\mM_{1}\cup\mM_{2},
\end{split}\end{align}
and 
\begin{align}\label{aa4.73c}
\tau_{A7}\leqslant 6\, \tau_{A6} \left[1+ \mO_0\left(A^{-1/2}\right)\right].
\end{align}
\end{lemma}
\begin{proof} 
Note that from  \eqref{33.0}, \eqref{a.01}, \eqref{a.02},
 \eqref{aa5.66} and \eqref{aa5.67},   on $\mM$, we have
 \begin{align}\label{aa4.74c} \begin{split}
&  \alpha_{A}'\left(|\mu|^2\right)  \langle\rho_{A},\eta\rangle 
  = \mO_{0}(A^{-1/2}),\quad 
 \alpha_{A}''\left(|\mu|^2\right)  \langle\rho_{A},\eta\rangle 
  = \mO_{0}(A^{-3/2}),\\
&     - \phi_{A}''(\beta_{A})
 \langle\rho_{A},\eta\rangle +
  (\phi_{A}'(\beta_{A}) )^2 |\eta|^2 =\mO_{0}(A^{-3/2}).
 \end{split}\end{align}    
Recall that $\wi{\phi}'<0$ on $(0,1]$.
By \eqref{a.02x}, \eqref{a.100}
and the second equation of \eqref{aa5.66}, 
there exist $C>0$, $A_{10}>0$ such that for $A>A_{10}$, 
\begin{align}\label{5.70c} \begin{split}
&\phi_{A}'(\beta_{A}) <0 \quad \text{ on } \mM\setminus \partial \mM,\\
&\left|\phi_{A}'(\beta_{A}) \right|\geqslant \frac{C}{A} \quad \text{ if }
|\mu|^2\geqslant \frac{4A}{3}.
 \end{split}\end{align}    

By Lemma \ref{ta5.10}, \eqref{a.103},  
 (\ref{aa4.28}) and (\ref{aa4.74c}), we get
\begin{multline}\label{aa4.74}
\tau_{A6}=  \phi_{A}'(\beta_{A})\mO_0\left(A^{-1}\right) 
\Big[1-\phi_{A}(\beta_{A})\Big]^2
+ \phi_{A}'(\beta_{A})\mO_0\left(A^{-1/2}\right) 
\Big[1-\phi_{A}(\beta_{A})\Big]\\
+ \mO_{0}(A^{-3/2})
\Big[1-\phi_{A}(\beta_{A})-\alpha_{A}\left(|\mu|^2\right)\Big]^2 \\
+ 2\phi_{A}'(\beta_{A})\Big[1-\phi_{A}(\beta_{A})
-\alpha_{A}\left(|\mu|^2\right)\Big]
\Big(1+ \mO_0\left(A^{-1/2}\right)\Big).
\end{multline}
By \eqref{aa4.66}, \eqref{aa4.68} and the first equation of \eqref{a23},  
there exists $C>0$ such that for  $A> 0$ large enough, 
if $|\mu|^2\leqslant \frac{4A}{3}$, then
\begin{align}\label{5.71}\begin{split}
&0\leqslant 1-\phi_{A}(\beta_{A})
\leqslant C \left|1-\phi_{A}(\beta_{A})-\alpha_{A}\left(|\mu|^2\right)\right|,\\
& \left|1-\phi_{A}(\beta_{A})-\alpha_{A}\left(|\mu|^2\right)\right|\leqslant C
\left|A\phi_{A}'(\beta_{A})\right| .
\end{split}\end{align}
Due to \eqref{a.101}, \eqref{aa5.66}, (\ref{a25}) and \eqref{5.70c}, 
if $\frac{4A}{3}\leqslant |\mu|^2\leqslant \frac{5A}{3}$,
(\ref{5.71}) still  holds  for some constant $C>0$.
By \eqref{5.71}, the first three terms in \eqref{aa4.74} 
can be controlled by $\Big|\phi_{A}'(\beta_{A})$ 
$\Big[1-\phi_{A}(\beta_{A})-\alpha_{A}\left(|\mu|^2\right)\Big]\Big| 
\mO_0\left(A^{-1/2}\right)$
if $|\mu|^2\leqslant \frac{5A}{3}$. 
Thus from \eqref{aa4.74},
the first identity  in \eqref{aa4.72} holds when
$|\mu|^2\leqslant \frac{5A}{3}$.

For $|\mu|^2\geqslant \frac{5A}{3}$, by 
(\ref{33.21}), $\alpha_{A}''(|\mu|^2)=\alpha_{A}'(|\mu|^2)=0$,
thus the first two terms of $\tau_{A6}$ are zero. 
By (\ref{aa4.63})-(\ref{aa4.65}), 
\eqref{5.40}, \eqref{aa4.28} and the third equation in \eqref{aa4.74c}, 
we have
\begin{align}\label{aa4.78} \begin{split}
\tau_{A6}&=\mO_0\left(A^{-3/2}\right)\phi_{A}(\beta_{A})^2
- 2\phi_{A}'(\beta_{A}) \phi_{A}(\beta_{A}) 
\Big(1+ \mO_0\left(A^{-1/2}\right)\Big).
\end{split}\end{align}
From \eqref{a.101}, \eqref{5.40}, (\ref{5.70c}) and \eqref{aa4.78},
  the first identity in \eqref{aa4.72} holds 
when $|\mu|^2\geqslant \frac{5A}{3}$.

From \eqref{aa4.77}, the first identity in \eqref{aa4.72} and (\ref{5.70c}),  
we get \eqref{aa4.73} for  $\tau_{A6}$.

For the second identity in \eqref{aa4.72}, 
by Lemma \ref{ta5.10} 
and \eqref{aa4.74c},  
we obtain the asymptotics of the terms of $\tau_{A7}$ in \eqref{aa4.28}
in order as follows :
\begin{multline}\label{aa4.80}
\tau_{A7}=\phi_{A}'(\beta_{A}) \Big[1-\phi_{A}(\beta_{A}) \Big]
\mO_0\left(A^{-1/2}\right) \\
+\alpha_{A}\left(|\mu|^2\right) \Big[1-\phi_{A}(\beta_{A})-
\alpha_{A}\left(|\mu|^2\right)  \Big] \mO_0(A^{-3/2}) \\
+ \phi_{A}'(\beta_{A}) \Big\{\Big(1-\phi_{A}(\beta_{A})
+ \mO_0\left(A^{-1/2}\right)\Big)
\Big[1-\phi_{A}(\beta_{A})-\alpha_{A}\left(|\mu|^2\right)  \Big]\\  -
\alpha_{A}\left(|\mu|^2\right) \Big(\phi_{A}(\beta_{A}) +
\mO_0\left(A^{-1/2}\right) \Big)\Big\},
\end{multline}
here the factor $1-\phi_{A}(\beta_{A})$ in the first term of 
\eqref{aa4.80} is from $\tau_{A4}$ and the factor   
$1-\phi_{A}(\beta_{A})-\alpha_{A}\left(|\mu|^2\right) $ is from $\tau_{A5}$.

If $|\mu|^2\leqslant \frac{5A}{3}$, by \eqref{b4.97}, 
the first equation of \eqref{5.71}  (which holds for
$|\mu|^2\leqslant \frac{5A}{3}$ as explained after  \eqref{5.71}),  
we get for $A>0$ large enough,
\begin{align}\label{33.44}
\left|\alpha_{A}\left(|\mu|^2\right)\right|\leqslant (C+1)
\left|(1-\phi_{A}(\beta_{A}))^2-\alpha_{A}\left(|\mu|^2\right)\right|.
\end{align}
Thus by  \eqref{b4.97}, \eqref{5.71} for $|\mu|^2\leqslant \frac{5A}{3}$
and  \eqref{33.44},
the first two terms of  \eqref{aa4.80}  is bounded by 
$\Big|\phi_{A}'(\beta_{A}) 
\Big[ (1-\phi_{A}(\beta_{A}) )^{2}$ 
$-\alpha_{A}\left(|\mu|^2\right)\Big]\Big| \mO_0\left(A^{-1/2}\right)$.
From  \eqref{b4.97},  \eqref{aa4.80}  and \eqref{33.44},
the second identity in \eqref{aa4.72} holds for
$|\mu|^2\leqslant \frac{5A}{3}$.

If $|\mu|^2\geqslant \frac{5A}{3}$, then by \eqref{aa4.26},
\eqref{33.21} and \eqref{5.40}, we have 
\begin{align}\label{5.72}
    \tau_{A1}=1, \quad \tau_{A2}-\tau_{A4}= \phi_{A}(\beta_{A}),
    \quad    \tau_{A5} = -\phi_{A}(\beta_{A}) .
\end{align}
By 
(\ref{aa4.63}), \eqref{33.21}, 
\eqref{aa4.28}, \eqref{aa4.74c}  and \eqref{5.72}, we get
the first term of $\tau_{A7}$ is zero and
\begin{multline}\label{aa4.81}
\tau_{A7}= \phi_{A}(\beta_{A}) ^2 \mO_0\left(A^{-3/2}\right)
+  \phi_{A}'(\beta_{A}) \Big\{ -\phi_{A}(\beta_{A}) ^2
\Big(1+ \mO_0\left(A^{-1/2}\right)\Big)\\
- \Big[1-2\phi_{A}(\beta_{A})\Big]\phi_{A}(\beta_{A}) \Big(1+
\mO_0\left(A^{-1/2}\right)\Big)- \phi_{A}(\beta_{A})\Big(1+
\mO_0\left(A^{-1/2}\right)\Big)\Big\}.
\end{multline}
From \eqref{a.101}, (\ref{5.70c}) and \eqref{aa4.81}, we get that if
$|\mu|^2\geqslant \frac{5A}{3}$,
\begin{align}\label{5.73}
\tau_{A7}=- \phi_{A}'(\beta_{A}) \phi_{A}(\beta_{A}) 
\Big[2-\phi_{A}(\beta_{A})\Big] \Big(1+
\mO_0\left(A^{-1/2}\right)\Big).
\end{align}

Now \eqref{33.21} and \eqref{5.73} imply
 the second identity in \eqref{aa4.72} for $|\mu|^2\geqslant \frac{5A}{3}$.
 By  \eqref{aa4.77}, (\ref{b4.96}) and \eqref{5.70c}, we get  
 \eqref{aa4.73} for $\tau_{A7}$.
 From Lemma \ref{t5.6}, \eqref{aa4.72} and  \eqref{aa4.73}, 
we get \eqref{aa4.73c}. This concludes the proof of Lemma \ref{ta5.11}.
\end{proof}

\subsection{Proof of Lemma \ref{t33.6} (II):
evaluation of $I_{A\cdot}$ over  zero($\psi_{A}^\mM$)}\label{s33.5}

In this subsection, we evaluate the 
terms $I_{A\cdot}$ in \eqref{a.3} on zero($\psi_{A}^\mM$), 
the zero set of $\psi_{A}^\mM$.
The main point is that we use $\eta^{N}$ (resp. $\eta^{M}$)  to replace  
$\mu^{N}$, $\theta^{N}$, $\gamma_{A}^{N}$ 
(resp. $\mu^{M}$, $\gamma_{A}^{M}$) 
which are difficult to control over $\mM$.


\begin{lemma}\label{ta5.6}
    On $\{z\in \mM: \psi_{A}^\mM(z)=0\}$, the following identities hold:
\begin{align}\label{aa4.49}  \begin{split}
   \tau_{A2}\,  \mu^M =& -  {\tau_{A4}}{} \,   \eta^M,
    \quad  \tau_{A2}\,  \gamma_{A}^M =-  {2\tau_{A5}}{}\,\eta^M,
\end{split}\end{align}
and
\begin{align}\label{aa4.51}  \begin{split}
   \tau_{A2} \,  \mu^N =& -  {\tau_{A4}}{} \,  \eta^N,
    \quad  \tau_{A2}\,  \gamma_{A}^N = -  {2\tau_{A5}}{}\,\eta^N,
    \quad {\tau_{A2}}\,  \theta^N
=  ({\tau_{A2}-\tau_{A4}})\, \eta^N .
\end{split}\end{align}
\end{lemma}
\begin{proof} Let $z\in \mM$ be such that $\psi_{A}^\mM(z)=0$.
In view of  \eqref{aa4.48} 
the equation $\psi_{A}^M(z)=0$ in
(\ref{aa4.47}) is equivalent to the first equation of
\eqref{aa4.49}. Similarly,  the equation $\psi_{A}^N(z)=0$ in
(\ref{aa4.47}) is equivalent to the first equation of
\eqref{aa4.51}.

By \eqref{aa4.26},  \eqref{aa4.26c}, \eqref{aa4.27} and the first equation of
\eqref{aa4.49}, we get  at $z$:
\begin{multline}\label{aa4.50}
    {\tau_{A2}}\, \gamma_{A}^{M} = 2\tau_{A1}{\tau_{A2}} \, \mu^{M}
+ 2\alpha_{A}(|\mu|^2){\tau_{A2}}\, \eta^M\\
=- 2\tau_{A1}{\tau_{A4}} \, \eta^{M}
+ 2\alpha_{A}(|\mu|^2){\tau_{A2}}\, \eta^M
= -  {2\tau_{A5}} \, \eta^M.
\end{multline}
The second equation in \eqref{aa4.51} follows similarly.
By \eqref{2.11} and the first equation in \eqref{aa4.51}, 
we get the third equation in \eqref{aa4.51}.
 The proof of Lemma \ref{ta5.6} is completed.
\end{proof}

For any $x\in M,\, y\in N$, $W\in T_{x}M$, $V\in T_{y}N$,
let $B(W)\in \End(\Lambda
(T^{*(0,1)}(M\times N)))_{(x,y)}$ be defined by
\begin{align}\label{aa4.85}
B(W) = \sqrt{-1}\,  c\left( J^M W\right)c(W) + |W|^2.
\end{align}
 Clearly,
the endomorphisms $B(W)$, $\sqrt{-1} c(  W)c(V)$  of $\Lambda
(T^{*(0,1)}(M\times N))_{(x,y)}$ are self-adjoint and $B(J^M W) =B(W)=B(-W)$.


\begin{lemma} \label{ta5.7}  On $\{z\in \mM: \psi_{A}^\mM(z)=0\}$,
    the following identities hold for $I_{A\cdot}$ in \eqref{a.3}:
\begin{align}\label{33.55}\begin{split}
\sqrt{-1} \Big(I_{A1}+ &I_{A2}\Big)= \frac{\tau_{A2}}{2}
\sum_{j=1}^{\dim G} B\left(V_j^M\right)+\tau_{A6}\,B\left(\eta^M\right)\\
& +\sqrt{-1}\tau_{A2}\, c\left(J^MV_j^M\right)c\left(V_j^N\right)
+2\sqrt{-1}\tau_{A6}\, c\left(J^M \eta^M\right)c\left(\eta^N\right)\\
& +\sqrt{-1}\tau_{A4}\, c\left(J^NV_j^N\right)c\left(V_j^M\right)
+2\sqrt{-1}\tau_{A7}\, c\left(J^N \eta^N\right)c\left(\eta^M\right),\\
I_{A3}=& \tau_{A4}\, c\left(J^NV_j^N\right)c\left(V_j^N\right)
+2 \tau_{A7}\, c\left(J^N \eta^N\right)c\left(\eta^N\right) .
\end{split}  \end{align}
\end{lemma}
\begin{proof} 
Let $z\in \mM$ be such that $\psi_{A}^\mM(z)=0$.
By \eqref{2.11} and \eqref{aa4.26}, we get
    \begin{align}\label{6.1}  \begin{split}
(d^{M}\tau_{A1})^{*}&= 2\alpha_{A}''(|\mu|^2) (|\theta|^2 -|\mu|^2)  
J^{M} \mu^{M} 
+2\alpha_{A}'(|\mu|^2)  J^{M}\eta^{M},\\
(d^{N}\tau_{A1})^{*}&= 2\alpha_{A}'(|\mu|^2)  J^{N} \theta^{N}.
\end{split}\end{align}

Using \eqref{2.11}, (\ref{aa5.11}) and  \eqref{aa4.49},
we infer at $z$,
\begin{align}\label{33.33} \begin{split}
(d^M \beta_{A})^* &= J^M \gamma_{A}^M
=-\frac{2\tau_{A5}}{\tau_{A2}}J^M \eta^M ,\\
(d^M \langle\rho_{A},\eta\rangle)^* 
&=J^M \eta^M +2\phi_{A}'(\beta_{A})
|\eta|^{2}\frac{\tau_{A5}}{\tau_{A2}}J^M \eta^M .
\end{split}\end{align}
By \eqref{2.11}, \eqref{aa4.26}, (\ref{aa4.49}), 
(\ref{6.1}) and (\ref{33.33}), at $z$, we get
\begin{multline}\label{6.3}
(d^{M}\tau_{A2})^{*} =-2 \phi_{A}'(\beta_{A}) \langle\rho_{A},\eta\rangle
(d^M \tau_{A1})^*\\
- 2 \phi_{A}''(\beta_{A}) \langle\rho_{A},\eta\rangle 
\tau_{A1}(d^M \beta_{A})^*
- 2 \phi_{A}'(\beta_{A}) \tau_{A1}
(d^M \langle\rho_{A},\eta\rangle)^*  \\
= \left\{ 4 \phi_{A}'(\beta_{A}) \langle\rho_{A},\eta\rangle
 \left[\alpha_{A}''(|\mu|^2) \Big(|\theta|^2 -|\mu|^2\Big) 
\frac{\tau_{A4}}{\tau_{A2}}  - \alpha_{A}'(|\mu|^2)  \right]\right.\\
\left.-4 \Big[- \phi_{A}''(\beta_{A})\langle\rho_{A},\eta\rangle 
+ (\phi_{A}'(\beta_{A}))^2 |\eta|^2 \Big]  
\frac{\tau_{A5}}{\tau_{A2}} \tau_{A1}
- 2 \phi_{A}'(\beta_{A}) \tau_{A1} \right\}J^{M}\eta^{M},
\end{multline}
and
\begin{multline}\label{6.4}
(d^{M}\tau_{A4})^{*} =-2 \phi_{A}'(\beta_{A}) \langle\rho_{A},\eta\rangle
\alpha_{A}'(|\mu|^2) 2 J^{M}\mu^{M}\\
+ \Big[- \phi_{A}'(\beta_{A})  
- 2 \phi_{A}''(\beta_{A}) \langle\rho_{A},\eta\rangle 
\alpha_{A}(|\mu|^2)\Big](d^M \beta_{A})^*\\
- 2 \phi_{A}'(\beta_{A}) \alpha_{A}(|\mu|^2)
(d^M \langle\rho_{A},\eta\rangle)^* \\
= \left\{  4 \phi_{A}'(\beta_{A})\alpha_{A}'(|\mu|^2)  
 \langle\rho_{A},\eta\rangle \frac{\tau_{A4}}{\tau_{A2}}\right.
 +  2 \phi_{A}'(\beta_{A})\frac{\tau_{A5}}{\tau_{A2}}\\
 -4 \Big[- \phi_{A}''(\beta_{A})\langle\rho_{A},\eta\rangle 
+ (\phi_{A}'(\beta_{A}))^2
|\eta|^2 \Big] \alpha_{A}(|\mu|^2)  
\frac{\tau_{A5}}{\tau_{A2}}\\
- 2 \phi_{A}'(\beta_{A}) \alpha_{A}(|\mu|^2)
\Big\}J^{M}\eta^{M}.
\end{multline}

From \eqref{2.11} and (\ref{aa4.48}), we get
\begin{align}\label{6.5} \begin{split}
 &   \left(d^M \psi_{Aj}\right)^* = 
   (d^{M}\tau_{A2})^{*}\, \mu_{j}+  (d^{M}\tau_{A4})^{*} \, \eta_{j}
    +  \tau_{A2} J^{M} V_{j}^{M},\\
&    \left(d^N \psi_{Aj}\right)^* = 
   (d^{N}\tau_{A2})^{*} \, \mu_{j}+  (d^{N}\tau_{A4})^{*} \, \eta_{j}
    +  \tau_{A4} J^{N} V_{j}^{N}.
 \end{split}\end{align}
From (\ref{aa4.26c}), 
(\ref{aa4.28}), the first equation of (\ref{aa4.49}) 
and (\ref{6.3})-(\ref{6.5}), we get at $z$,
\begin{align}\label{aa4.42} \begin{split}
&c\left((d^M \psi_{Aj})^* \right) c\left(V_j^M\right) 
 =\tau_{A2}\,  c\left(J^M V_{j}^M\right)c\left(V_{j}^M\right)
+  2 \tau_{A6}\, c\left(J^M \eta^M\right)c\left(\eta^M\right),\\
&\left\langle \left(1+ \frac{J^M}{\sqrt{-1}} \right) V_j^M,
(d^M \psi_{Aj})^*  \right\rangle = \frac{1}{\sqrt{-1}}\Big(\tau_{A2}
\sum_{j}\left|V_j^M\right|^2 +2 \tau_{A6}\left|\eta^M\right|^2\Big).
\end{split}\end{align}
Using  (\ref{aa4.28}), the first equation of \eqref{aa4.51} 
and (\ref{6.3})-(\ref{6.5}), we get at $z$,
\begin{align}\label{aa4.40}
    \begin{split}
c((d^M\psi_{Aj})^*) c(V_j^N) =&  \tau_{A2}\, c(J^M V_{j}^M)c(V_{j}^N)
+ 2 \,  \tau_{A6} \, c(J^M \eta^M) c(\eta^N).
\end{split}\end{align}

By  (\ref{aa5.11}), (\ref{aa4.51}) and (\ref{6.1}),
it follows that at $z$,
\begin{subequations}
\begin{align}\label{33.37}
& (d^N\beta_{A})^*=2\alpha_{A}
  \left(|\mu|^2\right) \frac{\tau_{A2}-\tau_{A4}}{\tau_{A2}}J^N \eta^N,\\
  \label{33.38a}
&  (d^{N}\tau_{A1})^{*}
= 2\alpha_{A}'(|\mu|^2) \frac{\tau_{A2}-\tau_{A4}}{\tau_{A2}}J^N 
\eta^N,\\
\label{33.38}& (d^N\rho_{Aj})^*= \Big(1-\phi_{A}(\beta_{A})\Big)J^NV^N_j
 -2\phi_{A}'(\beta_{A})\alpha_{A}\left(|\mu|^2\right)\eta_j
\frac{\tau_{A2}-\tau_{A4}}{\tau_{A2}}J^N \eta^N.
\end{align}
\end{subequations}
From \eqref{2.11}, \eqref{33.5}, (\ref{aa4.51}) and (\ref{33.38}), we have
\begin{multline}\label{6.8}
    (d^N \langle\rho_{A},\eta\rangle)^* 
    =\Big(\theta_{j} -\phi_{A}(\beta_{A})\eta_{j} \Big) 
    (d^{N}\eta_{j})^{*} 
    + \eta_{j}(d^N\rho_{Aj})^*\\
    =\left[1-2\phi_{A}(\beta_{A})
+ \Big(1-2\phi_{A}'(\beta_{A})\alpha_{A}\left(|\mu|^2\right) |\eta|^{2}\Big)  
\frac{\tau_{A2}-\tau_{A4}}{\tau_{A2}} \right] J^N \eta^N.
\end{multline}
By \eqref{aa4.26}, (\ref{aa4.49}), (\ref{6.1})
and (\ref{33.37})-(\ref{6.8}), we get at $z$,
\begin{multline}\label{6.9}
(d^{N}\tau_{A2})^{*} =-2 \phi_{A}'(\beta_{A}) \langle\rho_{A},\eta\rangle
(d^N \tau_{A1})^*\\
- 2 \phi_{A}''(\beta_{A}) \langle\rho_{A},\eta\rangle 
\tau_{A1}(d^N \beta_{A})^*
- 2 \phi_{A}'(\beta_{A}) \tau_{A1}
(d^N \langle\rho_{A},\eta\rangle)^*  \\
= \left\{- 4\phi_{A}'(\beta_{A}) \langle\rho_{A},\eta\rangle
\alpha_{A}'(|\mu|^2) \frac{\tau_{A2}-\tau_{A4}}{\tau_{A2}}\right.\\ 
+4 \Big[- \phi_{A}''(\beta_{A})\langle\rho_{A},\eta\rangle 
+ (\phi_{A}'(\beta_{A}))^2
|\eta|^2 \Big] \alpha_{A}\left(|\mu|^2\right) 
\frac{\tau_{A2}-\tau_{A4}}{\tau_{A2}}  \tau_{A1}\\
\left. - 2 \phi_{A}'(\beta_{A}) \tau_{A1}
\left(1-2\phi_{A}(\beta_{A}) + \frac{\tau_{A2}-\tau_{A4}}{\tau_{A2}} 
\right)
\right\}J^{M}\eta^{N},
\end{multline}
and
\begin{multline}\label{6.10}
(d^{N}\tau_{A4})^{*} =\Big[- \phi_{A}'(\beta_{A})  
- 2 \phi_{A}''(\beta_{A}) \langle\rho_{A},\eta\rangle 
\alpha_{A}(|\mu|^2)\Big](d^N\beta_{A})^*\\
- 2 \phi_{A}'(\beta_{A}) \alpha_{A}(|\mu|^2)
(d^N \langle\rho_{A},\eta\rangle)^* \\
= \left\{  4 \Big[-  \phi_{A}''(\beta_{A}) \langle\rho_{A},\eta\rangle 
+ (\phi_{A}'(\beta_{A}))^{2} |\eta|^{2}
\Big]\alpha_{A}(|\mu|^2)^{2}\, 
\frac{\tau_{A2}-\tau_{A4}}{\tau_{A2}}\right.\\
\left. - 2 \phi_{A}'(\beta_{A})  \alpha_{A}\left(|\mu|^2\right) 
\Big[1-2\phi_{A}(\beta_{A}) + 2 \frac{\tau_{A2}-\tau_{A4}}{\tau_{A2}}
\Big]
\right\}J^{M}\eta^{N}.
\end{multline}       
From (\ref{aa4.26c}), (\ref{aa4.28}), (\ref{aa4.51}), (\ref{6.5}), 
(\ref{6.9}) and (\ref{6.10}), we get at $z$,
\begin{align}\label{aa4.40c}
    \begin{split}
  c((d^N\psi_{Aj})^*) c(V_j^M)
       = & \tau_{A4}\, c(J^N V_{j}^N)c(V_{j}^M)
       + 2\, \tau_{A7} \,  c(J^N \eta^N)c(\eta^M) ,\\
 c((d^N\psi_{Aj})^*) c(V_j^N)
       = &  \tau_{A4}\, c(J^N V_{j}^N)c(V_{j}^N)
       + 2\, \tau_{A7} \,  c(J^N \eta^N)c(\eta^N) .
\end{split}\end{align}   
    
 By (\ref{a.3}), (\ref{aa4.85}), (\ref{aa4.42}), 
(\ref{aa4.40}) and  (\ref{aa4.40c}), we get (\ref{33.55}).
\end{proof}




\begin{lemma}\label{ta5.12} For any $k>0$, the following
inequalities hold for $W\in TM$, $V\in TN$ :
\begin{align}\label{aa4.86}
\begin{split}
&B(W)\geqslant 0, \\
&\sqrt{-1} c(  W)c(V) \geqslant -\frac{1}{2 k} B(W) - k |V|^2 .
\end{split}
\end{align}
\end{lemma}
 \begin{proof} It is enough to prove it for 
 $V=v+\ov{v}$, $W=w+\ov{w}$, and $\{v,w\}$ an orthonormal basis
 of $\C^2$ with the standard Hermitian product. 
Using \eqref{a1.20} and \eqref{aa4.85}, we find
\begin{align}\label{aa4.87}
\begin{split}
B(W)&= -2 \left({w}^*\wedge + i_{\ov{w}}\right) \left({w}^*\wedge
- i_{\ov{w}}\right) + 2 = 4 {w}^*\wedge i_{\ov{w}}.
  \end{split}
\end{align}
Thus   the first inequality in \eqref{aa4.86} holds (cf.
\cite[(2.9), (2.13)]{TZ98}).

 For any $\sigma\in \Lambda \ov{\C^2}^*$, we write
$\sigma= \sigma_1\, {w}^*\wedge {v}^* + \sigma_2\, {w}^*+\sigma_3\, {v}^* 
+  \sigma_4,$
where $\sigma_{i}\in \C$ for $i=1,2,3,4$.
By \eqref{a1.20}, we get 
\begin{align}\label{aa4.90}
    \begin{split}
   \sqrt{-1} c(W) c(V)\sigma =2 \sqrt{-1} \left(  - \sigma_1
+ \sigma_2  {v}^* 
   -\sigma_3 {w}^*
+ \sigma_4 {w}^*\wedge {v}^*\right) .
\end{split}\end{align}
From \eqref{aa4.87} and  \eqref{aa4.90},  
we find that for any $k>0$,
\begin{align}\label{aa4.92}\begin{split}
\langle &\sqrt{-1} c(W)c(V)\sigma, \sigma \rangle = 4 \,\text{Im} 
\left(\sigma_{1} \ov{\sigma_{4}}- \sigma_{2} \ov{\sigma_{3}}\right) \\
&\geqslant  -\frac{2}{k}\left(|\sigma_1 |^2 + |\sigma_2 |^2\right)
- 2k |\sigma |^2 
= -\frac{1}{2k} \langle B(W) \sigma, \sigma\rangle - 2k |\sigma |^2  .
 \end{split}\end{align}

From 
\eqref{aa4.92}, we get  the second inequality
of \eqref{aa4.86}.
\end{proof}

\subsection{Proof of Lemma \ref{t33.6} (III): final step}\label{sa.6}
 Recall that $z\in\mM$ satisfies $\psi_{A}^\mM(z)=0$.
By Lemma \ref{t33.4}, $z\in \mM\setminus \partial \mM$.

By Lemma \ref{ta5.10}, $\tau_{A2}, \tau_{A4}>0$ 
on $\mM$ for $A$ large enough. Thus 
by (\ref{aa4.73}), (\ref{33.55}) and 
the second equation in \eqref{aa4.86} with $k=8$, we get
\begin{align}\label{aa4.95}\begin{split}
    \sqrt{-1}\left(I_{A1}+ I_{A2}\right) \geqslant &
    \frac{1}{2} \left(\frac{7}{8}\tau_{A2}   
    -  \frac{1}{8}   \tau_{A4}\right) \sum_{j=1}^{\dim G} B\left(V_{j}^M\right)
  -  \Big(8 \tau_{A2} + 8 \tau_{A4}\Big)
\sum_{j=1}^{\dim G} \left|V_{j}^N\right|^2\\
&   + \left(\frac{7}{8}\tau_{A6} 
-  \frac{1}{8}   \tau_{A7}  \right) B\left(\eta^M\right)
   - \Big(16 \tau_{A6}+ 16 \tau_{A7}\Big)  \left|\eta^N\right|^2.
\end{split}\end{align}
By Lemma \ref{ta5.10},  we obtain for $A>0$ large enough,
\begin{align}\label{aa4.96}\begin{split}
  & \frac{7}{8} \tau_{A2}   
  -  \frac{1}{8} \tau_{A4}
=  \frac{3}{4} + \frac{1}{8} \phi_{A}(\beta_{A}) +
\mO_0\left(A^{-1/2}\right)\geqslant \frac{1}{ 2}.
\end{split}\end{align}
By Lemma \ref{ta5.11}, for $A>0$ large enough, 
as  $z\in \mM\setminus \partial \mM$,
\begin{align}\label{aa4.100}
    \frac{7}{8}  \tau_{A6} 
    -  \frac{1}{8} \tau_{A7} 
     \geqslant \frac{1}{8} \tau_{A6} \Big(1+ \mO_0(A^{-1/2})\Big)>0.
\end{align}

Recall that $V_{j}^N$, $\eta$ are defined on the compact manifold $N$.
By Lemmas \ref{ta5.10}, 
 (\ref{aa4.72}), (\ref{aa4.86}) and (\ref{aa4.95})-(\ref{aa4.100}),  
there exists $C'>0$ such that for $A>0$ large enough, 
the following inequality holds:
\begin{align}\label{33.56}
 \sqrt{-1}\left(I_{A1}+ I_{A2}\right)\geqslant -C'\, \Id
\quad \text{ on } \{z\in \mM, \psi_{A}^M(z)=0\} .
 \end{align}
 By (\ref{aa4.64}), (\ref{aa4.72}) and (\ref{33.55}),  
 there exists $C''>0$ such that for $A>0$ large enough, we have
\begin{align}\label{33.59}
 \left|I_{A3}\right|\leqslant C'' \quad \text{ on }
 \{z\in \mM, \psi_{A}^M(z)=0\} .
 \end{align}
By Lemma \ref{ta5.10}, 
\eqref{33.0}, \eqref{aa4.48} and
(\ref{a.103}), for $A>0$ large enough, 
we get over $\mM$:
\begin{align}\label{aa4.82}
\begin{split}
 2  \left\langle  \psi_{A}, \theta\right\rangle &
= 2\tau_{A2}|\mu| ^2 + 2\tau_{A4} |\eta|^2
+ 2\left(\tau_{A2}+\tau_{A4}\right) \langle \mu,\eta\rangle\\
&\geqslant 2A + \mO_0\left(A^{1/2}\right) \geqslant A,\\
|\psi_{A}|&=\mO_0\left(A^{1/2}\right).
\end{split}\end{align}

By (\ref{33.56})--(\ref{aa4.82}),
  we get (\ref{33.14}).  This completes  the  proof of Lemma \ref{t33.6}.

\section{Functoriality of  quantization}\label{s4}
This section is organized as follows. In Section \ref{s4.1},  
we establish the product formula for  quantization, Theorem \ref{t0.4}.
 In Section \ref{s4.2}, we explain the compatibility 
of quantization and its restriction to a subgroup. 

We will use the assumptions and notation  
in the Introduction and in Section \ref{s33.1}.

\subsection{Proof of Theorem \ref{t0.4}}\label{s4.1}
Let  $c>0$ be a regular value of  $|\theta|^2$.
By \cite[Theorem 4.3]{TZ99}, \cite[Prop. 7.10]{Par01}
(cf. also Theorems \ref{t1.3}, \ref{t1.5}), the
following identity holds:
\begin{align}\label{2.8}
    \Ind\left(\sigma_{L\otimes F,\theta}^{(M\times N)_c}\right)_{\gamma=0}
    = Q\left(\left(L {\otimes} F\right)_{\gamma=0}\right).
\end{align}
Here $0$ need not to be a regular value of $\theta$.

On the other hand, by Theorem \ref{t0.1}b), we have
\begin{align}\label{44.1}
   \Ind\left(\sigma_{L\otimes F,\theta}^{(M\times N)_c}\right)_{\gamma=0}
 = Q\left(L {\otimes} F\right)_{\gamma=0}.
\end{align}
Therefore, by (\ref{2.8}) and (\ref{44.1}), we get  (\ref{0.13}).
Thus, to prove Theorem \ref{t0.4}, 
we only need to prove the
following identity, which has been  stated in (\ref{0.14}),
\begin{align}\label{44.3}
 Q\left(L {\otimes} F\right)_{\gamma=0}
=\sum_{\gamma\in\Lambda_+^*}Q(L)_\gamma\cdot Q(F)_{\gamma,*}.
\end{align}


We first establish 
the following lemma, which has been
stated in (\ref{0.101}).

\begin{lemma}\label{t44.1} There exists $a'\geqslant 0$ such that for any
    regular value $a> a'$ of $|\mu|^2:M\rightarrow {\R}$,
    the following identity holds:
\begin{align}\label{44.4}
  \sum_{\gamma\in\Lambda_+^*}Q(L)_\gamma\cdot Q(F)_{\gamma,*}
  =\Ind\left(\sigma^{M_a\times N}_{L\otimes F,\mu}\right)_{\gamma=0}.
\end{align}
\end{lemma}
\begin{proof} 
We denote the finite set 
$\{\gamma\in\Lambda_+^*: Q(F)_{\gamma,*}\neq 0\}$ by $\Lambda_+^*(F)$. 
By Theorem \ref{t0.1}, there exists $a_1\geqslant 0$ such that for
any regular value $a> a_1$ of $|\mu|^2$, we have
\begin{align}\label{44.5}
Q(L)_\gamma =\Ind\left(\sigma^{M_a }_{L ,\mu}\right)_\gamma 
\quad \mbox{  for any } \gamma\in \Lambda_+^*(F).
\end{align}

Let $a> a_1$ be a regular value of $|\mu|^2$.
For $0\leqslant t\leqslant 1$, let $\sigma_t$ be the symbol on $M_a\times
N$ defined to be a deformation of $\sigma^{M_a\times N}_{L\otimes
F,\mu}$ as follows,
\begin{align}\label{44.6}
\sigma_t=\sigma^{M_a\times N}_{L\otimes F,\mu}-(1-t)\sqrt{-1}
\pi^*c\left(\mu^N\right)
\end{align}
where $\pi:T(M_a\times N)\rightarrow M_a\times N$ is the canonical
projection (cf. (\ref{1.2})).

By (\ref{1.2}) and (\ref{2.25}),   when
$t=0$, $\sigma_0$ is the external  product of $\sigma^{M_a }_{L
,\mu}$ and $\sigma^N_{F,0}$ in the sense of \cite{A74} (cf.
\cite[(3.11)]{Par01}). Then by the multiplicativity
 of the transversal
index (\cite[Theorem 3.5]{A74}, \cite[(3.12)]{Par01}) and by
(\ref{44.5}), we get
\begin{align}\label{44.7}
 \sum_{\gamma\in\Lambda_+^*}Q(L)_\gamma\cdot
   Q(F)_{\gamma,*}=\Ind\left(\sigma_0\right)_{\gamma=0}.
\end{align}

For $0\leqslant t\leqslant 1$, set
\begin{align}\label{44.8}
 V_t=\mu^{M_a\times N}- (1-t)\mu^N.
\end{align}
Then by (\ref{2.11}), (\ref{2.25}) and (\ref{44.8}), we have
\begin{align}\label{44.9}
 V_t= \mu^M +t\mu^N.
\end{align}
As $a> a_1$ is a regular value of $|\mu|^2$,
 $\mu^M$ does not vanish on $\partial M_a$. 
From \eqref{44.9}, 
$ \mu^{M_a\times N}$, $V_t$ do not vanish on $\partial(M_a\times
N)=(\partial M_a)\times N$ for $0\leqslant t\leqslant 1$.

By (\ref{1.2}), (\ref{44.6}), (\ref{44.8}) and (\ref{44.9}), 
the set
$\{(z,v)\in T_G(M_a\times N): \mbox{there exists } 0\leqslant  t\leqslant 1
\mbox{ such that }  \ \sigma_t(z,v)\ \mbox{is non-invertible}\}
\subset \{(x,y,0)\in  T_G(M_a\times N):  \mu^M(x)=0, x\in M_a, y\in N\}$
 is a compact
subset of $T_G (\widehat{M_a\times N})$. Thus $\sigma_t$ forms a
continuous family of transversally elliptic  symbols in the sense of
\cite{A74} and \cite[\S3]{Par01}. Hence by (\ref{44.6}),
(\ref{44.7}) and the homotopy invariance of the transversal index
(cf. \cite[Theorems 2.6, 3.7]{A74}, \cite[\S 3]{Par01}),
we get (\ref{44.4}).
  The proof of Lemma \ref{t44.1} is completed.
\end{proof}

Let $A>0$ be a regular value of both $|\mu|^2$ and
$\frac{1}{2}|\theta|^2$.
 We may and we will assume that $A>0$ is large enough so that both
 Theorem \ref{t33.2} and Lemma \ref{t44.1} hold.

 Let $Y:\mM\rightarrow \kg$ be a $G$-equivariant map such that
 (\ref{33.3}) holds.
 By the additivity of the transversal index (cf. \cite[Theorem 3.7, \S 6]{A74} 
 and \cite[Prop. 4.1]{Par01}), we  have
\begin{align}\label{44.12}
\Ind\left(\sigma^{(M\times N)_{2A}}_{L\otimes
F,\theta}\right)_{\gamma=0}= \Ind\left(\sigma^{\mM}_{L\otimes
F,Y}\right)_{\gamma=0}+ \Ind\left(\sigma^{M_A\times
N}_{L\otimes F,\mu}\right)_{\gamma=0}.
\end{align}
By Theorems \ref{t1.3} and \ref{t33.2}, we find
\begin{align}\label{44.13}
\Ind\left(\sigma^{\mM}_{L\otimes F,Y}\right)_{\gamma=0}=0.
\end{align}
By Theorem \ref{t0.1}b), (\ref{44.4}), (\ref{44.12}) and 
(\ref{44.13}), we get (\ref{44.3}).
The proof of Theorem \ref{t0.4}
is completed.

\subsection{Restriction commutes with quantization}\label{s4.2}
Set
\begin{align}\label{0.6}
Q_G(L)^{-\infty} =\bigoplus_{\gamma\in \Lambda^*_+ }Q(L)_\gamma\,
\cdot V_\gamma^G \in R[G].
\end{align}
By Theorem \ref{t0.2}, $Q_G(L)^{-\infty}$ is equal to the formal
geometric quantization in the sense of  Weitsman \cite[Definition
4.1]{W01} (where the fundamental properness assumption of the
moment map was introduced into the framework of geometric
quantization) and Paradan \cite[Definition 1.2]{Par07}.

On the other hand, let $H$ be a compact connected subgroup of $G$
such that the moment map of the induced   action of $H$ on $M$ is
also proper. By combining Theorem \ref{t0.2}, (\ref{0.6}) with
\cite[Theorem 1.3]{Par07}, one   gets the following  relation
between $Q_G(L)^{-\infty}$ and $Q_H(L)^{-\infty}$.

\begin{thm}\label{t0.100} Any irreducible representation of $H$ has
a finite multiplicity in $Q_G(L)^{-\infty}$.
Moreover, when both sides are viewed as virtual representation spaces
of $H$, the following identity holds:
\begin{align}\label{0.89}
\left.Q_G(L)^{-\infty}\right|_H = Q_H(L)^{-\infty}.
\end{align}
\end{thm}
It would be interesting to give a direct proof of Theorem \ref{t0.100}.

\def\cprime{$'$} \def\cprime{$'$}
\providecommand{\bysame}{\leavevmode\hbox to3em{\hrulefill}\thinspace}
\providecommand{\MR}{\relax\ifhmode\unskip\space\fi MR }
\providecommand{\MRhref}[2]{%
  \href{http://www.ams.org/mathscinet-getitem?=#1}{#2}
}
\providecommand{\href}[2]{#2}

\end{document}